\documentclass[11pt]{amsart}
\usepackage{amssymb}
\usepackage{amsmath}
\usepackage{amsthm}
\usepackage{color}
\usepackage{hyperref,url}
\usepackage{enumitem}
\usepackage{slashed}
\usepackage[normalem]{ulem}

\usepackage{physics}
\usepackage{amsmath}
\usepackage{tikz}
\usepackage{mathdots}
\usepackage{yhmath}
\usepackage{cancel}
\usepackage{color}
\usepackage{siunitx}
\usepackage{array}
\usepackage{multirow}
\usepackage{amssymb}
\usepackage{gensymb}
\usepackage{tabularx}
\usepackage{booktabs}
\usetikzlibrary{fadings}
\usetikzlibrary{patterns}
\usetikzlibrary{shadows.blur}
\usetikzlibrary{shapes}

\theoremstyle{plain}
\newtheorem{theorem}{Theorem}[section]
\newtheorem{proposition}[theorem]{Proposition}
\newtheorem{conjecture}[theorem]{Conjecture}
\newtheorem{corollary}[theorem]{Corollary}
\newtheorem{lemma}[theorem]{Lemma}

\theoremstyle{definition}
\newtheorem{definition}[theorem]{Definition}

\newcommand{\N}{\mathbb{N}}
\newcommand{\Z}{\mathbb{Z}}

\newcommand{\R}{\mathbb{R}}
\newcommand{\C}{\mathbb{C}}

\newcommand{\e}{\varepsilon}
\newcommand{\supp}{\mathrm{supp}}

\newcommand{\s}{\psi}

\renewcommand{\a}{\alpha}

\newcommand{\mc}{\mathcal}
\newcommand{\mb}{\mathbb}

\def\Xint#1{\mathchoice
{\XXint\displaystyle\textstyle{#1}}%
{\XXint\textstyle\scriptstyle{#1}}%
{\XXint\scriptstyle\scriptscriptstyle{#1}}%
{\XXint\scriptscriptstyle\scriptscriptstyle{#1}}%
\!\int}
\def\XXint#1#2#3{{\setbox0=\hbox{$#1{#2#3}{\int}$ }
\vcenter{\hbox{$#2#3$ }}\kern-.6\wd0}}

\def\dashint{\Xint-}

\usepackage[utf8]{inputenc}
\usepackage{calc}
\usepackage{accents}

\usepackage{bbm}
\usepackage{scalerel,stackengine}
\stackMath
\newcommand\widecheck[1]{%
\savestack{\tmpbox}{\stretchto{%
  \scaleto{%
    \scalerel*[\widthof{\ensuremath{#1}}]{\kern-.6pt\bigwedge\kern-.6pt}%
    {\rule[-\textheight/2]{1ex}{\textheight}}
  }{\textheight}%
}{0.5ex}}%
\stackon[1pt]{#1}{\scalebox{-1}{\tmpbox}}%
}

\newcommand{\stkout}[1]{\ifmmode\text{\sout{\ensuremath{#1}}}\else\sout{#1}\fi}

\title{Decoupling inequalities for short generalized Dirichlet sequences}
\date{}

\author{Yuqiu Fu}
\address{Department of Mathematics, MIT,
Cambridge, MA 02139}
\email{yuqiufu@mit.edu}

\author{Larry Guth}
\address{Department of Mathematics, MIT,
Cambridge, MA 02139}
\email{lguth@math.mit.edu}

\author{Dominique Maldague}
\address{Department of Mathematics, MIT,
Cambridge, MA 02139}
\email{dmal@mit.edu}

\begin{document}

\begin{abstract}
  We study decoupling theory for functions on $\R$ with Fourier transform supported in a neighborhood of short Dirichlet sequences $\{\log n\}_{n=N+1}^{N+N^{1/2}}$, as well as sequences with similar convexity properties. We utilize the wave packet structure of functions with frequency support near an arithmetic progression. 
\end{abstract}

\maketitle

\section{Introduction}

In this paper, we study decoupling theory for functions $f:\R\to\C$ with Fourier support near certain convex sequences. As a model case of decoupling, consider the truncated parabola $\mb{P}^1=\{(t,t^2):|t|\le 1\}$. Let $R\ge 1$ be a large parameter and write $\mc{N}_{R^{-1}}(\mb{P}^1)$ as a disjoint union of caps $\theta=\mc{N}_{R^{-1}}(\mb{P}^1)\cap (I\times \R)$, where $I$ is an $R^{-1/2}$-interval. 
The decoupling inequality of Bourgain and Demeter \cite{bourgain2015proof} says that if $2\le p\le 6$, then for any $\e>0$, there exists $C_\e$ such that
\[ \|\sum_\theta f_\theta\|_{L^p(\R^2)}\le C_\e R^\e(\sum_\theta\|f_\theta\|_{L^p(\R^2)}^2)^{1/2}     \]
whenever $f_\theta:\R^2\to\C$ are Schwartz functions satisfying $\supp\widehat{f}_\theta\subset\theta$.

This paper explores analogues between decoupling for $\mathbb{P}^1$ and short Dirichlet sequences $\{\log n\}_{n=N+1}^{N+N^{1/2}}$, as well as sequences with similar convexity properties described in the following definition. 

\begin{definition} \label{defnDirichlet_intro}
Let $N \geq 2.$  We call $\{a_n\}_{n=1}^{N}$ a generalized Dirichlet sequence (with parameter $N$) if it satisfies the property
\begin{equation}\label{defa_n_intro}
  a_{2}-a_1 \in [\frac{1}{4N},\frac{4}{N}], \qquad (a_{i+2}-a_{i+1})-(a_{i+1}-a_i) \in [\frac{1}{4N^2},\frac{4}{N^2}].
\end{equation}
We will call $\{a_n\}_{n=1}^{N^{1/2}}$ satisfying \eqref{defa_n_intro} an $N^{1/2}$-
short generalized Dirichlet sequence.
\end{definition}
For simplicity, we say short (generalized) Dirichlet sequence to mean $N^{1/2}$-short (generalized) Dirichlet sequence, unless otherwise specified. 
Note that the reflected short Dirichlet sequence, $\{-\log (N+N^{-1/2}-n+1)\}_{n=1}^{N^{1/2}},$ satisfies \eqref{defa_n_intro}. 

Now we describe our decoupling set-up. From now on $C, c>0$ will denote absolute  constants that may vary from line to line. For the convenience of reading we may regard $C,c$ as $1$. 
For $1\le L\le c N^{1/2}$ and each $j=1,\ldots,\frac{N^{1/2}}{L}$, define
\[ I_j=\bigcup_{i=(j-1)L+1}^{jL} B_{L^2/N^2}(a_i)
, \]
where $B_{L^2/N^2}(a_i)$ means the $L^2/N^2$ interval centered at $a_i$. Let $\Omega$ be the $L^2/N^2$-neighborhood of $\{a_n\}_{n=1}^{N^{1/2}}.$ We consider the partition
\begin{equation}\label{decomL_intro}
\Omega=
\bigsqcup_{j}I_j. 
\end{equation} 
We choose the $L^2/N^2$-neighborhood of $\{a_n\}_{n=1}^{N^{1/2}}$ because every $I_j$ is essentially an $L^2/N^2-$neighborhood of an arithmetic progression, which we call a fat AP. To see this we calculate for $1\leq n \leq N^{1/2}-L,$
\[a_{n+L}-a_{n}-L(a_{n+1}-a_n) = \sum_{m=1}^L (a_{n+m}-a_{n+m-1}-(a_{n+1}-a_n)) \sim \sum_{m=1}^L (m-1)/ N^{2} \sim L^2 / N^{2}. \]
So indeed $I_j$ lies in a $CL^2/N^2$-neighborhood of an $L$-term AP with common difference $a_{(j-1)L+1}-a_{(j-1)L}$ and starting point $a_{(j-1)L}.$
Also, note that the common differences for distinct $I_j$ are $cL/N^2$-separated.



%

We denote the partition $\{I_j\}_{j=1}^{N^{1/2}/L}$ by $\mathcal{I}.$
The first main result of this paper is the following decoupling theorem for $\Omega=\bigsqcup_{I\in \mathcal{I}} I$.
\begin{theorem}\label{decouplingthm_intro}
  Let $\Omega$ and $\mathcal{I}$ be defined as in the last paragraphs. Then
  for $2\leq p\leq 6$ and every $\e>0$ we have
  \begin{equation}\label{decoupling_intro}
    \|\sum_{I\in \mathcal{I}} f_{I}\|_{L^p(\R)} \lesssim_\e N^\e \left( \sum_{I\in \mathcal{I}} \|f_{I}\|_{L^p(\R)}^2 \right)^{1/2}
  \end{equation}
  for functions $f_{I}$ with $\supp \widehat{f_{I}}\subset I.$
\end{theorem}

The range of $p$ is sharp in the sense that \eqref{decoupling_intro} cannot hold for $p>6,$ which can be seen by taking $\widehat{f_I}$ to be a smooth bump with height $1$ adapted to $I$ for every $I.$ Indeed for this choice of $f_I,$ we have $|\sum_I f_I| \sim \frac{L^2}{N^2} N^{1/2} $ on $B_{cN^{1/2}}(0)$, and $\|f_I\|_{L^p(\R)}\sim \|\widehat{f_I}\|_{L^{p'}(\R)}\sim (L\frac{L^2}{N^2})^{1/p'}$ where $1/p+1/p'=1.$ So
$$\|\sum_{I\in \mc{I}}f_I\|_{L^p(\R)}\gtrsim \frac{L^2}{N^2} N^{1/2} (N^{1/2})^{1/p},\quad \left( \sum_{I\in \mathcal{I}} \|f_{I}\|_{L^p(\R)}^2 \right)^{1/2}\sim \left(\frac{N^{1/2}}{L}\right)^{1/2}\left(\frac{L^3}{N^2}\right)^{1-1/p}.$$
Then \eqref{decoupling_intro} would imply
$$\left({N^{1/2}}/{L} \right)^{1/2-3/p} \lesssim_\e N^\e,$$
and hence $p\leq 6.$ We shall compare Theorem \ref{decouplingthm_intro} with the $\ell^2L^p$ decoupling inequality of the parabola in \cite{bourgain2015proof}, which has the same critical exponent $6.$ Indeed we will see many similarities between short generalized Dirichlet sequences and $\mb{P}^1$ from a Fourier analytic point of view.

The notion of strict convexity of a sequence $\{a_n\}$ in $\R$ will parallel the role of curvature of the parabola in decoupling. Some key geometric aspects in the proof of decoupling for $\mathbb{P}^1$ are: (1) identifying caps $\theta$ as approximate $R^{-1/2}\times R^{-1}$ rectangles, which give rise to dual tubes $\theta^*$ of dimension $R^{1/2}\times R$, and (2) noting that $\theta$ are separated in angle and so are $\theta^*$. The $|f_\theta|$ are roughly constant on translates of $\theta^*$. 

In the $\{a_n\}_{n=1}^{N^{1/2}}$ setting, corresponding to $f_\theta$ we have $f_{I_j}$ which are functions $f_{I_j}:\R\to\C$ satisfying $\supp\widehat{f}_{I_j}\subset I_j$. We may (1) identify the $\frac{L^2}{N^2}$-neighborhood of $I$ as approximately an $\frac{L^2}{N^2}$-neighborhood of an arithmetic progressions (called a fat AP), giving rise to dual $I^*$ defined in Definition \ref{dualI}, which are also fat APs, and (2) note that distinct $I$ are separated in step-size of the corresponding arithmetic progressions (and the same for $I^*$). The $|f_{I}|$ are also roughly constant on translates of $I^*$  \cite{bourgain1991remarks,bourgain1993distribution}. 


Bourgain made use of this locally constant property to connect a conjecture of Montgomery with the Kakeya conjecture \cite{bourgain1991remarks,bourgain1993distribution}.   
To prove a decoupling inequality we need to identify another geometric analogy, the ``ball'', which is roughly the smallest set restricting to which in the physical space essentially preserves the frequency support.


 For the $R^{-1}$-neighborhood of the parabola, the ``ball'' is a ball $B_R$ of radius of $R.$ We will define the the ``ball'' $P(L)$ in the short generalized Dirichlet sequence setting  in Section \ref{comp}. $P(L)$ will be a fat AP which sometimes degenerates to  a Euclidean ball. With these notions of caps, tubes, and balls in the short generalized Dirichlet sequence setting, we are able to exploit the wave packet structure of a function with frequency support on $I\in \mc{I},$ and prove a bilinear Kakeya-type estimate (Proposition \ref{BKprop}) and a bilinear restriction-type estimate (Proposition \ref{bilresprop}) that look almost identical to those in the parabola setting. The choice of $N^{1/2}$ plays an important role in making this resemblance possible, which we will discuss at the end of Section \ref{longsec}. 

The proof of Theorem \ref{decouplingthm_intro} is based on the high-low decomposition method in \cite{guth2020improved}. We do not intend to get a logarithmic decoupling constant as in \cite{guth2020improved}, but we want to prove a refined decoupling inequality as in \cite{guth2020falconer} which creates some technical differences.

The partition $\Omega=\bigsqcup_{I\in\mc{I}}I$ is maximal in the sense that if $\Omega=\bigsqcup I'$ where $I'$ is the union of more than $CL$ many adjacent intervals, then $I'$ is no longer essentially a fat AP. 
Because of this, we will call $\Omega=\bigsqcup_{I\in \mc{I}}I$ the canonical partition and refer to Theorem \ref{decouplingthm_intro} as decoupling for the canonical partition, or simply decoupling.
In the spirit of small cap decoupling as in \cite{demeter2020small}, we may also consider the ``small cap'' decoupling for short generalized Dirichlet sequences. 
Now we let $L_1\in [1,L]$ be an integer, and we partition $\Omega$ into $L_1$ consecutive intervals $J_j$:
\begin{equation}\label{decomL_1_intro}
    \Omega=\bigsqcup_{j=1}^{N^{1/2}/L_1} J_j=\bigsqcup_{j=1}^{N^{1/2}/L_1} \left(\bigcup_{i=(j-1)L_1+1}^{jL_1} B_{L^2/N^2}(a_i)\right).
\end{equation}
We let $\mc{J}$ denote the partition $\{J_j\}_{j=1}^{N^{1/2}/L_1}.$
The next decoupling result in this paper is small-cap type decoupling inequalities. 

\begin{theorem}\label{smallcapthm_intro} 
  Let $1\leq L_1\leq L \leq N^{1/2},$ and $\{J\}_{J\in \mc{J}}$ be defined as in the paragraph above. 
  Suppose $p\geq 4.$
  Then for every $\e>0,$
  \begin{equation}\label{s1_intro}
    \|\sum_{J\in \mc{J}}f_J\|_{L^p(\R)} \lesssim_\e N^\e  \left( \frac{N^{\frac{1}{2}-\frac{2}{p}}L^{\frac{2}{p}}}{L_1^{1-\frac{2}{p}}} +\left( \frac{N^{1/2}}{L_1} \right)^{\frac{1}{2}-\frac{1}{p}} \right)\left( \sum_{J\in \mc{J}} \|f_J\|_{L^p(\R)}^p \right)^{1/p}
  \end{equation}
  for function $f_J:\R\rightarrow \C$ with $\supp \widehat{f_J}\subset J.$
\end{theorem}

Inequality \eqref{s1_intro} is sharp up to $C_\e N^\e$ for every fixed $p,L,L_1$ satisfying the condition in Theorem \ref{smallcapthm_intro}. The first factor in front of $( \sum_{J\in \mc{J}} \|f_J\|_{L^p(\R)}^p)^{1/p}$ is sharp because of the example $\widehat{f_J}$ equals to a smooth bump adapted to $J$ with height $1$ for every $J\in \mc{J}.$ The calculation is similar to the one in the paragraph below Theorem \ref{decouplingthm_intro}. The second factor is sharp because of the example $\widehat{f_J}$ equals to a random sign times a smooth bump adapted to a ball of radius $ L^2/N^2$ inside $J$ with height $1$ for every $J\in \mc{J},$ where the random signs are chosen so that $\int_{\R}|\sum_{J}f_J|^p\sim \int_{\R} (\sum_{J}|f_J|^2)^{p/2}$ by Khintchine's inequality. 

The structure of the proof of Theorem \ref{smallcapthm_intro} is similar to that of Theorem 3.1 in \cite{demeter2020small},  consisting of three ingredients: refined decoupling for the canonical partition, refined flat decoupling, and an incidence estimate. Refined decoupling for the canonical partition is a refined version of Theorem \ref{decouplingthm_intro} which we will prove in Sections \ref{decsec}, \ref{highlowsec}, and \ref{inductionsec} in order to derive Theorem \ref{decouplingthm_intro}.  We show the other two counterparts in Section \ref{smallcapsec}. 

\subsection{$L^p$ estimates for short generalized Dirichlet ploynomials}



A straight corollary of Theorem \ref{smallcapthm_intro} is essentially sharp $L^p$ estimates for short generalized  Dirichlet polynomials $\sum_{n=1}^{N^{1/2}}b_n e^{ita_n}.$
\begin{corollary}\label{discretesmallcor_intro}
  Let $\{a_n\}_{n=1}^{N^{1/2}}$ be a short generalized  Dirichlet sequence.
  Suppose $p\geq 4$ and $N\leq T \leq N^2.$ We have for every $\e>0,$ 
  \begin{equation}\label{s120_intro}
    \|\sum_{n=1}^{N^{1/2}}b_n e^{ita_n}\|_{L^p(B_T)} \lesssim_\e  N^{\e}   \left(N^{\frac{1}{2}} +T^{\frac{1}{p}} N^{\frac{1}{4}-\frac{1}{2p}} \right) \|b_n\|_{\ell^p}.
  \end{equation}
  for every $B_T,$ and every $\{b_n\}_{n=1}^{N^{1/2}}\subset \C,$ 
\end{corollary}

If we let $L\in [1,N^{1/2}]$ be the integer such that $N^2/L^2=T,$ then Corollary \ref{discretesmallcor_intro} follows from Theorem \ref{smallcapthm_intro} with that $L,$ and $L_1=1,$ applied to functions $f_J(t)=b_ne^{ita_n} \phi(t)$ for every $J,$ where $\phi$ is a Schwartz function adapted to $B_T$ with Fourier support inside $B_{T^{-1}}(0).$

The inequality \eqref{s120_intro} is sharp up to $C_\e N^\e$. This is from discrete versions of the examples described below Theorem \ref{smallcapthm_intro}, taken with $L_1=1$: $b_n=1$ for every $n,$ and $b_n$ equal to random signs.

We will in fact prove a more general version of Theorem \ref{smallcapthm_intro} which allows us to get essentially sharp $(\ell^q,L^p)$ estimates for $\sum_{n=1}^{N^{1/2}}b_n e^{ita_n}$ in the range $p\geq 4, \frac{1}{p}+\frac{3}{q}\leq 1.$ See Theorem \ref{smallcapbilthm} and Corollary \ref{discretesmallcor}.


After this work was done we learned from James Maynard a general transference method, which can in particular transfer the $L^p$ estimate on a short generalized Dirichlet polynomial to a $2$-dimensional $L^p$ estimate on an exponential sum with frequency support near a convex curve in $\R^2.$ This allows us to derive Corollary \ref{discretesmallcor_intro} directly from the small cap decoupling inequalities for the parabola in \cite{demeter2020small}. We provide that particular argument in detail in Section \ref{appsec}. 

The starting point of this paper was to see whether decoupling methods could be used to make progress on Montgomery's conjecture on Dirichlet polynomials  \cite{montgomery1971topics,montgomery1994ten}.
Our investigation led us in a different direction, proving decoupling inequalities for short generalized Dirichlet sequences. 
\begin{conjecture}[Montgomery's Conjecture]\label{montconjmean}
  For every $p\geq 2$ and every $\e>0$ we have
  \begin{equation}\label{montconjmeaneqn}
      \| \sum_{n=N+1}^{2N} b_n n^{it} \|_{L^p(B_T)} \leq C_\e T^\e N^{1/2}(N^{p/2}+T)^{1/p} \|b_n\|_{\ell^\infty}
  \end{equation}
  for every ball $B_T$ of radius $T,$ and every $\{b_n\}_{n=N+1}^{2N}\subset \C.$
\end{conjecture}

Conjecture \ref{montconjmean} is widely open. In fact it has significant implications which are also hard conjectures. It is shown in \cite{montgomery1971topics} that
Conjecture \ref{montconjmean} implies the density conjecture for the Riemann zeta function.  Bourgain observed in \cite{bourgain1991remarks,bourgain1993distribution} that a stronger version of Conjecture \ref{montconjmean} on large value estimate of Dirichlet polynomials implies the Kakeya maximal operator conjecture in all dimensions. Conjecture \ref{montconjmean} itself also implies a weaker statement that a Kakeya set has full Minkowski dimension (see \cite{green2003restriction}).

Our Corollary \ref{discretesmallcor_intro} proves some $L^p$ estimates for ``short" Dirichlet polynomials  which do not directly connect to Montgomery's conjecture. In fact we believe to make progress on Montgomery's conjecture significant new ideas are needed. 

On the other hand, combining Theorem \ref{decouplingthm_intro} with flat decoupling we obtain $\ell^2 L^p$ decoupling inequalities for generalized Dirichlet sequences (with $N$ many terms instead of $N^{1/2}$), and the decoupling inequalities we get are essentially sharp for the class of generalized Dirichlet sequences. As a corollary we have essentially sharp $(\ell^2,L^p)$ estimates on generalized Dirichlet polynomials, but the Dirichlet polynomial $\sum_{n=N+1}^{2N} b_n e^{it\log n},$ has more structure and admits better estimates. This has to do with examples of generalized Dirichlet sequences containing a $cN^{1/2}$-term AP with common difference $CN^{-1/2},$ which $\{\log n\}_{n=N+1}^{2N}$ cannot contain by a number theory argument.  We discuss these in detail in Section \ref{longsec}.

The paper is structured as follows. In Section \ref{localconstsec} we will illustrate the wave packet structure of functions with frequency support in a fat AP. In Section \ref{BKBRsec} we  prove a bilinear Kakeya-type estimate and a bilinear restriction-type estimate for functions with frequency support in a neighborhood of a short generalized Dirichlet sequence $\{a_n\}_{n=1}^{N^{1/2}}.$ 
Section \ref{decsec}, \ref{highlowsec}, and \ref{inductionsec} are dedicated to proving Theorem \ref{decouplingthm_intro}.
Section \ref{decsec} introduces a refined decoupling inequality for the canonical partition (Theorem \ref{refineddecthm}), which implies Theorem \ref{decouplingthm_intro}, and which we will actually prove.
Section \ref{highlowsec} sets up a high-low frequency decomposition for square functions at different scales, and in Section \ref{inductionsec} we finish the proof of Theorem \ref{refineddecthm}. Section \ref{longsec} discusses the decoupling problem for ($N$-term) generalized Dirichlet sequences. In Section \ref{smallcapsec} we prove Theorem \ref{smallcapthm_intro}. Section \ref{appsec} is about the transference method for one-dimensional exponential sum estimates like \eqref{s120_intro}.

\vspace{5mm}
\noindent
{\bf Notation.} $C$ will denote a positive absolute constant that may vary from lines to lines, and it may be either small or large. $A\lesssim B$ means $A\leq CB.$ $A\sim B$ means $A\lesssim B$ and $B\lesssim A.$ We will also use $\mc{O}(A)$ to denote a quantity that is less than or equal to $CA.$ $A\lesssim_q B$ will mean $A\leq C_q B$ for some constant depending on $q.$ Similarly $\mc{O}_q(A)$ denotes a quantity that is less than or equal to $C_qA.$ There will be a parameter $N$ and $A\lessapprox B$ denotes $A\lesssim_\e N^\e B$ for every $\e>0.$

\vspace{5mm}
\noindent
{\bf Acknowledgements.}  We would like to thank James Maynard for thoughtful discussions related to this paper. In particular we learned from him the transference method described in Section \ref{appsec}.
LG is supported by a Simons Investigator grant.

\section{Locally constant property} \label{localconstsec}
We set up some notations and describe the locally constant property related to fat APs in this section.

\begin{definition} \label{dualI}
We let $P^{\delta}_{v}(a)$ denote the $\delta$-neighborhood of the arithmetic progression on $\R$ which contains $a$ and has common difference $v.$ We call $P^{\delta}_v(x_0)\cap B_{R}(x_0),$ or simply $P^{\delta}_v\cap B_{R},$ a fat AP with thickness $\delta,$ common difference $v,$ and diameter $R.$ We will call $P^{R^{-1}}_{v^{-1}}\cap B_{\delta^{-1}}$ a fat AP dual to $P^{\delta}_v\cap B_{R}.$ 
\end{definition}

To exploit the locally constant property of a function with frequency support in a fat AP, we first construct a family of functions $\psi_k:\R\rightarrow \C$ adapted to a fat AP (in the frequency space).
\begin{lemma} \label{psilem}
  For every $x_0\in \R,$ $\delta\leq v/2,$ $M\geq 1,$ and $k\geq 1$ there exists a function $\psi_k:\R \rightarrow \C$ with the property
\begin{equation}\label{supppsi}
  \widehat{\psi_k}(\xi)=1 \text{ on } P_{v}^{\delta}(x_0) \cap B_{Mv}(x_0), \qquad
  \supp \widehat{\psi_k} \subset P_v^{2\delta}(x_0) \cap B_{8^kMv}(x_0),
\end{equation}
and $\psi_k$ decays at order $k$ outside of the dual fat AP $P_{v^{-1}}^{(Mv)^{-1}}(0) \cap B_{\delta^{-1}}(0):$
\begin{equation}\label{decaypsi}
   (M\delta) 1_{P_{v^{-1}}^{(Mv)^{-1}}(0) \cap B_{\delta^{-1}}(0)}\lesssim_k |\psi_k(x)| \lesssim_{k}  \frac{M\delta}{\left(1+\frac{d(x,v^{-1}\Z)}{(Mv)^{-1}}\right)^k \left( 1+\frac{d(x,B_{\delta^{-1}}(0))}{\delta^{-1}} \right)^{k}}.
\end{equation}
\end{lemma}

We say such a $\psi_k$ is adapted to the fat AP $P^{\delta}_v(x_0) \cap B_{Mv}(x_0)$ in the frequency space with order of decay $k.$
\begin{proof}
Since translation in frequency space corresponds to modulation in the physical space, we may assume $x_0=0.$

We start with the Dirichlet kernel
$$D_M(x)=\sum_{|j|\leq M} e^{2\pi i jx}=\frac{\sin((2M+1)\pi x)}{\sin(\pi x)}.$$
We define $\tilde{D}_1(x)=D_M(x).$ Then we define $\tilde{D}_k(x)$ inductively by
$$\tilde{D}_k(x)=d_k^{-1}\tilde{D}_{k-1}(x) D_{8^{k-1}M/2}(x),$$
where $d_k=\|\widehat{D_{8^{k-1}M/2}}\|_{L^1(\R)}$ is the total measure of the measure $\widehat{D_{8^{k-1}M/2}}.$
Equivalently we can define $\tilde{D}_k$ explicitly as
$$\tilde{D}_k=\tilde{d}_k D_M\prod_{1\leq s \leq k-2}D_{8^s M/2}$$
for some suitable constant $\tilde{d}_k>0.$

Since $\tilde{D}_1=D_M$ has the property that
$$ \widehat{\tilde{D}_1}(\xi) = \sum_{|j|\leq M} \delta_0(\xi-j),  $$
by induction we can show that
$$ \widehat{\tilde{D}_k} (\xi) =\sum_{|j|\leq M} \delta_0(\xi-j)+\sum_{M<|j|\leq 8^{k}M/4} b_{j,k} \delta_0 (\xi-j) $$
for some $0\leq b_{j,k} \leq 1.$ From the explicit expression of the Dirichlet kernel we see that $\tilde{D}_1$ decays at order $1$ outside of $P^{M^{-1}}_{1}(0):$
$$|\tilde{D}_0(x)|=|D_M(x)| \lesssim \frac{M}{1+\frac{d(x,\Z)}{M^{-1}}}.$$
By induction on $k$ we obtain $\tilde{D}_k$ decays at order $k$ outside of $P^{M^{-1}}_{1}(0):$
\begin{equation}\label{decaytildeD}
  |\tilde{D}_k(x)|\lesssim_{k} \frac{M}{\left(1+\frac{d(x,\Z)}{M^{-1}}\right)^k}.
\end{equation}

Now let $\phi(x)$ be a Schwartz function such that $\hat{\phi}$ is a smooth bump adapted to $B_1(0)$
$$ \hat{\phi}(\xi)=1 \text{ on } B_1(0), \qquad \supp \hat{\phi} \subset B_2(0).$$
Let $\phi_{\delta^{-1}}(x)$ be the function $\phi(\delta x).$ Note that $\phi_{\delta^{-1}}$ decays rapidly outside of $B_{\delta^{-1}}(0).$ We let $\psi_k$ be given by
$$\widehat{\psi_k}:= \widehat{\phi_{\delta^{-1}}} * \widehat{\tilde{D}_k}(v^{-1}\xi) /v= \sum_{|j|\leq M} \widehat{\phi_{\delta^{-1}}}(\xi-jv)+\sum_{M<|j|\leq 8^{k}M/4} b_{j,k} \widehat{\phi_{\delta^{-1}}}(\xi-jv).$$
From this definition we immediately see property \eqref{supppsi} holds. Writing $\psi_k$ as
$$\psi_k(x)= \phi_{\delta^{-1}}(x) {\tilde{D}_k}(vx)$$
we observe from \eqref{decaytildeD} and the rapid decay of $\phi_{\delta^{-1}} $ outside $B_{\delta^{-1}}(0)$ that \eqref{decaypsi} holds.

\end{proof}

For every fat AP $P=P^{(Mv)^{-1}}_{v^{-1}}(x_0)\cap B_{\delta^{-1}}(x_0)$ with $\delta\leq v,$ and every $k\geq 100,$ let $W_{P,k}$ be the weight function
$$W_{P,k}(x)=\frac{1}{\left(1+\frac{d(x,x_0+v^{-1}\Z)}{(Mv)^{-1}}\right)^k \left( 1+\frac{d(x,B_{\delta^{-1}}(x_0))}{\delta^{-1}} \right)^{k}}.$$
We will use the following notation
$$\int_{W_{P,k}} f(x) dx:=\int_{\R} f(x)W_{P,k}(x) dx, \quad \dashint_{W_{P,k}} f(x) dx:=\frac{1}{\|W_{P,k}\|_{L^1(\R)}} \int_{\R} f(x)W_{P,k}(x) dx,$$
$$\|f\|_{\stkout{L}^p(W_{P,k})}:=\left(\dashint_{W_{P,k}} |f|^p(x) dx\right)^{1/p}.$$
For measurable sets $E\subset \R$ we use similar notations for average integrals and $L^p$ norms:
$$\dashint_{E} f(x) dx:=\frac{1}{|E|} \int_{E} f(x) dx, \quad
\|f\|_{\stkout{L}^p(E)}:=\left(\dashint_{E} |f|^p(x) dx\right)^{1/p}.$$

For a fat AP $P$ and translated copies of a smaller fat AP $P',$ we have the following pointwise inequality
\begin{equation}\label{weighttransitive}
  1_{P}(x)\lesssim_k \sum_{P'\subset P} W_{P',k}(x)\lesssim_k W_{P,k}(x).
\end{equation}
Here $\sum_{P'\subset P}$ means summing over a tiling (with $\mc{O}(1)$ overlap) of $P$ by $P'.$

If we look at  translated copies $P''$ of $P,$ we have
\begin{equation}\label{weighttransitive2}
  \sum_{P''\subset \R} W_{P'',k}(x) W_{P,k}(P'') \lesssim_{k} W_{P,k}(x).
\end{equation}
Here $\sum_{P''\subset \R}$ means summing over a tiling (with $\mc{O}(1)$ overlap) of $\R$ by $P'',$ and $W_{P,k}(P'')$ is defined to be $W_{P,k}(\sup P''),$ which is comparable to $W_{P,k}(x)$ for any $x\in P''.$

\begin{proposition}[locally constant property] \label{locconstprop}
  Suppose $f$ satisfies $\supp \hat{f}\subset P_{v}^\delta \cap B_{Mv}.$ Then for every dual fat AP $P=P^{(Mv)^{-1}}_{v^{-1}} \cap B_{\delta^{-1}}$ and every $1\leq q<p< \infty$ we have
  $$\|f\|_{\stkout{L}^p(W_{P,k})}\lesssim_{p,q,k} \|f\|_{\stkout{L}^q(W_{P,\frac{qk}{p}})}, \qquad \text{if} \qquad \frac{qk}{p}\geq 100,$$
  $$\|f\|_{L^\infty(P)}\lesssim_{k}\|f\|_{\stkout{L}^1(W_{P,k})}.$$
\end{proposition}
\begin{proof}
  We first prove the second inequality. Fix $k\geq 100.$ From \eqref{supppsi} we have
  $$f(x)=f*\psi_k(x)=\int_{\R} f(y)\psi_{k}(x-y)dy,$$
  where $\psi_k$ is the function in Lemma \ref{psilem} adapted to $P_{v}^\delta \cap B_{Mv}$ in the frequency space with order of decay $k.$
  Therefore for $x\in P$ we have
  \begin{align*}
    |f(x)| & \leq \int_{\R} |f(y)||\psi_{k}(x-y)|dy \\
     & \leq \int_{\R} |f(y)| \sup_{x\in P} |\psi_{k}(x-y)| dy \\
     & \lesssim_k \delta M \int_{\R} |f(y)| W_{P,k}(y) dy \\
     & \sim_k \dashint_{W_{P,k}} |f(y)| dy.
  \end{align*}
  For the third inequality we used \eqref{decaypsi}.

  Now we prove the first inequality in the proposition. We claim that from \eqref{weighttransitive2} (applied with $k$ replaced by $\frac{qk}{p}$) and the assumption $q<p$ we only need to show
  \begin{equation}\label{3}
    \|f\|_{\stkout{L}^p(P)} \lesssim_{p,q,k} \|f\|_{\stkout{L}^{q}(W_{P,k})}.
  \end{equation}
  Indeed if \eqref{3} holds, then
  \begin{align*}
    \int_{W_{P,k}} |f|^p & \lesssim_k \sum_{P'\subset \R} \int_{P'} |f|^p W_{P,k}(P') \\
     & \lesssim_{p,q,k} |P|^{1-\frac{p}{q}} \sum_{P'\subset \R} W_{P,k}(P') ( \int_{W_{P',\frac{qk}{p}}} |f|^q  )^{p/q}\\
     & \leq |P|^{1-\frac{p}{q}} \left( \int_{\R} |f(x)|^q \sum_{P'\subset \R} W_{P,k}(P')^{q/p} W_{P',\frac{qk}{p}}(x) dx \right)^{p/q} \\
     & \lesssim_{p,q,k} |P|^{1-\frac{p}{q}} \left( \int |f(x)|^q \sum_{P'\subset \R} W_{P,\frac{qk}{p}}(P') W_{P',\frac{qk}{p}}(x) dx \right)^{p/q} \\
    (\text{by} \eqref{weighttransitive2}) \, & \lesssim_{p,q,k} |P|^{1-\frac{p}{q}} ( \int |f|^q W_{P,\frac{qk}{p}} )^{p/q},
  \end{align*}
  which is exactly the first inequality in the proposition.

  To show \eqref{3} we observe that the second inequality in the proposition together with H\"{o}lder's inequality implies that
  $$\|f\|_{\stkout{L}^p(P)}\leq \|f\|_{L^\infty (P)} \lesssim_{p,q,k} \|f\|_{\stkout{L}^q (W_{P,\frac{qk}{p}})},$$
  which is \eqref{3}.
\end{proof}

\section{Bilinear Kakeya-type and restriction-type estimates}\label{BKBRsec}

Kakeya and restirction-type estimates are closely related to decoupling, and we will use the bilinear version of them in the proof of Theorem \ref{decouplingthm_intro}, but first we need to introduce a more general decoupling set-up for the purpose of induction.


\subsection{General set-up}
To prove Theorem \ref{decouplingthm_intro} we will do a broad-narrow argument which involves re-scaling of a segment of $\{a_n\}_{n=1}^{N^{1/2}}.$ To properly set up our induction hypothesis we consider the following more general class of generalized Dirichlet sequences.
\begin{definition}[Generalized Dirichlet sequence] \label{defnDirichlet}
Let $\theta \in (0,1],$ and $N \geq 2.$  We call $\{a_n\}_{n=1}^{N}$ a generalized Dirichlet sequence (with parameters $N,\theta$) if it satisfies the property
\begin{equation}\label{defa_n}
  a_{2}-a_1 \in [\frac{1}{4N},\frac{4}{N}], \qquad (a_{i+2}-a_{i+1})-(a_{i+1}-a_i) \in [\frac{\theta}{4N^2},\frac{4\theta}{N^2}].
\end{equation}
We will call $\{a_n\}_{n=1}^{N^{1/2}}$ satisfying \eqref{defa_n_intro} a $N^{1/2}$-
short generalized Dirichlet sequence (with parameters $N,\theta$).
\end{definition}
As before we write ``short'' for ``$N^{1/2}$-short'' for simplicity. Comparing with Definition \ref{defa_n_intro} we see an extra parameter $\theta$ which measures the convexity of the sequence. From now on we use Definition \ref{defnDirichlet} for the definition of generalized Dirichlet sequence. 

We shall also incorporate $\theta$ in our decoupling set-up.
From the spacing property \eqref{defa_n} of $\{a_n\}_{n=1}^{N^{1/2}}$ we see that each $I\in \mathcal{I}$ is essentially contained in an $L^2\theta/N^2-$neighborhood of an arithmetic progression. Indeed if we define $v_j=a_{(j-1)L+2}-a_{(j-1)L+1},$ then $I_j$ is contained in the $CL^2\theta/N^2-$neighborhood of the arithmetic progression containing $a_{jL}$ with common difference $v_j,$
that is,
$$I_j\subset P^{CL^2\theta/N^2}_{v_j}(a_{jL}) \cap B_{CL/N}(a_{jL}).$$

We let 
$\Omega$ be the $\theta L^2/N^2$-neighborhood of $\{a_n\}_{n=1}^{N^{1/2}}.$
For $1\le L\le c N^{1/2}$ and each $j=1,\ldots,\frac{N^{1/2}}{L}$, define
\[ I_j=\bigcup_{i=(j-1)L+1}^{jL} B_{\theta L^2/N^2}(a_i)
.\]
We denote the collection of $I_j$ by $\mc{I},$ and consider the partition
\[\Omega=\bigsqcup_{I \in \mc{I}} I.\]
This will be our new decoupling set-up for the canonical partition, and from now on the notation here supersedes that in the Introduction. For small-cap type decoupling we postpone the description of the corresponding general set-up to Section \ref{smallcapsec}.

\subsection{Analogies between $\{a_n\}_{n=1}^{N^{1/2}}$ and $\mathbb{P}^1$} \label{comp}

For $I=I_j\in \mc{I},$ we let
$$
 \tilde{I}_j:= P^{CL^2\theta/N^2}_{v_j}(a_{jL}) \cap B_{CL/N}(a_{jL})
$$
with $C$ large enough so that
$$I=I_j\subset \tilde{I}_j=\tilde{I}.$$
Here $v_j=a_{(j-1)L+2}-a_{(j-1)L+1}$ and $v_j \sim N^{-1}.$

For each $I\in \mathcal{I},$ we denote by $P_I(x)$ the fat AP dual to $\tilde{I}$ and centered at $x,$ that is,
\begin{equation}\label{1}
  P_I(x):=P_{v_j^{-1}}^{CN/L}(x) \cap B_{CN^2/(L^2\theta)}(x)
\end{equation}
if $I=I_j,$ and we simply write $P_I$ if stressing the center $x$ is unnecessary.
We let $P(L,y)$ denote a larger fat AP
\begin{equation}\label{defnLball}
  P(L,y):= P_{v_1^{-1}}^{CN^{3/2}/L^2}(y) \cap B_{CN^2/(L^2\theta)}(y),
\end{equation}
and we simply write $P(L)$ if stressing the center $y$ is unnecessary. If $L\lesssim N^{1/4}$ we have $N^{3/2}/L^2 \geq N$ and in that case $P(L)$ is a ball $B_{CN^2/(L^2\theta)}.$

The starting point of this paper is to make use of an analogy between the extension operator on $\{a_n\}_{n=1}^{N^{1/2}}$
$$\{b_n\}_{n=1}^{N^{1/2}} \mapsto \sum_{n=1}^{N^{1/2}} b_n e^{it a_n}$$ 
and the extension operator on the truncated parabola $\mathbb{P}^1$
$$f\mapsto \int_{[-1,1]} f(\xi)e^{i(x\xi+t\xi^2)}d\xi.$$
We list the correspondence between objects in this paper and in the parabola setting. For simplicity we assume $\theta=1$ in the following list.
\begin{enumerate}
    \item The parameter $L\in [1,N^{1/2}]$ is the length of the ``cap'' that we are looking at, and that determines a canonical neighborhood $\Omega$ with width $L^2/N^2.$ The corresponding parameter in the parabola setting is $R,$ which determines the length ($R^{-1/2}$) of the cap and a canonical neighborhood with width $R^{-1}.$
    \item The $\tilde{I},P_I$ defined above is analogous to the cap and tube in the context of parabola decoupling. Let $\Theta$ be a partition of $\mc{N}_{R^{-1}}(\mathbb{P}^1)$, the $R^{-1}$-neighborhood of the truncated parabola $\mathbb{P}^1$ (over $[-1,1]$), into $R^{-1/2}\times R^{-1}$ caps $\theta$. The dual object of $\theta$ is a tube $T$ of dimension $R^{1/2}\times R.$
    \item $P(L)$ is defined to be the smallest fat AP with the property that, for a function $F$ with frequency support on $\Omega,$ ``restricting'' $F$ in the physical space to $P(L)$ will essentially preserve its frequency support. The corresponding object for the parabola is $B_R,$ a ball of radius $R.$ 
\end{enumerate}

\begin{figure}
\begin{tabular}{ cc }
\scalebox{0.7}{
\begin{tikzpicture}[x=0.75pt,y=0.75pt,yscale=-1,xscale=1]

\draw  [color={rgb, 255:red, 65; green, 117; blue, 5 }  ,draw opacity=1 ][line width=3.75]  (134.58,144.44) .. controls (134.58,73.75) and (191.88,16.44) .. (262.58,16.44) .. controls (333.27,16.44) and (390.58,73.75) .. (390.58,144.44) .. controls (390.58,215.13) and (333.27,272.44) .. (262.58,272.44) .. controls (191.88,272.44) and (134.58,215.13) .. (134.58,144.44) -- cycle ;
\draw   (272.5,17) -- (270.9,272) -- (252.65,271.88) -- (254.25,16.89) -- cycle ;
\draw   (390.21,152.61) -- (135.24,156.75) -- (134.94,138.51) -- (389.91,134.37) -- cycle ;
\draw   (293.08,21.36) -- (250.55,272.79) -- (232.56,269.75) -- (275.09,18.32) -- cycle ;
\draw  [draw opacity=0][dash pattern={on 2.53pt off 3.02pt}][line width=2.25]  (286.1,83.37) .. controls (305.63,94.04) and (320.09,112.52) .. (325.04,134.49) -- (247.29,151.27) -- cycle ; \draw  [dash pattern={on 2.53pt off 3.02pt}][line width=2.25]  (286.1,83.37) .. controls (305.63,94.04) and (320.09,112.52) .. (325.04,134.49) ;

\draw (107,89.4) node [anchor=north west][inner sep=0.75pt]  [font=\large]  {$\textcolor[rgb]{0.25,0.46,0.02}{B}\textcolor[rgb]{0.25,0.46,0.02}{_{R}}$};
\draw (233,39.4) node [anchor=north west][inner sep=0.75pt]    {$T_{1}$};
\draw (292,40.4) node [anchor=north west][inner sep=0.75pt]    {$T_{2}$};
\draw (346,154.4) node [anchor=north west][inner sep=0.75pt]    {$T_{R^{1/2}}$};
\end{tikzpicture}} & 
\scalebox{0.75}{

\tikzset{every picture/.style={line width=0.75pt}} 

\begin{tikzpicture}[x=0.75pt,y=1.15pt,yscale=-1,xscale=1]

\draw    (91.5,229.33) -- (111.5,229.33) ;
\draw    (149.5,229.33) -- (169.5,229.33) ;
\draw    (210.5,230.33) -- (230.5,230.33) ;
\draw    (270.5,230.33) -- (290.5,230.33) ;
\draw    (328.5,230.33) -- (348.5,230.33) ;
\draw    (386.5,230.33) -- (406.5,230.33) ;
\draw    (91.5,209.33) -- (111.5,209.33) ;
\draw    (152.5,209.33) -- (172.5,209.33) ;
\draw    (214.5,209.33) -- (234.5,209.33) ;
\draw    (275.5,209.33) -- (295.5,209.33) ;
\draw    (335.5,209.33) -- (355.5,209.33) ;
\draw    (396.5,209.33) -- (416.5,209.33) ;
\draw    (90.5,131.33) -- (110.5,131.33) ;
\draw    (160.5,131.33) -- (180.5,131.33) ;
\draw    (228.5,130.33) -- (248.5,130.33) ;
\draw    (297.5,130.33) -- (317.5,130.33) ;
\draw    (351.5,130.33) -- (371.5,130.33) ;
\draw    (416.5,130.33) -- (436.5,130.33) ;
\draw [color={rgb, 255:red, 65; green, 117; blue, 5 }  ,draw opacity=1 ][line width=3.75]    (90.5,87.33) -- (139.6,87.33) ;
\draw [color={rgb, 255:red, 65; green, 117; blue, 5 }  ,draw opacity=1 ][line width=3.75]    (148.85,87.33) -- (197.95,87.33) ;
\draw [color={rgb, 255:red, 65; green, 117; blue, 5 }  ,draw opacity=1 ][line width=3.75]    (208.67,87.33) -- (257.77,87.33) ;
\draw [color={rgb, 255:red, 65; green, 117; blue, 5 }  ,draw opacity=1 ][line width=3.75]    (269.4,87.33) -- (318.5,87.33) ;
\draw [color={rgb, 255:red, 65; green, 117; blue, 5 }  ,draw opacity=1 ][line width=3.75]    (329.5,87.33) -- (378.6,87.33) ;
\draw [color={rgb, 255:red, 65; green, 117; blue, 5 }  ,draw opacity=1 ][line width=3.75]    (387.85,87.33) -- (436.95,87.33) ;
\draw [line width=3]  [dash pattern={on 3.38pt off 3.27pt}]  (130.5,144) -- (130.5,193) ;

\draw (37,219.4) node [anchor=north west][inner sep=0.75pt]    {$P_{I_{1}}$};
\draw (37,198.4) node [anchor=north west][inner sep=0.75pt]    {$P_{I_{2}}$};
\draw (35,122.4) node [anchor=north west][inner sep=0.75pt]    {$P_{I_{N^{1/2} /L}}$};
\draw (35,80.4) node [anchor=north west][inner sep=0.75pt]    {$\textcolor[rgb]{0.25,0.46,0.02}{P}\textcolor[rgb]{0.25,0.46,0.02}{(}\textcolor[rgb]{0.25,0.46,0.02}{L}\textcolor[rgb]{0.25,0.46,0.02}{)}$};

\end{tikzpicture}} \\
\end{tabular}
    \caption{The ball $B_R\subset\R^2$ contains the union of tubes $T_i$ having the same center, each which is dual to $\theta_i$, where $\underset{i}{\sqcup}\theta_i$ partitions $\mc{N}_{R^{-1}}(\mathbb{P}^1)$. On the right, we see analogous dual fat APs, one $P_{I_i}$ per $I_i$ which partition $\Omega$ into $L$ consecutive intervals. We see that $P(L)$ contains the union of the $P_{I_i}$ which have the same starting point. }
    \label{fig:balls}
\end{figure}
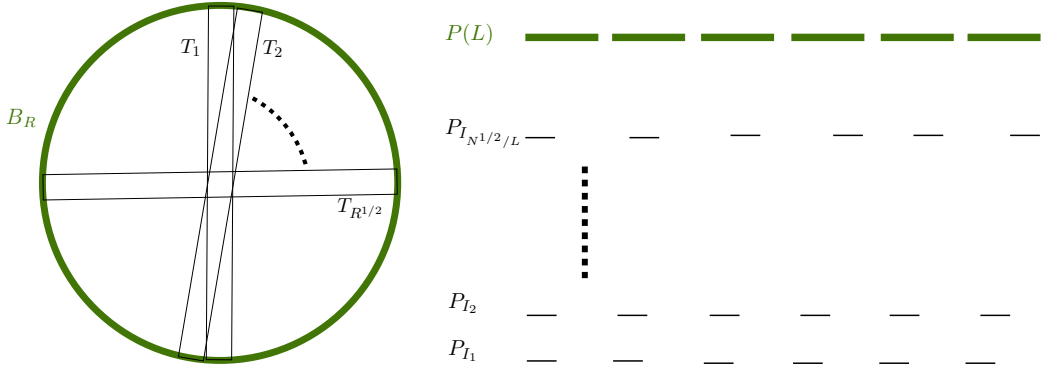

See Figure \ref{fig:balls} which illustrates the analogous properties of tubes $T$ with the ball $B_R$ and fat APs $P_I$ with $P(L)$. Bourgain made use of the first two analogies in \cite{bourgain1991remarks,bourgain1993distribution}. The new ingredient we need is the third analogy, which gives an appropriate notion of ball in the short generalized Dirichlet sequence setting.
It is very important that we define $P(L)$ to be the smallest fat AP with such a property. If we naively use $B_{N^2/L^2}$ as the ball $P(L),$ the whole argument that follows will break down.

To make the third point precise, we prove the following lemma. We introduce one more notation. For a general fat AP $P=P^\delta_v(x_0)\cap B_{Mv}(x_0)$ and $s>0,$ $sP$ will denote the fat AP $P_v^{s\delta}(x_0)\cap B_{sMv}(x_0).$

\begin{lemma}\label{balllem}
   Fix a $P(L).$ For every $I\in \mathcal{I}$ and every $P_I$ with $P_I\cap P(L)\neq \emptyset,$ $P_I$ is contained in $2P(L).$
\end{lemma}
\begin{proof}
  In fact since for every $j,$ $|v_j-v_1|\lesssim N^{-3/2}\theta,$ we have $|v_j^{-1}-v_1^{-1}|\lesssim N^{1/2}\theta.$ Therefore $P_I \cap P(L) \neq \emptyset$ implies
$$d(P_I,P(L))\lesssim (N^{1/2}\theta) \frac{N^2/(L^2\theta)}{N}=\frac{N^{3/2}}{L^2},$$
which implies $P_I \subset 2P(L)$ if $C$ is large enough in the definition of $P(L).$
\end{proof}

We note that the above Lemma holds if we replace $I,\mc{I},P_I,P(L)$ by $\theta,\Theta,T,B_R$ respectively.

\subsection{Transversality and Bilinear Kakeya-type estimate}

For $I\in \mathcal{I},$ let $v_I^{-1}$ denote the common difference of $P_I,$ that is, if $I=I_j$ then $v_I=v_j.$ We say $I,J\in \mc{I}$ are {transversal} if $|v_I^{-1}-v_J^{-1}|\gtrsim N^{1/2}\theta,$ or equivalently, if $d(I,J)\gtrsim N^{-1/2}$ on $\R.$ We now prove a bilinear Kakeya-type estimate for two transversal families of $P_I.$
\begin{proposition}[bilinear Kakeya-type estimate] \label{BKprop}
  Suppose $g_1=\sum_I a_I 1_{P_I}$ and $g_2=\sum_J b_J 1_{P_J}$ where $a_I, b_J$ are positive real numbers, $I,J\in \mc{I}$ and $P_I$ are transversal to $P_J.$ Then
  \begin{equation}\label{BK}
    \dashint_{P(L)} g_1 g_2 \lesssim \dashint_{2P(L)} g_1 \dashint_{2P(L)} g_2.
  \end{equation}
\end{proposition}

For comparison we state the bilinear Kakeya-type estimates for $R^{1/2}$ by $R$ tubes in $\R^2.$ 

\begin{proposition}
  Suppose $g_1=\sum_i a_i 1_{T_i}$ and $g_2=\sum_j b_j 1_{T_j}$ where $a_i, b_j$ are positive real numbers, $T_i,T_j$ are $R^{1/2}$ by $R$ tubes and every  $T_i$ is transversal to every $T_j$ (in the sense that the angle between $T_i,T_j$ is $\gtrsim 1$). Then
  \begin{equation*}
    \dashint_{B_R} g_1 g_2 \lesssim \dashint_{2B_R} g_1 \dashint_{2B_R} g_2.
  \end{equation*}
\end{proposition}

\begin{proof}[Proof of Proposition \ref{BKprop}]
  For simplicity of notation we assume $C=1$ in \eqref{1}, \eqref{defnLball}. For general $C$ the argument works the same way.

  Since
  $$\dashint_{P(L)} g_1 g_2 \leq \sum_{I,J:\, P_I\cap P(L)\neq \emptyset, \, P_J\cap P(L)\neq \emptyset} a_Ib_J |P(L)|^{-1}|P_I\cap P_J| $$
  it suffices to show that for $I,J$ transversal we have
  \begin{equation}\label{2}
    |P_I\cap P_J| \lesssim \frac{|P_I|^2}{|P(L)|}.
  \end{equation}
  We consider two cases $L\geq C_1N^{1/4}$ and $L\leq C_1N^{1/4}$ separately, where $C_1$ is a sufficiently large constant that will be chosen.

  \vspace{3mm}
  {\bf Case 1: $L\geq C_1N^{1/4}.$}
  Without loss of generality we assume $P_I, P_J$ both start at the origin (meaning that the first term of the underlying AP is $0$). Let $P_{I,k}$ denote the $k-$th interval in $P_I.$ If $V_I, V_J$ are the common difference of $P_I,P_J$ respectively, then from transversality assumption we have $|V_I-V_J|\sim N^{1/2}\theta.$ So for some integer $K\sim \frac{N/L}{N^{1/2}\theta}=\frac{N^{1/2}}{L\theta}$ we have
  $$d(P_{I,k},P_{J,k})\leq N/L \quad \text{  if  } \quad 1\leq k\leq K$$
  and
  $$d(P_{I,k},P_{J,k}) \in [N/L,N] \quad \text{  if  }  \quad K \leq  k \lesssim \frac{N}{N^{1/2}\theta}=\frac{N^{1/2}}{\theta}. $$
  Since $L\geq C_1N^{1/4}$ we know that if $C_1$ is sufficiently large than $\frac{N^{1/2}}{\theta}N=\frac{N^{3/2}}{\theta}$ is larger than $N^2/(L^2\theta),$ which is the diameter of $P_I.$ Therefore we have
  $$|P_I \cap P_J|\lesssim \frac{N^{1/2}}{L\theta} \frac{N}{L}= \frac{N^{3/2}}{L^2\theta}=\frac{|P_I|^2}{|P(L)|}.$$

  \vspace{3mm}
  {\bf Case 2: $L\leq C_1N^{1/4}.$} From the first case we know that
  $$|P_I \cap P_J \cap B_{CN^{3/2}/\theta}|\lesssim N^{3/2}/(L^2\theta).$$
  Therefore by the triangle inequality we have
  $$|P_I \cap P_J|\lesssim  \frac{N^{3/2}}{L^2\theta} \frac{N^2/(L^2\theta)}{N^{3/2}/\theta}=\frac{N^2}{L^4\theta}=
  \frac{|P_I|^2}{|P(L)|}.$$
  Here we recall that $P(L)$ degenerates to the Euclidean ball $B_{N^2/(L^2\theta)}$ if $L\leq N^{1/4}.$

  So we have shown \eqref{2} and hence \eqref{BK}.
\end{proof}

\subsection{Bilinear restriction-type estimate}

To prove a bilinear restriction estimate, we will use the above bilinear Kakeya estimate and induction on $L.$
First we identify where the (square of the) square function $\sum_{I\in \mathcal{I}} |f_I|^2$ is locally constant on. Note that $\supp \widehat{f_I} \subset I-I \subset P_{v_I}^{CL^2\theta/N^2}(0)\cap B_{CL/N}(0).$ Since $|v_I-v_1|\lesssim N^{-3/2}\theta$ for every $I\in \mc{I},$ we have
$$\bigcup_{I\in \mathcal{I}} (I-I) \subset P^{CL\theta/N^{3/2}}_{v_1} \cap B_{CL/N}.$$
Therefore $\sum_I |f_I|^2$ is locally constant on dual fat AP of the form $P_{v_1^{-1}}^{CN/L} \cap B_{CN^{3/2}/(L\theta)}.$ Observe that if we define $L_1=(N^{1/2}L)^{1/2},$ then
$$P_{v_1^{-1}}^{CN/L} \cap B_{CN^{3/2}/(L\theta)} = P_{v_1^{-1}}^{C N^{3/2}/L_1^2} \cap B_{CN^2/(L_1^2\theta)}=CP(L_1).$$

Now suppose $I', I''$ are unions of $I$ in $\mathcal{I},$ and $I',I''$ are transversal in the sense that $d(I',I'') \gtrsim N^{-1/2}$ on $\R.$ Then we have the following bilinear restriction estimate. The proof closely resembles the multilinear Kakeya implies multilinear restriction proof in \cite{bennett2006multilinear}.
\begin{proposition}[bilinear restriction-type estimate] \label{bilresprop}
  Suppose $\supp \widehat{F_1} \subset I'$ and $\supp \widehat{F_2} \subset I''.$ Then we have
  \begin{equation}\label{BR}
    \dashint_{P(L)} |F_1|^2|F_2|^2 \lesssim_\e N^\e |P(L)|^{-2 }\int_{\R} |F_1|^2 \int_{\R} |F_2|^2.
  \end{equation}
\end{proposition}
\begin{proof}
  We define $BR(L)$ to be the smallest constant such that
  $$\dashint_{P(L)} |F_1|^2|F_2|^2 \leq BR(L) |P(L)|^{-2} \int_{\R} |F_1|^2 \int_{\R} |F_2|^2$$
  holds for all $F_1,F_2$ with $\supp \widehat{F_1} \subset I'$ and $\supp \widehat{F_2} \subset I''.$ We let $BK(L)$ be the smallest constant such that
  $$\dashint_{P(L)} g_1 g_2 \leq BK(L) |P(L)|^{-2} \int_{\R} g_1 \int_{\R} g_2$$
  holds for all $g_1=\sum a_I P_I$ ,and $g_2=\sum b_J P_J$ where $a_I, b_J$ are positive real numbers and $I,J\in \mc{I}$ with $I\subset I',$ $J\subset I''.$ Note that with this definition of $BK(L)$ we also have
  \begin{equation}\label{BKdist}
    |P_I|^{-2}\dashint_{P(L)} (\sum_{I\subset I'} |g_{1,I}|*1_{P_I(0)})(\sum_{J\subset I''}|g_{2,J}|*1_{P_J(0)})\lesssim BK(L) |P(L)|^{-2} (\int_\R \sum_{I} |g_{1,I}|)(\int_\R \sum_{J} |g_{2,J}|)
  \end{equation}
  for all functions $g_{1,I},$ $g_{2,J}.$

  We have shown in Proposition \ref{BKprop} that
  $$BK(L)\lesssim 1.$$
  Now we want to show $BR(L)\lesssim_\e N^\e.$
  First we prove
  \begin{equation}\label{BRBK}
    BR(L) \lesssim BR(L_1) BK(L).
  \end{equation}
  From the definition of $BR$ and local $L^2$ orthogonality (Lemma \ref{localL^2orth} below) we have
  \begin{align*}
    \dashint_{P(L)} |F_1F_2|^2 & \lesssim \dashint_{P(L)} \|F_1F_2\|^2_{\stkout{L}^2 (P(L_1,x))} dx \\
     & \lesssim BR(L_1) \dashint_{P(L)} \|F_1\|_{\stkout{L}^2(W_{P(L_1,x),200})}^2 \|F_2\|_{\stkout{L}^2(W_{P(L_1,x),200})}^2\\
     & \lesssim BR(L_1) \dashint_{P(L)} \left( \sum_{I\subset I'} \|F_{1,I}\|^2_{\stkout{L}^2(W_{P(L_1,x),200})} \right)\left( \sum_{J\subset I''} \|F_{2,J}\|^2_{\stkout{L}^2(W_{P(L_1,x),200})} \right).
  \end{align*}
  We claim that
  \begin{equation}\label{10}
    \dashint_{P(L)} \sum_{I,J} \|F_{1,I}\|^2_{\stkout{L}^2(W_{P(L_1,x)},200)} \|F_{2,J}\|^2_{\stkout{L}^2(W_{P(L_1,x)},200)} \lesssim BK(L) |P(L)|^{-2} \|F_1\|^2_{L^2(\R)} \|F_2\|^2_{L^2(\R)},
  \end{equation}
  which together with previous argument will imply \eqref{BRBK}. Since $\sum_{P(L_1)\subset \R} W_{P(L_1,x),200}(P(L_1))\lesssim 1,$ it suffices to show that
  $$\dashint_{P(L)} \sum_{I,J} \|F_{1,I}\|^2_{\stkout{L}^2(P(L_1,x))} \|F_{2,J}\|^2_{\stkout{L}^2(P(L_1,x))} \lesssim_k BK(L) |P(L)|^{-2} \|F_1\|^2_{L^2(\R)} \|F_2\|^2_{L^2(\R)}.$$

  We choose $\psi_{I,200}$ adapted to $P_I(0)$ in the frequency space with order of decay $200$ as in Lemma \ref{psilem}. Let $\phi_I:=\widecheck{\psi_{I,200}}/|P_I|.$ If we define $G_{1,I}= (\widehat{F_{1,I}}/ \widehat{\phi_{I}})\check{\, },$ then due to the support property of $\widehat{F_{1,I}}$ we have pointwise
  \begin{equation}\label{11}
    |\widehat{G_{1,I}}|\sim |\widehat{F_{1,I}}|.
  \end{equation}
  Also by definition we have $F_{1,I}=G_{1,I}* \phi_I.$
  We define $G_{2,J}=(\widehat{F_{2,J}}/\widehat{\phi_J})\check{\, }$ for $F_{2,J}$ in the same way.

  Now for $y\in \R$ such that $x+y\in P(L_1,x),$ we have
  $$|F_{1,I}(x+y)|^2 =|(G_{1,I}*\phi_I)(x+y)|^2\lesssim (|G_{1,I}|^2*|\phi_I|)(x+y)\lesssim |G_{1,I}|^2*1_{CP_I}/|P_I|,$$
  where we used Jensen's inequality for the first inequality. Therefore we have
  $$\|F_{1,I}\|^2_{\stkout{L}^2(P(L_1,x))} \lesssim |G_{1,I}|^2*1_{CP_I}/|P_I|.$$
  and similarly
  $$\|F_{2,J}\|^2_{\stkout{L}^2(P(L_1,x))} \lesssim |G_{2,J}|^2*1_{CP_J}/|P_I|.$$
  Hence using \eqref{BKdist} we obtain
  \begin{align*}
    \dashint_{P(L)} \sum_{I,J} \|F_{1,I}\|^2_{\stkout{L}^2(P(L_1,x))} \|F_{2,J}\|^2_{\stkout{L}^2(P(L_1,x))} & \lesssim |P_I|^{-2}\sum_{I,J} \dashint_{P(L)} (|G_{1,I}|^2*1_{CP_I})(|G_{2,J}|^2*1_{CP_J}) \\
     & \lesssim BK(L) |P(L)|^{-2} (\int_\R \sum_I |G_{1,I}|^2)( \int_\R \sum_J |G_{2,J}|^2) \\
     & \lesssim BK(L) |P(L)|^{-2} (\int_\R \sum_I |F_{1,I}|^2)( \int_\R \sum_J |F_{2,J}|^2) \\
     & \lesssim BK(L) |P(L)|^{-2} (\int_\R |F_1|^2)(\int_\R |F_2|^2),
  \end{align*}
  where the second last inequality is due to \eqref{11}. So we have proved \eqref{10} and therefore \eqref{BRBK}.

  Now we prove $BR(L)\lesssim_\e N^\e .$   Define $L_m=(L_{m-1}N^{1/2})^{1/2}.$
  Fix an $\e>0.$ We define $M$ to be the smallest integer such that $L_M\gtrsim N^{1/2-\e}.$ So $M\lesssim_\e 1.$
  Plugging in $BK(L_m)\lesssim 1$ and applying \eqref{BRBK} repeatedly we get
  $$BR(L)\leq C^M BR(L_M).$$
   Since $BR(L_M)\lesssim_{\e} N^{C\e}$ for some universal constant $C$ (because of the locally constant property Proposition \ref{locconstprop}) we conclude $BR(L)\lesssim_{\e} N^{C\e},$ which is what we want.

\end{proof}

Now we give a proof of the local $L^2$ orthogonality used in the proof above. We denote $(LN^{1/2})^{1/2}$ by $L'.$ So $P(L')=P(L_1)=P_{v_1^{-1}}^{CN/L} \cap B_{CN^{3/2}/(L\theta)}.$
\begin{lemma}[local $L^2$ orthogonality]\label{localL^2orth}
  For every $f_I$ with $\supp \widehat{f_I} \subset I$ we have
  \begin{equation}\label{localL^2ortheqn}
    \|\sum_{I\in \mathcal{I}} f_I\|_{L^2(W_{P(L'),k})}^2 \lesssim_k \sum_{I \in \mathcal{I}} \|f_I\|_{L^2(W_{P(L'),k})}^2
  \end{equation}
\end{lemma}
\begin{proof}
  Due to \eqref{weighttransitive2} it suffices to prove
  $$\|\sum_{I\in \mathcal{I}} f_I\|_{L^2(P(L'))}^2 \lesssim_k \sum_{I \in \mathcal{I}} \|f_I\|_{L^2(W_{P(L'),k})}^2.$$
  We choose $\psi_k$ adapted to $P(L')^*:=P_{v_1}^{CL\theta/N^{3/2}}(0) \cap B_{CL/N}(0)$ in the frequency space with order of decay $k$ as in Lemma \ref{psilem}.  Here $P(L')^*$ is dual to $P(L').$ Since $\supp \widehat{\psi_k} \subset 8^kP(L')^*,$ and $\{I+8^kP(L')^*\}_{I\in \mathcal{I}}$ is $\mathcal{O}_k(1)-$overlapping, we conclude
  \begin{align*}
    \|\sum_{I\in \mathcal{I}} f_I\|_{L^2(P(L'))}^2 & \lesssim_k |P(L')| \|\sum_{I\in \mathcal{I}} f_I \psi_k\|_{L^2(\R)}^2 \\
     & \lesssim_k |P(L')| \sum_{I \in \mathcal{I}} \|f_I \psi_k\|_{L^2(\R)}^2  \\
     & \lesssim_k \sum_{I \in \mathcal{I}} \|f_I\|_{L^2(W_{P(L'),k})}^2.
  \end{align*}
\end{proof}


\section{Decoupling for the canonical partition}\label{decsec}

We focus on proving Theorem \ref{decouplingthm_intro} in Section \ref{decsec}, \ref{highlowsec}, and \ref{inductionsec}, and in these three sections decoupling will refer to decoupling for the canonical partition.

We restate Theorem \ref{decouplingthm_intro} but for all short generalized Dirichlet sequences with $\theta\in (0,1].$ 

\begin{theorem}\label{decouplingthm}
  Let $\Omega$ and $\mathcal{I}$ be defined as in the last paragraphs. Then
  for $2\leq p\leq 6$ and every $\e>0$ we have
  \begin{equation}\label{decoupling}
    \|\sum_{I\in \mathcal{I}} f_{I}\|_{L^p(\R)} \lesssim_\e N^\e \log^C (\theta^{-1}+1)  \left( \sum_{I\in \mathcal{I}} \|f_{I}\|_{L^p(\R)}^2 \right)^{1/2}
  \end{equation}
  for functions $f_{I}$ with $\supp \widehat{f_{I}}\subset I.$
\end{theorem}

Recall that $\{a_n\}_{n=1}^{N^{1/2}}$  satisfies
\begin{equation}
  a_{i+1}-a_i\sim \frac{1}{N}, \qquad (a_{i+2}-a_{i+1})-(a_{i+1}-a_i) \sim \frac{\theta}{N^2}
\end{equation}
where here, the $\sim$ notation means within a factor of $4$. The parameter $\theta$ is in $(0,1].$ $\Omega$ is the $L^2\theta/N^2-$neighborhood of $\{a_n\}_{n=1}^{N^{1/2}},$ and
$$\Omega=\bigsqcup_{I\in \mathcal{I}} I,$$
where each $I$ {is an $L^2\theta/N^2$-neighborhood of $L$ consecutive terms in $\{a_n\}_{n=1}^{N^{1/2}}$. }

\subsection{Local decoupling and refined decoupling inequalities}

We first formulate a local decoupling inequality which implies (in fact is equivalent to) the global decoupling inequality \eqref{decoupling}.

\begin{proposition}\label{localdecsec} {Let $p\ge 2$. Suppose that for some $k\ge 100$, }
  \begin{equation}\label{localdecoupling}
    \|\sum_{I\in \mathcal{I}} f_{I}\|_{L^p(P(L))} \lesssim_\e  N^\e \log^C (\theta^{-1}+1)  \left( \sum_{I\in \mathcal{I}} \|f_{I}\|_{L^p(W_{P(L),k})}^2 \right)^{1/2}
  \end{equation}
{  holds for every $f_{I}$ with $\supp \widehat{f_{I}}\subset I.$ Then \eqref{decoupling} is true.
}
\end{proposition}
\begin{proof}
  Suppose \eqref{localdecoupling} holds for some $k\geq 100.$
  Since $\sum_{P(L)\subset \R} W_{P(L),k} \lesssim_k 1,$ and $p\geq 2,$ by Minkowski's inequality we have
  \begin{align*}
    \|\sum_{I} f_I\|_{L^p(\R)}^p & \leq \sum_{P(L)\subset \R} \int_{P(L)} |f|^p \\
     & \lesssim_\e N^\e \log^C (\theta^{-1}+1) \sum_{P(L)} (\sum_I \|f_I\|^{2}_{L^p(W_{P(L),k})} )^{p/2}\\
     & \lesssim N^\e \log^C (\theta^{-1}+1) (\sum_I \|f_I\|^{2}_{L^p(\sum_{P(L)} W_{P(L),k})} )^{p/2} \\
     & \lesssim N^\e \log^C (\theta^{-1}+1) (\sum_I \|f_I\|^{2}_{L^p(\R)} )^{p/2},
  \end{align*}
  which is \eqref{decoupling}.
\end{proof}

{The following local decoupling inequality will imply Theorem \ref{decouplingthm} by Proposition \ref{localdecsec}. }

\begin{theorem}[Local decoupling]\label{localdecouplingthm}
  Suppose $2\leq p\leq 6.$ Then
  \begin{equation}\label{localdecouplingeq}
    \|\sum_{I\in \mathcal{I}} f_{I}\|_{L^p(P(L))} \lesssim_{\e} N^\e \log^C (\theta^{-1}+1) \left( \sum_{I\in \mathcal{I}} \|f_{I}\|_{L^p(W_{P(L),100})}^2 \right)^{1/2}
  \end{equation}
  for $f_{I}$ with $\supp \widehat{f_{I}}\subset I.$
\end{theorem}

Theorem \ref{localdecouplingthm} is a consequence of the following refined decoupling theorem, which we focus on proving in the next two sections. The analogous result for the parabola can be found in \cite{guth2020falconer,demeter2020small}.

\begin{theorem}[Refined decoupling]\label{refineddecthm}
  Suppose $2\leq p\leq 6.$ For every $P(L)$ and every $X\subset P(L),$ we have
  \begin{equation}\label{refineddeceq}
    \|\sum_I f_I\|_{L^p(X)}\lesssim_\e N^\e \log^C(\theta^{-1}+1) (\sup_{x\in X}\sum_{I} \|f_I\|_{\stkout{L}^2(W_{P_{I}(x),100})}^2)^{1/2-1/p}(\sum_I \|f_I\|^2_{L^2(W_{P(L),100})})^{1/p}
  \end{equation}
  
  for $f_{I}$ with $\supp \widehat{f_{I}}\subset I.$
\end{theorem}

Now assuming Theorem \ref{refineddecthm} we show how it implies Theorem \ref{localdecouplingthm}.
\begin{proof}[Proof of Theorem \ref{localdecouplingthm} assuming Theorem \ref{refineddecthm}]
  Let $f=\sum_I f_I.$ Taking $X=P(L)$ in \eqref{refineddeceq} we see that
  $$\|f\|_{L^p(P(L))}\lesssim_\e N^\e \log^C(\theta^{-1}+1) (\sup_{x\in P(L)}\sum_{I} \|f_I\|_{\stkout{L}^2(W_{P_{I}(x),100})}^2)^{1/2-1/p}(\sum_I \|f_I\|^2_{L^2(W_{P(L),100})})^{1/p},$$
  To prove Theorem \ref{decouplingthm} we will do dyadic pigeonholing on the $L^2-$norm of wave packets of $f,$ using Proposition \ref{wpdecomprop}. More precisely we write
  $$f=\sum_I f_I=\sum_I \sum_{P_I} \phi_{P_{I}}f_I=\sum_{\lambda: \text{ dyadic}} \ \sum_{I,P_I:\|\phi_{P_I}f_I\|_{L^2(W_{P_I,100})}\in [\lambda/2,\lambda)} \phi_{P_I}f_I.$$
  Without loss of generality we assume $(\sum_I \|f_I\|_{L^p(W_{P(L),100})}^2)^{1/2}=1.$ Then
  $$\|\sum_{I,P_I: \|\phi_{P_I}f_I\|_{L^2(W_{P_I,100})}\notin [N^{-C}\theta^C , N^C\theta^{-C}] } \phi_{P_I} f_I \|_{L^p(P(L))} \lesssim 1$$
  for sufficiently large $C.$ Therefore there exists a $\lambda$ such that
  $$\|f\|_{L^{p}(P(L))}\lesssim C_\e N^{\e}{\log^C(\theta^{-1}+1)} \|\sum_{I,P_I: \|\phi_{P_I}f_I\|_{L^2(W_{P_I,100})}\in [\lambda/2,\lambda) } \phi_{P_I} f_I\|_{L^p(P(L))}+1.$$
  By a further dyadic pigeonholing argument on $I$, we may assume for every $I,$ either $\#\{P_I:\|\phi_{P_I}f_I\|_{L^2(W_{P(L),100})}\in [\lambda/2,\lambda)\}=0$ or $\#\{P_I:\|\phi_{P_I}f_I\|_{L^2(W_{P(L),100})}\in [\lambda/2,\lambda)\}\in [A/2, A)$ for some constant $A.$
  We denote by $\#I$ the number of $I$ such that $\#\{P_I:\|\phi_{P_I}f_I\|_{L^2(W_{P(L),100})}\in [\lambda/2,\lambda) \}\in [A/2, A).$
  For simplicity of notation we will also drop writing the condition $\|\phi_{P_I}f_I\|_{L^2(W_{P(L),100})}\in [\lambda/2,\lambda)$ in the summation.

  Now apply Theorem \ref{refineddecthm} to get
  \begin{multline} \label{889}
    \|\sum_{I,P_I} \phi_{P_I} f_I\|_{L^p(P(L))} \lesssim_\e \log^C(\theta^{-1}+1)N^\e (\sup_{x\in P(L)}\sum_{I}  \| \sum_{P_I} \phi_{P_I}f_I\|_{\stkout{L}^2(W_{P_{I}(x),100})}^2)^{1/2-1/p} \\
    (\sum_I \|\sum_{P_I} \phi_{P_I}f_I\|_{L^2(W_{P(L),100})}^2 )^{1/p}.
  \end{multline}
  To estimate the first factor on the right hand side of \eqref{889} we note that for every $x\in P(L),$
  $$\sum_{I}  \| \sum_{P_I} \phi_{P_I}f_I\|_{\stkout{L}^2(W_{P_{I}(x),100})}^2 \lesssim \sum_{I} \sum_{P_I} \|\phi_{P_I}f_I\|_{\stkout{L}^2(W_{P_{I}(x),100})}^2  \lesssim (\#I)\lambda^2|P_I|^{-1}$$
  because of $(\sum_{P_I} \phi_{P_I}(y))^2\lesssim \sup_{P_I}\phi_{P_I}^2(y)\leq \sum_{P_I} \phi_{P_I}^2(y)$ and \eqref{weighttransitive2}.
  Therefore
  $$\sup_{x\in P(L)} \sum_{I}  \|\sum_{P_I} \phi_{P_I}f_I\|_{\stkout{L}^2(W_{P_{I}(x),100})}^2  \lesssim (\#I)\lambda^2|P_I|^{-1}.$$
  To estimate the second factor on the right hand side of \eqref{889} we calculate
  $$\sum_I \|\sum_{P_I} \phi_{P_I}f_I\|_{L^2(W_{P(L),100})}^2  \lesssim \sum_I \sum_{P_I} \| \phi_{P_I}f_I\|_{L^2(W_{P(L),100})}^2 \lesssim (\#I) \lambda^2 A.$$
  To summarize, \eqref{889} implies that
  $$\|\sum_{I,P_I} \phi_{P_I} f_I\|_{L^p(P(L))} \lesssim_\e \log^C(\theta^{-1}+1)N^\e |P_I|^{1/p-1/2} (\#I)^{1/2}A^{1/p}\lambda.$$

  Now by H\"{o}lder's inequality we have
  \begin{align*}
    (\sum_I \|f_I\|_{L^p(W_{P(L),100})}^2)^{1/2} & \geq \left( \sum_I (\sum_{P_I} \|\phi_{P_I}^{1/2}f_I\|_{L^p(W_{P(L),100})}^p )^{2/p} \right)^{1/2} \\
     & \gtrsim \left( \sum_I (\sum_{P_I} \|\phi_{P_I}f_I\|_{L^2(W_{P(L),100})}^p |P_I|^{1-p/2} )^{2/p} \right)^{1/2} \\
     & \gtrsim |P_I|^{1/p-1/2} (\#I)^{1/2}A^{1/p} \lambda.
  \end{align*}

  Hence we have \eqref{localdecouplingeq}.
\end{proof}

\subsection{Induction scheme for proving Theorem \ref{refineddecthm}}

We fix $p,L$ and let $Dec(N,\theta)=Dec_{p}(N,L,\theta)$ denote the smallest constant such that
\begin{equation}\label{20000}
  \|\sum_I f_I\|_{L^p(X)}\leq Dec(N,\theta) (\sup_{x\in X}\sum_{I} \|f_I\|_{\stkout{L}^2(W_{P_{I}(x),100})}^2)^{1/2-1/p}(\sum_I \|f_I\|^2_{L^2(W_{P(L),100})})^{1/p}
\end{equation}
 holds for every sequence $\{a_n\}_{n=1}^{N^{1/2}}$ satisfying \eqref{defa_n}, every $P(L),$ every $X\subset P(L),$ and every $f_{I}$ with $\supp \widehat{f_{I}}\subset I.$ For a specific choice of the short generalized Dirichlet sequence $\{a_n\}_{n=1}^{N^{1/2}}$ satisfying \eqref{defa_n} we will call the smallest constant the refined decoupling constant of $\{a_n\}_{n=1}^{N^{1/2}}$ such that \eqref{20000} holds for every $X\subset P(L),$ and every $f_{I}$ with $\supp \widehat{f_{I}}\subset I.$
 Note that $Dec_p(N,L,\theta)$ is the supremum of all refined decoupling constants of sequences $\{a_n\}_{n=1}^{N^{1/2}}$ satisfying \eqref{defa_n}.

We will deduce Theorem \ref{refineddecthm}, which now is equivalent to $Dec(N,\theta)\lesssim_\e N^\e \log^C(\theta^{-1}+1),$ from the following main proposition.

\begin{proposition}\label{induction}
  For every $\e>0$ and every $1\leq K \leq N^{\e/2}$ satisfying $N^{1/2}/K\geq L,$ 
  \begin{equation}\label{inductioneqn}
    Dec(N,\theta) \lesssim_\e  \sup_{\theta'\in [\theta/4, \theta]}Dec(N/K^2,\theta'/K^2) + K^{D}N^\e \log^D(\theta^{-1}+1).
  \end{equation}
  Here $D$ is an absolute constant.
\end{proposition}

We postpone the proof of Proposition \ref{induction} to Section \ref{inductionsec}.
Here we show how it implies Theorem \ref{refineddecthm}.
\begin{proof}[Proof of Theorem \ref{refineddecthm} assuming Proposition \ref{induction}]
  For some sufficiently large $S_0$ we have $Dec(N,\theta)\leq C_s N^{s}\leq C_s N^s \log^D(\theta^{-1}+1)$ for $s\geq S_0.$
  Now suppose $Dec(N,\theta)\leq C_s N^s \log^D(\theta^{-1}+1) $ for some $s\leq S_0.$ Then from \eqref{inductioneqn} we have for every $\e>0$ and $K$ with $N^{1/2}/K\geq L,$
  \begin{align*}
    Dec(N,\theta) & \leq C_\e\left(\sup_{\theta'\in [\theta/4,\theta]} C_s (\frac{N}{K^2})^s \log^D({K^2}(\theta')^{-1}+1) + K^D N^\e \log^D(\theta^{-1}+1) \right). \\
     & \leq C_\e\left(CC_s (\frac{N}{K^2})^s \log^D({K^2}\theta^{-1}+1) + K^D N^\e \log^D(\theta^{-1}+1) \right) \\
     &
     \leq C_\e\left(CC_s (\frac{N}{K^2})^s (C\log^D(\theta^{-1}+1) + C\log^D(K^2)) + K^D N^\e \log^D(\theta^{-1}+1) \right).
  \end{align*}
  If we choose $\e$ to be $s/2$ and let $\frac{N^s}{K^{2s}}=K^DN^\e=K^DN^{s/2},$ that is, $K=N^{\frac{s}{2(2s+D)}},$ then for some constant $C'_s$ depending only on $s,$
  $$Dec(N,\theta) \leq C'_{s} N^{s\left(1- \frac{1}{2s+D}\right)} (\log^D(\theta^{-1}+1)+\log^D N )$$
  if $N^{1/2}N^{\frac{-s}{2(2s+D)}} \geq L.$ If $N^{1/2}N^{\frac{-s}{2(2s+D)}} \leq L,$ then $|\mc{I}|\lesssim N^{\frac{s}{2(2s+D)}}$ and by the triangle inequality and Cauchy-Schwarz inequality we have
  $$Dec(N,\theta)\lesssim N^{\frac{s}{2(2s+D)}}.$$
  We can assume that $D$ is large enough such that $\max\{2,S_0\}\leq D.$ Then $\frac{1}{2s+D}\sim D^{-1}$ and $K\leq N^{\e/2}$, so for some absolute constant $c>0$,
  $$Dec(N,\theta)\lesssim_s N^{s(1-c)} \log^D(\theta^{-1}+1). $$
  Conclude that
  $$Dec(N,\theta) \lesssim_\e N^{\e} \log^D(\theta^{-1}+1) $$
  for every $\e>0.$

\end{proof}

\subsection{Two applications}
Before ending this section, we record two applications of Theorem \ref{decouplingthm}. Technically these are corollaries of the $\ell^2 L^6$ decoupling inequality for the parabola in \cite{bourgain2015proof}, by deriving the corresponding $(\ell^2,L^6)$ estimate on short generalized Dirichlet polynomials using the method described in Section \ref{appsec}. 

First we may estimate approximate solutions to the equation $a_{n_1}+ a_{n_2}+a_{n_3}=a_{n_4}+a_{n_5}+a_{n_6}$ for a short generalized Dirichlet sequence $\{a_n\}_{n=1}^{N^{1/2}}.$ The number of exact solutions of such equations for general convex sequences was studied in \cite{iosevich2006combinatorial}.
\begin{corollary}\label{numberslncor}
  Let $\{a_n\}_{n=1}^{N^{1/2}}$ be a short generalized Dirichlet sequence with parameter $\theta.$ Then
  \begin{multline}\label{numbersln}
    \#\{(a_{n_1},\ldots,a_{n_6}): 1\leq n_i\leq N^{1/2}, |(a_{n_1}+ a_{n_2}+a_{n_3})-(a_{n_4}+a_{n_5}+a_{n_6})| \leq \theta/N^2 \} \\
     \lesssim_\e \log^C(\theta^{-1}+1) N^{3/2+\e}.
  \end{multline}
  This estimate is sharp up to $N^\e\log^C(\theta^{-1}+1)$ due to $N^{3/2}$ many diagonal solutions.
\end{corollary}
In particular if we take $a_n=\log (n+N+1)$ in the above corollary, then $\theta\sim 1$ and \eqref{numbersln} reads
\begin{equation}\label{tripleprod}
    \#\{(n_1,\ldots,n_6): N+1\leq n_i\leq N+ N^{1/2}, |n_1 n_2n_3-n_4n_5n_6| \lesssim N \}
     \lesssim_\e N^{3/2+\e}.
\end{equation}
We note that the triple products $n_1n_2n_3$ with $N+1 \leq n_1,n_2,n_3 \le N+N^{1/2}$ lies in the interval $[N^3,N^3+CN^{5/2}].$ So \eqref{tripleprod} impies that  the triple products $\{n_1n_2n_3:N+1 \leq n_1,n_2,n_3 \leq N+N^{1/2} \}$  are roughly evenly distributed in $[N^3,N^3+CN^{5/2}]$ with $cN$ separation. Indeed if we split the interval $[N^3,N^3+CN^{5/2}]$ into intervals of length $cN$ and let $E_\lambda$ denotes the number of $cN-$intervals which contains at least $\lambda$ many triple products $n_1n_2n_3,$ then \eqref{tripleprod} says that
\[\lambda^2 E_\lambda \leq C_\e N^{3/2+\e}.\]
Consequently if we choose $\lambda \geq 10 C_\e N^\e,$ then we have $\lambda E_\lambda \leq \frac{9}{10} N^{3/2}.$ $\lambda E_\lambda$ is the number of triple products $n_1n_2n_3$ that lie in a $cN-$interval which contains at least $\lambda$ many triple products. The total number of triple products is $N^{3/2}$ so we can conclude most of the triple products lie in $cN-$intervals each of which contains few triple products. 

\begin{proof}[Proof of Corollary \ref{numberslncor}]
  We let $\phi$ be a Schwartz function whose Fourier transform is given by a smooth bump function adapted to $B_{\theta/N^{2}}(0):$
  $$\hat{\phi}=1 \text{ on } B_{\theta/ N^{2}}(0), \quad \supp{\hat{\phi}}\subset B_{2\theta/ N^{2}}(0) \quad 0\leq \hat{\phi}\leq 1, \quad \hat{\phi} \text{ is even.}$$
  Applying Theorem \ref{decouplingthm} with $p=6,L=1$ we obtain
  \begin{align}
    \int_\R |\sum_{n=1}^{N^{1/2}} e^{ia_nx}\phi(x)|^6 & \lesssim_\e N^\e \log^C(\theta^{-1}+1) (\sum_{n=1}^{N^{1/2}}  \label{668} \|e^{ia_nx}\phi(x)\|_{L^6(\R)}^2 )^3 \\ \nonumber
     & \lesssim N^\e \log^C(\theta^{-1}+1) N^{3/2} \theta^5 N^{-10}.
  \end{align}
  We expand the left hand side of \eqref{668} as
  \begin{align*}
    \int_\R |\sum_{n=1}^{N^{1/2}} e^{ia_nx}\phi(x)|^6 dx & = \sum_{n_1,\ldots,n_6} \int_\R e^{i(a_{n_1}+ a_{n_2}+a_{n_3}-a_{n_4}-a_{n_5}-a_{n_6})x}|\phi|^6 dx\\
     & =\sum_{n_1,\ldots,n_6} \widehat{|\phi|^6}(a_{n_1}+ a_{n_2}+a_{n_3}-a_{n_4}-a_{n_5}-a_{n_6}).
  \end{align*}
  Since $\hat{\phi}$ is even we know that $\phi$ is real-valued and hence $\widehat{|\phi|^6}=\hat{\phi}*\cdots *\hat{\phi}$ is nonnegative and $\widehat{|\phi|^6} \gtrsim \theta^5 N^{-10}$ on $B_{c\theta /N^{2}}(0)$ for some small absolute constant $c>0$. Therefore 
  \begin{multline*}
    \int_\R |\sum_{n=1}^{N^{1/2}} e^{ia_nx}\phi(x)|^6 \gtrsim \theta^5 N^{-10}\#\{(a_{n_1},\ldots,a_{n_6}): 1\leq n_i\leq N^{1/2}, \\
    |(a_{n_1}+ a_{n_2}+a_{n_3})-(a_{n_4}+a_{n_5}+a_{n_6})|\leq \theta/N^2 \}.
  \end{multline*}
  Combining the above estimate and \eqref{668} we obtain \eqref{numbersln}.
\end{proof}

Another application of Theorem \ref{decouplingthm} is estimating the size of the intersection of an AP with a generalized Dirichlet sequence.
\begin{corollary}\label{APinconvex}
  Let $\{a_n\}_{n=1}^{N}$ be a generalized Dirichlet sequence with $\theta=1$ and let $a=N^{-\alpha}$ with $\alpha \in [0,2].$ Then 
  \[ |\{a_n\}_{n=1}^{n=N} \cap a\Z|
  \lesssim\begin{cases}
    N^{\alpha}  & \text{ if } \quad \alpha \in [0,\frac{1}{2}],  \\ 
   C_\e N^\e \log^C (\theta^{-1}+1) N^{1/3+\alpha/3}  & \text{ if } \quad \alpha \in [\frac{1}{2},2].
  \end{cases}
   \]
\end{corollary}
Corollary \ref{APinconvex} is sharp for $\alpha\in [0,\frac{1}{2}]$ (see Lemma \ref{longexample}), but it is likely not sharp for $\alpha \in [\frac{1}{2},2].$ Corollary \ref{APinconvex} has a slight connection to a conjecture of Rudin which states in a $N$-term AP we can find at most $\mathcal{O}(N^{1/2})$ many squares (numbers of the form $n^2$ for some $n\in \Z$). The best result so far seems to be in
\cite{bombieri2002note}, which proves at most $\mathcal{O}(N^{3/5}\log^{\mathcal{O}(1)}N)$ many squares can be found in a $N$-term AP.  We note that $\{\frac{n^2}{N^2}\}_{n=N+1}^{2N}$ is a generalized Dirichlet sequence. However we shall not expect to solve Rudin's conjecture exploiting only the convexity of the sequence $\{n^2:n\in \N\},$ as shown by the example  given in Lemma \ref{longexample}.

\begin{proof}[Proof of Corollary \ref{APinconvex}]

  The case $\alpha\in [0,1/2]$ is trivial as $\{a_n\}_{n=1}^{N}$ is contained in a ball of radius $\lesssim 1$ and $a\Z$ has at most $\lesssim a^{-1}=N^{\alpha}$ many terms in such a ball. 
  
  Now we suppose $\alpha\in [1/2,2].$ It suffices to show that for a short generalized Dirichlet sequence $\{a_n\}_{n=1}^{N^{1/2}},$  $H:=|\{n: 1\leq n \leq N^{1/2},\, a_n \in a\Z \}|
  $ satisfies 
  \[H \lesssim_\e C_\e  \log^C(\theta^{-1}+1) N^{\frac{\alpha}{3}-\frac{1}{6}+\e}.\] 
      We consider the function \[f(x)=\sum_{n:1\leq n \leq N^{1/2}, \, a_n\in a\Z} e^{{{2\pi}}ita_n}.\]
  
  {\bf Case 1. $\alpha\in [1,2].$} We apply Theorem \ref{localdecouplingthm} with $p=6,$ $L=1$ and $P(L)=P(L,0).$ Since $|f|\geq H/10$ on $\mathcal{N}_{cN^{1/2}} (a^{-1}\Z)$ with $c\gtrsim 1,$ we obtain
  \[H \left( \frac{N^2\theta^{-1}}{N^\alpha} N^{\frac{1}{2}}\right)^{\frac{1}{6}} \lesssim_\e N^\e \log^C(\theta^{-1}+1) H^{\frac{1}{2}} (N^2\theta^{-1})^{\frac{1}{6}}, \]
  {where we used that $P(L)$ is approximately an $N^2\theta^{-1}$ interval. Simplifying the above displayed math, we have}
  \[H \lesssim_\e C_\e  \log^C(\theta^{-1}+1) N^{\frac{\alpha}{3}-\frac{1}{6}+\e}.\]
  
  {\bf Case 2. $\alpha \in [1/2,1].$} We apply Theorem \ref{localdecouplingthm} with $p=6,$ $L=N^{1-\alpha}$ and $P(L)=P(L,0).$ Since $|f|\geq H/10$ on $\mathcal{N}_{cN^{1/2}} (a^{-1}\Z)$ with $c\gtrsim 1,$ we obtain
  \[H \left( \frac{N^{2\alpha}2\theta^{-1}}{N^\alpha} N^{\frac{1}{2}}\right)^{\frac{1}{6}} \lesssim_\e N^\e \log^C(\theta^{-1}+1) H^{\frac{1}{2}} (N^{2\alpha}2\theta^{-1})^{\frac{1}{6}}, \]
  that is,
  \[H \lesssim_\e C_\e  \log^C(\theta^{-1}+1) N^{\frac{\alpha}{3}-\frac{1}{6}+\e}.\]
\end{proof}

\section{High-low frequency decomposition for the square function} \label{highlowsec}

The proof of Proposition \ref{induction} is based on the method in \cite{guth2020improved}, which uses a high-low frequency decomposition for the square function. Such a decomposition is also used in \cite{guth2019incidence} to study incidence estimates for tubes.  We refer readers to  Section 2 of \cite{guth2020improved} for the intuition behind this method. We will set up the preliminaries in this section and prove Proposition \ref{induction} in Section \ref{inductionsec}.

\subsection{Wave-packet decomposition}

We start with a few definitions. Write $f=\sum\limits_{I\in \mc{I}} f_I,$ where $f_I$ will always denote a function with frequency support in $I.$

Fix $2\leq p\leq 6$ and $\e>0.$  For $m\in\N$, let $L_m=N^{1/2}N^{-\e m}.$ Without loss of generality we assume $L_M=L$ for some $M\in \N.$ So $M\lesssim_\e 1.$ For every $1\leq m\leq M$ we let $\mathcal{I}_m$ be the partition of $\Omega$ into $N^{1/2}/L_m$ many $I_m,$ each of which is the union of $L_m-$consecutive intervals in $\Omega.$ $L_m$ can be thought of as scales.

Note that
$$I_m\subset P_{v_m}^{CL_m^2\theta/N^2} \cap B_{CL_m/N}$$
{where $v_m\sim\frac{1}{N}$}. We denote the right hand side as $\tilde{I}_m:$
$$\tilde{I}_m:=P_{v_m}^{CL_m^2\theta/N^2} \cap B_{CL_m/N}.$$


Let $\mc{P}_{I_m}$ be a tiling of $\R$ by $P_{I_m}.$
For each $I_m$, we will now construct a partition of unity $\{\phi_{I_m}\}_{P_{I_m}\in\mc{P}_{I_m}}$  which will be used to perform the wave packet decomposition \[f_{I_m}=\sum_{P_{I_m}}\phi_{P_{I_m}}f_{I_m}. \]
We regard each summand $\phi_{P_{I_m}}f_{I_m}$ as a wave packet. Specifically,
we let $\psi_{I_m}$ be adapted to $\tilde{I}_m- \tilde{I}_m,$ which is of the form $P_{v_0}^{CL_m^2\theta/N^2}(0) \cap B_{CL_m/N}(0),$ in the frequency space as in Lemma \ref{psilem}, with order of decay $200$ outside of the dual fat AP $P_{I_m}.$
For each $P_{I_m}\in\mc{P}_{I_m}$, define
\begin{equation}\label{wpdef}
    \phi_{P_{I_m}}:=\|\s_{I_m}^2\|_{L^1(\R)}^{-1}\int_{P_{I_m}}|\s_{I_m}(x-y)|^2dy.
\end{equation}






\begin{proposition}[Wave-packet decomposition] \label{wpdecomprop}
   $\{\phi_{P_{I_m}}\}_{P_{I_m} \in \mc{P}_{L_m}}$ forms a partition of unity, that is, $\sum \phi_{P_{I_m}}=1,$ $\phi_{P_{I_m}}\geq 0.$ Each $\phi_{P_{I_m}}$ is a translated copy of the others,  and
  $$\supp \widehat{\phi_{P_{I_m}}} \subset 8^{400}( \tilde{I}_m-\tilde{I}_m) ,\qquad 1_{P_{I_m}} \lesssim {\phi_{P_{I_m}}} \lesssim W_{P_{I_m},200}.$$
\end{proposition}
\begin{proof}
  By definition we see that $\phi_{P_{I_m}}$ forms a partition of unity, and each $\phi_{P_{I_m}}$ is a translated copy of the others. Also it follows from the definition that
  $$1_{P_{I_m}} \lesssim |{\phi_{P_{I_m}}}|.$$

  Note that $\phi_{P(L_m)}$ equals to $\|\psi_{I_m}^2\|_{L^1(\R)}^{-1} |\psi_{I_m}|^2*1_{P_{I_m}}.$
  Therefore $\psi_{I_m}$ decays at order $200$ outside $P_{I_m}(0)$ implies that $\phi_{P(L_m)}$ decays at order $400$ outside $P_{I_m},$ and in particular
  $$|{\phi_{P_{I_m}}}|\lesssim W_{P_{I_m},200}.$$
  {The support property} $\text{supp }\widehat{\phi_{P_{I_m}}} \subset 8^{400} ( \tilde{I}_m-\tilde{I}_m )$ follows from the fact that
  $$\widehat{\phi_{P_{I_m}}}=\|\psi_{I_m}^2\|_{L^1(\R)}^{-1} \widehat{|\psi_{I_m}|^2} \widehat{1_{P_{I_m}}}$$
  {and from Lemma \ref{psilem}}. 
\end{proof}

\subsection{A pruning process and modified square functions \label{pruning}}

Now we define ``square functions'' (squared) at scales $L_m,$ which differ from the usual square functions by a pruning process of wave packets and taking spatial averages. The pruning process will depend on two parameters $\alpha$ and $r,$ which can be thought of as the values of $|f|$ and $\sum_{I_M}|f_{I_M}|^2=\sum_{I}|f_{I}|^2$ which dominate the $L^p$ norm of $f.$ We define
$\lambda=\lambda(\alpha,r)$ by
\begin{equation}\label{lambdadefn}
   \lambda=  \tilde{C}_\e N^{\e}\frac{r}{\alpha}
\end{equation}
where $\tilde{C}_\e$ is a sufficiently large constant depending on $\e$ which will be chosen later in the proof of Lemma \ref{fmcomparablelem}.

We first do the pruning process (with parameters $\alpha,r$), which inductively removes wave packets at each scale whose height exceeds $\lambda.$ As we shall see (Lemma \ref{fmcomparablelem}), those wave packets do not play a dominant role in the $L^p$ norm of $f.$ This process produces a family of functions $f_{m,I_m},f_{m,I_{m-1}},f_m$ that depend on $\alpha,r,$ which is implicit in the notation. We will write $f_{m,I_m,\alpha,r},f_{m,I_{m-1},\alpha,r},f_{m,\alpha,r}$ to emphasize such dependence when necessary.

Let $\mc{P}_{I_M,\lambda}=\{P_{I_M}\in\mc{P}_{L_M}:\|\phi_{P_{I_M}}f_{I_M}\|_{L^\infty(\R)}\le\lambda\},$
and define
\begin{align*} 
  f_{M,I_M}:=\sum_{P_{I_M}\in\mc{P}_{I_M,\lambda}}\phi_{P_{I_m}}f_{I_M}, \qquad
  f_M :=\sum_{I_M} f_{M,I_M}.
\end{align*} 
We let $f_{M,I_{M-1}}=\sum_{I_{M}\subset I_{M-1}}f_{M,I_M}.$
Now we define $f_m$ and $f_{m,I_m}$ inductively for $m=1,\ldots,M-1$ by
\begin{equation}\label{1333}
    f_{m,I_m}:=\sum_{P_{I_m}\in \mc{P}_{I_m,\lambda}}\phi_{P_{I_m}}f_{m+1,I_m}, \qquad 
f_m:=\sum_{I_m}f_{m,I_m}
\end{equation}
where $f_{m+1,I_{m}}= \sum_{I_{m+1}\subset I_{m}} f_{m+1,I_{m+1}}$ and $\mc{P}_{I_m,\lambda}=\{P_{I_m}\in \mc{P}_{I_m}:\|\phi_{P_{I_m}}f_{{m+1},I_m}\|_{L^\infty(\R)}\le\lambda\}.$ For notational convenience we also define $f_{M+1}=f$ and $f_{M+1,I_M}:=f_{I_M}=f_{I}.$

We note that
\begin{enumerate}[label=(\roman*)]
  \item \label{item1} $f_m=\sum_{I_m} f_{m,I_m}=\sum_{I_{m-1}} f_{m,I_{m-1}},$
  \item \label{item2} $\supp \widehat{f_{m,I_m}} \subset C \tilde{I}_m,$
  \item \label{item3} $\supp \widehat{f_{m,I_{m-1}}}\subset C \tilde{I}_{m-1},$
  \item \label{item5}  $|f_{m,I_m}|\leq |f_{m+1,I_{m}}|$ pointwise.
\end{enumerate}
\ref{item1} follows from the definitions. \ref{item5} holds because $\{\phi_{P_{I_m}}\}_{P_{I_m}}$ is a partition of unity. 
To see \ref{item2} and \ref{item3} we may induct on $m$ and note that
$$\bigcup_{I_m\subset I_{m-1}} C\tilde{I}_m \subset 2 \tilde{I}_{m-1} $$
when $N$ is sufficiently large depending on $\e.$ 

To define the ``square function'' $g_m$ at scale $L_m$ we introduce $\rho_{I_m},$ which is an $L^1-$normalized non-negative function adapted to $P_{I_m}(0)$ with decay order $100$
$$|P_{I_m}|^{-1}1_{P_{I_m(0)}}(x)\lesssim\rho_{I_m}(x)\lesssim
 \frac{W_{P_{I_m}(0),100}(x)}{\|W_{P_{I_m}(0),100}\|_{L^1(\R)}},$$
and $\supp \widehat{\rho_{I_m}}\subset C(\tilde{I}_m-\tilde{I}_m).$ Such a function can be constructed by taking  $|\psi|^2/\|\psi^2\|_{L^1}$ for $\psi$ adapted to $\tilde{I}_m$ with decay order $100$ as in Lemma \ref{psilem}.

Finally we define the ``square function'' by
$$g_m:=\sum_{I_m}|f_{m+1,I_m}|^2*\rho_{I_m},$$
for $1\leq m\leq M-1$ and for $m=M$
we define
$$ g_M:=\sum_{I_M}|f_{I_M}|^2*\rho_{I_M}.$$
We note here that $g_m$ for $1\leq m \leq M-1$ implicitly depends on $\alpha,r,$ and we will write $g_{m,\alpha,r}$ to emphasize such dependence when necessary. $g_M$ does not depend on $\alpha,r.$

\subsection{High-low decomposition}

To set up a high-low frequency decomposition for $g_m,$ we let $\eta_m(\xi)$ be an even smooth bump function that equals to $1$ on $B_{L_{m+1}/N}(0)$ and vanishes outside $B_{2L_{m+1}/N}(0),$ for every $1\leq m \leq M-1.$ We also assume that $\eta_m$ are rescalings of each other.

Define for $1\leq m\leq M-1,$
$$ g_m^{\ell}:=g_m*\widecheck{\eta_m}\quad\text{and}\quad g_m^h:=g_m-g_m^{\ell},$$
which are low and high frequency parts of $g_m.$ Both $g_m^\ell$ and $g_{m}^h$ satisfy some proprieties. We discuss them in the following two lemmas.

\begin{lemma}[Low lemma]\label{low}
For $1\leq m\leq M-1,$ we have the pointwise inequality
$$|g_m^{\ell}| \lesssim g_{m+1}.$$
\end{lemma}

\begin{proof}
  By definition
  $$g_m^\ell =(\sum_{I_m} |f_{m+1,I_m}|^2) * \rho_{I_m} * \widecheck{\eta_m}=(\sum_{I_m} |f_{m+1,I_m}|^2) *  \widecheck{\eta_m}*\rho_{I_m}.$$
  Using Plancherel's theorem,
  \begin{align}
    |f_{m+1,I_m}|^2*\widecheck{\eta_m}(x) & =\int |f_{m+1,I_m}(y)|^2 \widecheck{\eta_m}(x-y) dy \nonumber \\
     & =\int (\widehat{f_{m+1,I_m}} * \widehat{\overline{f_{m+1,I_m}}})(\xi) e^{2\pi ix\xi}\eta_m(\xi)d\xi \nonumber \\
     & =\sum_{I_{m+1},I'_{m+1} \subset I_m } \int (\widehat{f_{m+1,I_{m+1}}} * \widehat{\overline{f_{m+1,I'_{m+1}}}})(\xi) e^{2\pi ix\xi}\eta_m(\xi)d\xi. \label{110}
  \end{align}
  We note that $\widehat{f_{m+1,I_{m+1}}} * \widehat{\overline{f_{m+1,I_{m+1}'}}}$ is supported in $C\widetilde{I}_{m+1}-C\widetilde{I}_{m+1}'$ and $\widetilde{I}_{m+1}$ is of the form $P^{CL^2\theta/N^2}_{v_{I_{m+1}}}\cap B_{CL_{m+1}/N}.$
  Since $\eta_m$ is supported on $B_{2L_{m+1}/N}(0)$ we conclude that for every fixed $I_{m+1}$ there are only $\mathcal{O}(1)$ many $I'_{m+1}$ such that the integral in \eqref{110} is nonzero, and for those $I_{m+1}'$ we write $I_{m+1}'\sim I_{m+1}.$
  We let $\psi_{I_{m+1}}$ be adapted to $C(\tilde{I}_{m+1}-\tilde{I}_{m+1})$ as in Lemma \ref{psilem} with order of decay $200.$ Then, using Cauchy-Schwartz in the first two inequalities, we have
\begin{align*}
    ||f_{m+1,I_m}|^2*\widecheck{\eta_m}(x)| & =\sum_{I_{m+1}\subset I_m} \sum_{I_{m+1}'\sim I_{m+1}} f_{m+1,I_{m+1}}\overline{f_{m+1,I_{m+1}'}} * \widecheck{\eta_m} \\ & \le \sum_{I_{m+1}\subset I_m} \sum_{I_{m+1}'\sim I_{m+1}} (|f_{m+1,I_{m+1}}|^2*|\widecheck{\eta_m}|)^{1/2}(|f_{m+1,I_{m+1}'}|^2*|\widecheck{\eta_m}|)^{1/2}  \\
    & \lesssim  \sum_{I_{m+1}\subset I_m } |f_{m+1,I_{m+1}}|^2*|\widecheck{\eta_m}| \\
    & \lesssim \sum_{I_{m+1}\subset I_m } |f_{m+1,I_{m+1}}|^2*|\widecheck{\psi_{I_{m+1}}}|*|\widecheck{\eta_m}| \\
     & \leq \sum_{I_{m+1} \subset I_m}| f_{m+2,I_{m+1}}|^2*|\widecheck{\psi_{I_{m+1}}}|*|\widecheck{\eta_m}|
  \end{align*}   
  where the last inequality is because of $|f_{m+1,I_{m+1}}|\leq |f_{m+2,I_{m+1}}|$ pointwise.
  Now to finish the proof, it suffices to observe that
  $$ |\widecheck{\eta_m}|*|\widecheck{\psi_{I_{m+1}}}|*\rho_{I_m} \lesssim \rho_{I_{m+1}},$$
  since $|\widecheck{\eta_m}|$ decays rapidly outside $B_{N/L_{m+1}}(0),$  $|\widecheck{\psi_{I_{m+1}}}|$ decays at order $200$ outside $P_{I_{m+1}}(0),$ $\rho_{I_m}$ decays at order $100$ outside $P_{I_m}(0)$, and $B_{L_{m+1}/N}(0)+P_{I_m}(0)\subset CP_{I_{m+1}}(0).$
\end{proof}

Recall that
$$P(L_m)=P_{v_1^{-1}}^{CN^{3/2}/L_m^2} \cap B_{CN^2/(L_m^2\theta)}$$
(which degenerates to $B_{CN^2/(L_m^2\theta)}$ if $L_m\leq CN^{1/4}$) as defined in \eqref{defnLball}. Let $\phi_{P(L_M)}$ be a function such that
\begin{align*} 
{\text{supp }}\widehat{\phi_{P(L_M)}} &\subset P_{v_1}^{CL_M^2\theta/N^2}(0) \cap B_{CL_M^2/N^{3/2}}(0) \subset \bigcap_{I\in\mc{I}} (\tilde{I}-\tilde{I}), \qquad \\
&1_{P(L_M)} \lesssim |{\phi_{P(L_M)}}| \lesssim W_{P(L_M),200}.
\end{align*} 
To construct such a function we can take a $\psi$ in Lemma \ref{psilem} adapted to certain fat AP and apply a translation in the physical space to it.




\begin{lemma}[High lemma] \label{hilem}
For $1\leq m\leq M-1$ we have
$$\int |g_m^h|^2 W_{P(L_M),100} \lesssim N^\e \int \sum_{I_m}|f_{m+1,I_m}|^4 W_{P(L_M),100}.$$
\end{lemma}

\begin{proof} 
{Because of \eqref{weighttransitive2}, it suffices to show for every $P(L_M),$
\[\int_{P(L_M)} |g_m^h|^2 \lesssim N^\e \int \sum_{I_m}|f_{m+1,I_m}|^4 W_{P(L_M),100}.\]}

Calculate
  \begin{align*}
    \int |g_m^h|^2 W_{P(L_M),100} & \lesssim \int |g_m^h \phi_{P(L_M)}|^2  \\
    & = \int |\sum_{I_m} {(|f_{m+1,I_m}|^2)\widehat{\, } \, \, \widehat{\rho_{I_m}}}(1-\eta_m)*\widehat{\phi_{P(L_M)}}|^2.
  \end{align*}


  Note that
  $$\supp \left({(|f_{m+1,I_m}|^2)\widehat{\, } \, \,\widehat{\rho_{I_m}}}(1-\eta_m)*\widehat{\phi_{P(L_M)}} \right) \subset C(\tilde{I}_m-\tilde{I}_m) \setminus B_{L_{m+1}/{(2}N)}(0) .$$
  {Indeed, the high frequency cutoff $(1-\eta_m)$ removes the ball $B_{L_{m+1}/N}(0)$. The support of $\widehat{\phi_{P(L_M)}}$ is contained in a ball of radius $\le \frac{1}{2}L_M^2/N^{3/2}$ (if the $C$ in the definition of $P(L)$ as in \eqref{defnLball} is large enough), so convolution with $\widehat{\phi_{P(L_M)}}$ shrinks the high frequency cutoff by an amount smaller than $L_{m+1}/(2N)$. The structure of $\tilde{I}_m-\tilde{I}_m$ is unchanged by convolution by $\widehat{\phi_{P(L_M)}}$ because the thickness of $\tilde{I}_m$ is $\sim {L_m}/{N}$ and $\frac{1}{2}L_M^2/N^{3/2}\le L_{m+1}/(2N)\le N^{-\e} L_{m}/N$. }
  
  We claim that at every point on $\R,$ the collection of sets $\{C(\tilde{I}_m-\tilde{I}_m) \setminus B_{L_{m+1}/(2N)}(0)\}_{I_m}$ has at most $\mathcal{O}(N^\e)$ overlap. Assuming this claim, by the Cauchy-Schwarz inequality  we obtain
  \begin{align*}
    \int |g_m^h|^2 W_{P(L_M),100} & \lesssim N^\e \int \sum_{I_m}\left| (|f_{m+1,I_m}|^2)\widehat{\, } \, \,\widehat{\rho_{I_m}}(1-\eta_m)*\widehat{\phi_{P(L_M)}}\right|^2.
  \end{align*}
  So we have
  \begin{align*}
    \int |g_m^h|^2& W_{P(L_M),100}  \lesssim N^\e \sum_{I_m} \int \left| |f_{m+1,I_m}|^2*\rho_{I_m}*\widecheck{(1-\eta_m)} \right|^2 | \phi_{P(L_M)}|^2 \\
    & \lesssim N^\e \sum_{I_m} \left(\int \left| |f_{m+1,I_m}|^2*\rho_{I_m} \right|^2 | \phi_{P(L_M)}|^2 + \int \left| |f_{m+1,I_m}|^2*\rho_{I_m}*|\widecheck{\eta_m}| \right|^2 | \phi_{P(L_M)}|^2\right)  \\
     & \lesssim N^\e \sum_{I_m} \left(\int  |f_{m+1,I_m}|^4 (| \phi_{P(L_M)}|^2 *\rho_{I_m}) + \int |f_{m+1,I_m}|^4 (| \phi_{P(L_M)}|^2*\rho_{I_m}*|\widecheck{\eta_m}|) \right),
  \end{align*}
  where we{ used Cauchy-Schwartz and that $\rho_{I_m}$ and $\widecheck{\eta_m}$ have $L^1$ norms $\sim 1$ to justify
  \[ \left||f_{m+1,I_m}|^2*\rho_{I_m} \right|^2\lesssim |f_{m+1,I_m}|^4*\rho_{I_m} ,\quad \left||f_{m+1,I_m}|^2*\rho_{I_m}*|\widecheck{\eta_m}| \right|^2\lesssim |f_{m+1,I_m}|^4*\rho_{I_m}*|\widecheck{\eta_m}|.\]} Noting that
  $|\phi_{P(L_M)}|^2 *\rho_{I_m} \lesssim W_{P(L_M),100}$ and $|\phi_{P(L_M)}|^2*\rho_{I_m}*|\widecheck{\eta_m}| \lesssim W_{P(L_M),100},$ conclude 
  $$\int |g_m^h|^2 W_{P(L_M),100} \lesssim N^\e \sum_{I_m} \int |f_{m+1,I_m}|^4 W_{P(L_M),100}. $$

  Now we prove the claim. Recall that $\tilde{I}_m$ is a fat AP of the form $P^{CL^2\theta/N^2}_{v_{I_m}} \cap B_{CL_m/N}$ where $v_{I_m}\sim N^{-1}.$
  Suppose $x\in C(\tilde{I}_m-\tilde{I}_m) \setminus B_{L_{m+1}/(2N)}(0)$ and $x\in C(\tilde{I}'_{m}-\tilde{I}'_{m}) \setminus B_{L_{m+1}/N}(0)$ for distinct $\tilde{I}_m$ and $\tilde{I}'_m.$
  We denote the common difference of $\tilde{I}_m$ and $\tilde{I}_{m}'$ by $v$ and $v'$ respectively. Recalling that $v_{I_m}$ are $C\theta L_m/N^2$ separated, and the maximal separation is $C (N^{1/2}/L_m) (\theta L_m/N^2)=C\theta/N^{3/2},$ we have
  $$\theta L_m/N^2 \lesssim |v-v'| \lesssim \theta/ N^{3/2}.$$
  Suppose $x\in B_{CL_m^2\theta/N^2}(kv)$ and $x\in B_{CL_m^2\theta/N^2}(k'v')$ for some $k,k'\in \N.$ Then {since $x\not\in B_{L_{m+1}/(2N)}(0)$,} $L_{m+1} \lesssim k , k' \lesssim L_m.$ By definition $L_{m}= N^\e L_{m+1} \leq N^{1/2-\e},$ so we have
  $$L_{m+1} \frac{\theta L_m}{N^2} {\gtrsim} N^{-\e}\frac{\theta L_m^2}{N^2} , \qquad L_m \frac{\theta}{N^{3/2}} \leq \frac{\theta}{N^{1+\e}}\leq \frac{1}{N^{1+\e}}.$$
  It follows that $|k-k'|\lesssim 1$ and
  $$\text{either } |v-v'|\lesssim N^\e \theta L_m/N^2 \quad \text{ or } \quad |v-v'|\gtrsim 1/N^{3/2-\e}.$$
  The second case cannot happen if $N$ is sufficiently large (depending on $\e$).
  Since common differences $v$ are $\mathcal{O} (\theta L_m/N^2)-$separated, we conclude that there are at most $\mathcal{O}(N^\e)$ many $\tilde{I}_m'$ such that $x\in C(\tilde{I}'_{m}-\tilde{I}'_{m}) \setminus B_{L_{m+1}/(2N)}(0).$
\end{proof}

\subsection{The sets $\Omega_{m,\alpha,r}$ and $U_{\alpha,r}$}

The last part of our high-low decomposition set-up is to partition $P(L_M)$ into $\Omega_{m,\alpha,r},$ for a fixed pair $(\alpha,r).$
For $1\leq m\leq M-1$ we define $\Omega_{m,\alpha,r}$ to be
$$\Omega_{m,\alpha,r}:=\{x\in P(L_M): g_m(x) \leq 2|g_m^h(x)| , g_{m+1}(x)\leq 2|g_{m+1}^\ell(x)|,\ldots, g_{M-1}(x)\leq 2|g_{M-1}^\ell(x)| \}.$$
Here $g_{k}=g_{k,\alpha,r}.$
Also define $\Omega_{0,\alpha,r}$ to be
$$\Omega_{0,\alpha,r}:=\{x\in P(L_M): g_1(x) \leq 2|g_1^\ell(x)| , g_{2}(x)\leq 2|g_{2}^\ell(x)|,\ldots, g_{M-1}(x)\leq 2|g_{M-1}^\ell(x)| \}.$$

Clearly
$$P(L_M)=\bigcup_{0\leq m\leq M-1}\Omega_{m,\alpha,r}$$
for every $\alpha,r.$
For notational convenience we let $\Omega_{M,\alpha,r}=P(L_M).$

We define $U_{\alpha', r'}$ by
\begin{equation}\label{defnUalphar}
  U_{\alpha', r'} := \{x\in P(L_M):r'/2< g_{M}(x)\leq 2r',\quad \alpha'/2< |f(x)|\leq 2\alpha'\}.
\end{equation}
Recall that $g_M=\sum_{I_M}|f_{I_M}|^2*\rho_{I_M}$ is defined without the pruning process so in particular it does not depend on the pruning parameters $\alpha,r.$ 




We prove the following lemma, which shows that on $U_{\alpha,r}\cap \Omega_{m,\alpha,r},$ $|f_m-f_{m,\alpha,r}|$ is very small so that $|f_m|\sim |f_{m,\alpha,r}|.$  We define $f_0=f_1$ for notational convenience. Also recall we have defined $f_{M+1}=f$ and $f_{M+1,I_M}=f_{I_M}=f_{I}.$ 
\begin{lemma} \label{fmcomparablelem} If the constant $\tilde{C}_\e$ in the definition of $\lambda$ is large enough depending on $\e,$ then for every $\alpha, r,$ every  $1\leq m \leq M-1,$ and any subset $\mc{S}$ of the partition $\mc{I}_m=\{I_m\}$, we have
\[|\sum_{I_m\in\mc{S}}f_{I_m}-\sum_{I_m\in\mc{S}}f_{m,\a,r,I_m}|\le \frac{\a}{100}\]
on $U_{\alpha,r}\cap \Omega_{m,\alpha,r}$, and also on $U_{\alpha,r}\cap \Omega_{0,\alpha,r}$ if $m=1$. In particular if $\tilde{C}_\e$ in the definition of $\lambda$ is large enough depending on $\e,$ then for every $\alpha, r,$ every  $0\leq m \leq M-1,$
$$|f_{m,\alpha,r}|\in [\alpha/4,4\alpha],$$
on $U_{\alpha,r}\cap \Omega_{m,\alpha,r}$.
\end{lemma}

\begin{proof} Fix $\alpha,r.$ In the following proof $g_{k}$ means $g_{k,\alpha,r},$ and $f_{k,I_k},f_{k,I_{k-1}},f_{k}$ means $f_{k,I_k,\alpha,r},f_{k,I_{k-1},\alpha,r},f_{k,\alpha,r}$ respectively. 
First suppose $1\leq m\leq M-1.$
 By definition of $\Omega_{m,\alpha,r}$ and Lemma \ref{low} we know that on $U_{\alpha,r} \cap \Omega_{m,\alpha,r},$
  $$g_{m+1}\lesssim g_{m+2} \lesssim \cdots \lesssim g_{M} \lesssim r.$$
  We also have by the Cauchy-Schwarz inequality $g_{m}\lesssim_\e N^{\e}g_{m+1}.$
  Recall that $M\lesssim_\e 1$ so we have for $m\leq k\leq M,$
  $$g_{k} \lesssim_\e N^{\e} r \text{ on } U_{\alpha,r} \cap \Omega_{m,\alpha,r}.$$
  Let $m'$ be an integer between $m$ and $M$ and let $I_{m'}\in \mc{I}_{m'}.$ By the definition of $f_{m',I_{m'}}$ and $f_{m'+1,I_{m'}}$ we have for $x\in U_{\alpha,r}\cap \Omega_{m,\alpha,r}$
  \begin{align*}
     |f_{m',I_{m'}}(x)-f_{m'+1,I_{m'}}(x)|& =|\sum_{P_{I_{m'}}\notin \mathcal{P}_{I_{m'},\lambda}}\phi_{P_{I_{m'}}}(x)f_{m'+1,I_{m'}}(x)| \\
     & \lesssim \sum_{P_{I_{m'}}\notin \mathcal{P}_{I_{m'},\lambda}} |\phi_{P_{I_{m'}}}^{1/2}(x)f_{m'+1,I_{m'}}(x)|\phi_{P_{I_{m'}}}^{1/2}(x) \\
     & \lesssim  \sum_{P_{I_{m'}}\notin \mathcal{P}_{I_{m'},\lambda}}  \lambda^{-1} \| \phi_{P_{I_{m'}}}f_{m'+1,I_{m'}} \|_{L^\infty(\R)} \| \phi_{P_{I_{m'}}}^{1/2}f_{m'+1,I_{m'}} \|_{L^\infty(\R)}  \phi_{P_{I_{m'}}}^{1/2}(x) \\
     & \lesssim \lambda^{-1}  \sum_{P_{I_{m'}}\notin \mathcal{P}_{I_{m'},\lambda}} \| \phi_{P_{I_{m'}}}^{1/2}f_{m'+1,I_{m'}} \|_{L^\infty(\R)}^2 \phi_{P_{I_{m'}}}^{1/2}(x) \\
     & \lesssim \lambda^{-1}  \sum_{P_{I_{m'}}\notin \mathcal{P}_{I_{m'},\lambda}}
      \sum_{\tilde{P}_{I_{m'}}} \| \phi_{P_{I_{m'}}}f^2_{m'+1,I_{m'}} \|_{L^\infty(\tilde{P}_{I_{m'}})} \phi_{P_{I_{m'}}}^{1/2}(x) \\
     & {\lesssim \lambda^{-1}  \sum_{P_{I_{m'}}}
      \sum_{\tilde{P}_{I_{m'}}} \| \phi_{P_{I_{m'}}}\|_{L^\infty(\tilde{P}_{I_{m'}})}\|f^2_{m'+1,I_{m'}} \|_{\stkout{L}^1(W_{\tilde{P}_{I_{m'}}})} \phi_{P_{I_{m'}}}^{1/2}(x) }\\
    \end{align*}
    where we used that $\phi_{P_{I_{m'}}}\lesssim \phi^{1/2}_{P_{I_{m'}}}$. We also used the locally constant property Proposition \ref{locconstprop} for the last inequality. If we use $\phi_{I_{m'}}(\tilde{P}_{I_{m'}})$ to denote $\phi_{I_{m'}}(\sup \tilde{P}_{I_{m'}}),$ which is comparable to $\phi_{I_{m'}}(y)$ for any $y\in \tilde{P}_{I_{m'}},$ then we have
   \begin{align*}
    |f_{m',I_{m'}}(x)-f_{m'+1,I_{m'}}(x)|  & \lesssim \lambda^{-1}|P_{I_{m'}}|^{-1}  \sum_{P_{I_{m'}}}
      \sum_{\tilde{P}_{I_{m'}}} (\int {W_{\tilde{P}_{I_{m'}}}} \phi_{P_{I_{m'}}}(\tilde{P}_{I_{m'}})|f_{m'+1,I_{m'}}|^2)  \phi_{P_{I_{m'}}}^{1/2}(x) \\
      & \lesssim \lambda^{-1}|P_{I_{m'}}|^{-1} 
      \sum_{\tilde{P}_{I_{m'}}} (\int {W_{\tilde{P}_{I_{m'}}}} |f_{m'+1,I_{m'}}|^2)  \phi_{\tilde{P}_{I_{m'}}}^{1/2}(P_{I_{m'}}(x)) \\
      & \lesssim \lambda^{-1}|P_{I_{m'}}|^{-1} 
      \int  |f_{m'+1,I_{m'}}|^2(y)\sum_{\tilde{P}_{I_{m'}}}  W_{\tilde{P}_{I_{m'}}}(y) \phi_{\tilde{P}_{I_{m'}}}^{1/2}(P_{I_{m'}}(x)) dy  \\
      & \lesssim \lambda^{-1}|P_{I_{m'}}|^{-1} 
       \int |f_{m'+1,I_{m'}}|^2(y) \phi_{P_{I_{m'}}(x)}^{1/2}(y)dy.
     \end{align*}
    Noting that $|P_{I_{m'}}|^{-1}\phi_{P_{I_{m'}}(x)}^{1/2}(y) \lesssim \rho_{I_{m'}}(x-y),$ we get
   $$
     |f_{m',I_{m'}}(x)-f_{m'+1,I_{m'}}(x)| \lesssim \lambda^{-1} |f_{m'+1,I_{m'}}|^2*\rho_{I_{m'}}(x).$$
   Summing the above over $I_{m'}\subset \bigcup_{I_m\in \mc{S}} I_m$ we conclude
    \begin{multline*}
    |\sum_{I_{m'}\subset \bigcup_{I_m\in \mc{S}}I_m} f_{m',I_{m'}}(x)-\sum_{I_{m'}\subset \bigcup_{I_m\in \mc{S}}I_m} f_{m'+1,I_{m'}}(x)| \leq \lambda^{-1} \sum_{I_{m'}\in \mc{I}_{m'}} |f_{m'+1,I_{m'}}|^2*\rho_{I_{m'}}(x) \\
     = \lambda^{-1} g_{m'}(x)
   \lesssim_\e N^{\e} \frac{r}{\lambda}.
  \end{multline*} 
  Therefore if we choose the constant $\tilde{C}_\e$ in the definition of $\lambda=\tilde{C}_\e N^\e \frac{r}{\alpha}$ to be large enough depending on $\e,$ then we have for $x\in U_{\alpha, r}\cap \Omega_{m,\alpha,r},$
  $$\sum_{m\leq m' \leq M}|\sum_{I_{m'}\subset \bigcup_{I_m\in \mc{S}}I_m} f_{m',I_{m'}}(x)-\sum_{I_{m'}\subset \bigcup_{I_m\in \mc{S}}I_m} f_{m'+1,I_{m'}}(x)| \leq \alpha/100.$$
  Since by definition  $\sum_{I_{m'}\subset \bigcup_{I_m\in \mc{S}}I_m} f_{m',I_{m'}}=\sum_{I_{m'-1}\subset \bigcup_{I_m\in \mc{S}}I_m} f_{m',I_{m'-1}},$
  we have by the triangle inequality that 
  \[|\sum_{I_m\in \mc{S}}f_{I_m}-\sum_{I_m\in \mc{S}}f_{m,I_m}|\leq \alpha/100.\]

  The case $m=0$ follows from the above argument for $m=1$ as by definition $ f_0=f_{1}.$
\end{proof}

From now on we will assume that $\tilde{C}_\e$ is chosen large enough such that the conclusion of Lemma \ref{fmcomparablelem} holds.

\section{Proof of Proposition \ref{induction}} \label{inductionsec}

We prove Proposition \ref{induction} in this section, and therefore Theorem \ref{refineddecthm} and Theorem \ref{decouplingthm}. Still fix $2\leq p\leq 6,$ $\e>0,$ and $P(L_M)\subset \R.$ 

Suppose $1\leq K\leq N^{\e/2}$ and $N^{1/2}/K \geq L.$ Let $\mc{I}'$ be a partition of $\mc{N}_{N^{-1}K^{-1}}(\{a_n\}_{n=1}^{N^{1/2}})$ into $K$ many $I',$ which is a union of $N^{1/2}/K$ consecutive intervals in $\mc{N}_{N^{-1}K^{-1}}(\{a_n\}_{n=1}^{N^{1/2}}).$ We call $I',I''\in \mc{I}'$ non-adjacent if there exist at least two other $I'''\in \mc{I}'$ between $I'$ and $I''$ on the real line. Alternatively, we can list $I'\in \mc{I}'$ as $I'_{j}$ so that $I'_{j+1}$ is on the right side of $I'_j$ on the real line for every $j.$ Then we define $I'_j,I_{j'}'$ to be non-adjacent if $|j-j'|\geq 3.$ In displayed math we write ``non-adj." as the shorthand for non-adjacent.

For $f$ with $\supp \hat{f}\subset \Omega,$ we let $f_{I'}$ denote the projection of $f$ to $I'$ in the frequency space. So $f_{I'}=\sum_{I_M\subset I'} f_{I_M}.$

\subsection{Broad-narrow decomposition}

The following lemma is a broad-narrow analysis on $f$ with some complication. For parameters $\a,r>0$ and $m$, $0\le m\le M-1$, define
\[ f_{m,\a,r,I'}:=\sum_{I_m\subset I'} f_{m,\a,r,I_m},\]
where we recall that $f_{m,\a,r,I_m}$ is defined in \eqref{1333}. 

\begin{lemma}\label{BN} For every $X\subset P(L_M),$ there exist some $\alpha, r$ with $\alpha \geq r^{1/2},$ and $0\leq m\leq M-1,$  such that
  \begin{multline}\label{BNeqn}
    \int_{X} |f|^p \lesssim_\e  \sum_{I'\in \mc{I}'} \int_{X} |f_{I'}|^p + (\log N \log(\theta^{-1}+1))^C \frac{K^C}{\alpha^{4-p}} \max_{\substack{I',I'' \\\text{ non-adj.}}}\int_{X\cap U_{\alpha,r}\cap \Omega_{m,\alpha,r}} |f_{m,\alpha,r,I'}|^{2}|f_{m,\alpha,r,I''}|^{2}\\ \nonumber
    +(\sup_{x\in X}\sum_{I} \|f_I\|_{\stkout{L}^2(W_{P_{I}(x),100})}^2)^{p/2-1}(\sum_I \|f_I\|^2_{L^2(W_{P(L),100})}).
  \end{multline}
\end{lemma}

First we prove a technical lemma which is a pointwise broad-narrow analysis. 


By taking all parameters to have dyadic values, we may assume that for each $I_m$, $0\le m\le M$, and any $I'$, either $I_m\subset I'$ or $I_m\cap I'=\emptyset$. 

\begin{lemma}\label{BNtech} For every $\a,r>0$ and $0\le m\le M-1$, 
\[ |f_{m,\a,r}(x)|^2\lesssim \max_{I'}|f_{I'}(x)|^2+ K^C\max_{\substack{I',I''\\\text{non-adj.}}}|f_{m,\a,r,I'}(x)||f_{m,\a,r,I''}(x)| \]
for every $x\in X\cap U_{\a,r}\cap\Omega_{m,\a,r}$. 
\end{lemma}
\begin{proof} 
Let $x\in X\cap U_{\a,r}\cap\Omega_{m,\a,r}.$ If there exist $I',I''\in \mc{I}'$ non-adjacent such that $|f_{m,\alpha,r,I'}|,|f_{m,\alpha,r,I''}|\geq \frac{1}{100K}|f_{m,\alpha,r}(x)|,$ then we have
\begin{equation}\label{this}
  |f_{m,\a,r}(x)|^2\lesssim  K^2\max_{\substack{I',I''\\\text{non-adj.}}}|f_{m,\a,r,I'}(x)||f_{m,\a,r,I''}(x)|.  
\end{equation}

Now we assume there do not exist $I',I''\in \mc{I}'$ non-adjacent with $|f_{m,\alpha,r,I'}|,|f_{m,\alpha,r,I''}|\geq \frac{1}{100K}|f_{m,\alpha,r}(x)|.$
Note that $f_{m,\a,r}(x)=\sum_{I'}f_{m,\a,r,I'}(x)$ and the number of $I'$ is bounded by $K.$ So if we choose $I'''\in \mc{I}'$ with $|f_{m,\a,r,I'''}(x)|=\max_{I'\in \mc{I}'}|f_{m,\a,r,I'}(x)|,$ then 
\begin{equation}\label{this2}
   |f_{m,\a,r,I'''}(x)| \geq \frac{1}{2} |f_{m,\a,r}(x)|. 
\end{equation}

By Lemma \ref{fmcomparablelem} we have $|f_{m,\a,r}(x)|\in [\a/4,4\a],$ and
$|f_{m,\a,r,I'''}(x)-f_{I'''}(x)|\le \frac{\a}{100}.$ Therefore by the triangle inequality and \eqref{this2} we obtain 
\[|f_{I'''}(x)|\gtrsim \alpha \sim |f_{m,\a,r}(x)|.\] 
This combined with \eqref{this} proves the lemma. 


\end{proof}


\begin{proof}[Proof of Lemma \ref{BN}]

Since $P(L_M)=\bigsqcup\limits_{\alpha,r:\text{ dyadic}} U_{\alpha,r},$ we have
$$\int_{X} |f|^p \leq \sum_{\alpha, r:\text{ dyadic}} \int_{X\cap U_{\alpha,r}}|f|^p.$$
Without loss of generality we assume
\begin{equation}\label{normalize}
  (\sup_{x\in X}\sum_{I} \|f_I\|_{\stkout{L}^2(W_{P_{I}(x),100})}^2)^{1/2-1/p}(\sum_I \|f_I\|^2_{L^2(W_{P(L),100})})^{1/p}=1.
\end{equation}
Then  $X\cap U_{\alpha,r}=\emptyset$ if $\max\{ \alpha, r\} \geq C N^{C} \theta^{-C} $ for some sufficiently large constant $C.$ Also
$$( \int_{X\cap (\bigcup_{\min \{\alpha, r\}\leq C N^{-C}\theta^{C}}U_{\alpha,r})}|f|^p )^{1/p} \lesssim 1$$
if $C$ is sufficiently large.
So now we write
$$\int_{X} |f|^p \leq \sum_{\alpha, r} \int_{X\cap U_{\alpha,r} }|f|^p +C$$
where the number of $\alpha, r$ in the summation is $\mc{O}(\log N \log(\theta^{-1}+1)).$

We also observe that by H\"{o}lder's inequality and Fubini's theorem we have
$$\int_{X\cap \bigcup_{\alpha \leq r^{1/2}}U_{\alpha,r}}|f|^p\lesssim \int_{X}(\sum_{I}|f_I|^2*\rho_{I})^{p/2} \lesssim \|\sum_I |f_I|^2*\rho_{I}\|_{L^\infty(X)}^{p/2-1} ( \sum_{I} \|f_I\|^2_{L^2(W_{P(L),100})}).$$
Since
$$\|\sum_I |f_I|^2*\rho_{I}\|_{L^\infty(X)} \leq \sup_{x\in X} \sum_I |f_I|^2*\rho_{I}(x) \lesssim \sup_{x\in X} \sum_I \|f_I\|_{\stkout{L}^2(W_{P_I(x),100})}^2 $$
we obtain
$$\int_{X\cap \bigcup_{\alpha \leq r^{1/2}}U_{\alpha,r}}|f|^p\lesssim (\sup_{x\in X}\sum_{I} \|f_I\|_{\stkout{L}^2(W_{P_{I}(x),100})}^2)^{p/2-1}(\sum_I \|f_I\|^2_{L^2(W_{P(L),100})})=1.$$
So in summary
\begin{equation}\label{777}
  \int_{X}|f|^p\lesssim \sum_{\alpha ,r: \alpha \geq r^{1/2}} \int_{X\cap U_{\alpha,r}} |f|^p+ 1
\end{equation}

Next we further decompose $X\cap U_{\alpha,r}$ into $\bigcup_{m}(X\cap U_{\alpha,r}\cap \Omega_{m,\alpha,r}):$
$$\int_{X\cap U_{\alpha,r}} |f|^p \leq \sum_{m=0}^{M-1} \int_{X\cap U_{\alpha,r}\cap \Omega_{m,\alpha,r}} |f|^p.$$
By Lemma \ref{fmcomparablelem} we have for $0\leq m\leq M-1,$
$$\int_{X\cap U_{\alpha,r}\cap \Omega_{m,\alpha,r}}|f|^p \sim \int_{ X\cap U_{\alpha,r}\cap \Omega_{m,\alpha,r}}|f_{m,\alpha,r}|^p.$$
It then follows from Lemma \ref{BNtech} and Lemma \ref{fmcomparablelem} that

\begin{align*}
  \int_{X} |f|^p & \lesssim 1+ \sum_{\alpha, r:\alpha \geq r^{1/2}} \sum_{m=0}^{M-1}  \left(  \sum_{I'\in \mc{I}'} \int_{X\cap U_{\alpha,r}\cap \Omega_{m,\alpha,r}} |f_{I'}|^p \right.  \\
  & \qquad \qquad \qquad \qquad \left. + \frac{K^C}{\alpha^{4-p}} \max_{\substack{I', I''\\ \text{ non-adj.}}}  \int_{X\cap U_{\alpha,r}\cap \Omega_{m,\alpha,r}} |f_{m,\alpha,r,I'}|^{2}|f_{m,\alpha,r,I''}|^{2} \right)  \\
   & \lesssim 1+ C_\e   \sum_{I'\in \mc{I}'} \int_{X} |f_{I'}|^p + \sum_{\alpha, r:\alpha \geq r^{1/2}} \sum_{m} \frac{K^C}{\alpha^{4-p}} \max_{\substack{I', I''\\  \text{ non-adj.}}}  \int_{X\cap U_{\alpha,r}\cap \Omega_{
   m,\alpha,r}} |f_{m,\alpha,r,I'}|^{2}|f_{m,\alpha,r,I''}|^{2}
\end{align*}
where we used $M\lesssim_\e 1$ in the last inequality. By  pigeonholing we have
\begin{multline*}
  \sum_{\alpha, r: \alpha \geq r^{1/2}} \sum_{m} \frac{K^C}{\alpha^{4-p}} \max_{\substack{I', I''\\ \text{ non-adj.}}}  \int_{X\cap U_{\alpha,r}\cap \Omega_{m,\alpha,r}} |f_{m,\alpha,r,I'}|^{2}|f_{m,\alpha,r,I''}|^{2} \\
  \lesssim_\e (\log N \log(\theta^{-1}+1)) \frac{K^C}{\alpha^{4-p}}  \max_{\substack{I', I''\\ \text{ non-adj.}}} \int_{X\cap U_{\alpha,r}\cap \Omega_{m,\alpha,r}} |f_{m,\alpha,r,I'}|^{2}|f_{m,\alpha,r,I''}|^{2}
\end{multline*}
for some $\alpha, r$  with $\alpha \geq r^{1/2},$ $0\leq m\leq M-1$, which completes the proof. 

\end{proof}

Now fix $X\subset P(L_M).$ We have identified a pair $(\alpha,r)$ from Lemma \ref{BN}, and we fix that pair of $\alpha,r$ and suppress the dependence on $\alpha,r$ from now on in the notation. In particular write $g_m=g_{m,\alpha,r},$ $\Omega_{m}=\Omega_{m,\alpha,r},$ $f_{m,I'}=f_{m,\alpha,r,I'}$ and $f_{m,I_m}=f_{m,\alpha,r,I_m}$ where $\alpha,r$ are those chosen in Lemma \ref{BN}.

We estimate the broad and narrow parts separately, which together with Lemma \ref{BN} will imply Proposition \ref{induction}.

\subsection{Narrow part}


\begin{proposition} \label{narrowprop}
  For every $I'\in \mc{I}'$ we have
  \begin{equation}\label{narrowpropeqn}
      \int_{X} |f_{I'}|^p \lesssim \left(\sup_{\theta'\in [\theta/4,\theta]}Dec(N/K^2,\theta'/K^2)^p\right) (\sup_{x\in X}\sum_{I\subset I'} \|f_I\|_{\stkout{L}^2(W_{P_{I}(x),100})}^2)^{\frac{p}{2}-1}(\sum_{I\subset I'} \|f_I\|^2_{L^2(W_{P(L),100})}).
  \end{equation}
\end{proposition}
\begin{proof}
  First prove \eqref{narrowpropeqn} for $I'=\mc{N}_{L^2\theta/N^2}(\{a_n\}_{n=1}^{N^{1/2}/K}).$
  Note that $K^2I'=\mc{N}_{K^2L^2\theta/N^2}(\{K^2a_n\})_{n=1}^{N^{1/2}/K},$ and if we let $\tilde{a_n}=K^2a_n,$ $\tilde{N}=N/K^2$ and $\tilde{\theta}=\theta/K^2,$ then 
  $$\tilde{a_2}-\tilde{a_1} \in [\frac{K^2}{4N}, \frac{4K^2}{N}] = [\frac{1}{4\tilde{N}},\frac{4}{\tilde{N}}], \quad (\tilde{a_{n+1}}-\tilde{a_n})-(\tilde{a_n}-\tilde{a_{n-1}}) \in [\frac{K^2\theta}{4N^2}, \frac{4K^2\theta}{N^2}] = [\frac{\tilde{\theta}}{4\tilde{N}^2}, \frac{4\tilde{\theta}}{\tilde{N}^2}].$$
  If we define $\tilde{P}(L),$ $\tilde{P}_I$ by \eqref{defnLball}, \eqref{1} respectively with $N,L,\theta,v_j$ replaced by $\tilde{N},L,\tilde{\theta},K^2v_j,$ then for any $x_0$ we have  $\tilde{P}(L,x_0)\subset K^{-2}P(L,x_0),$ and $\tilde{P}_I(x_0)=K^{-2}P_I(x_0).$ Therefore by the definition of the refined decoupling constant for  $\mc{N}_{L^2\tilde{\theta}/\tilde{N}^2}(\{\tilde{a}_n\}_{n=1}^{\tilde{N}}),$ and the spatial change-of-variables $x\mapsto K^{-2}x$, we have
  $$\int_{X} |f_{I'}|^p \lesssim Dec(N/K^2,\theta/K^2)^p (\sup_{x\in X}\sum_{I\subset I'} \|f_I\|_{\stkout{L}^2(W_{P_{I}(x),100})}^2)^{p/2-1}(\sum_{I\subset I'} \|f_I\|^2_{L^2(W_{P(L),100})}).$$

  Now consider a general $I'\in \mc{I}'.$ Suppose $a_{l}$ is the first term in $I'\cap \{a_n\}_{n=1}^{N^{1/2}},$ and let $v_{l}=a_{l+1}-a_{l}.$ Then because of \eqref{defa_n} we have
  $v_{l}\in [v_1,2v_1].$ So we may choose $K_{l}\in [K/\sqrt{2},K]$ such that
  $$K_l^2v_l \in [\frac{1}{4\tilde{N}},\frac{4}{\tilde{N}}].$$
  Then
  $$K_l^2\left((a_{n+1}-a_n)-(a_n-a_{n-1})\right)\in [\frac{\theta K_l^2}{4N^2},\frac{4\theta K_l^2}{N^2}]=[\frac{\tilde{\theta}' }{4\tilde{N}^2},\frac{4\tilde{\theta}_l }{\tilde{N}^2}]$$
  for some $\tilde{\theta}_l\in [\tilde{\theta}/4,\tilde{\theta}].$ Let $\theta_l=K^2\tilde{\theta}_l,$ which lies in $[\theta/4,4\theta].$ 
  So again by a change of variable argument we have
  $$\int_{X} |f_{I'}|^p \lesssim Dec(N/K^2,\theta_l/K^2)^p (\sup_{x\in X}\sum_{I\subset I'} \|f_I\|_{\stkout{L}^2(W_{P_{I}(x),100})}^2)^{p/2-1}(\sum_{I\subset I'} \|f_I\|^2_{L^2(W_{P(L),100})}).$$
  Therefore we have shown \eqref{narrowpropeqn} for every $I'\in \mc{I}'.$

\end{proof}

The proof of Proposition \ref{narrowprop} actually shows that \eqref{narrowpropeqn} holds for every $f$ with frequency support in $\Omega$ (not only alternately spaced $f$) and every $X\subset P(L).$

\subsection{Broad part}

\begin{proposition} \label{broadprop} For $1\leq m\leq M-1$ and $I',I''\in\mc{I}'$ non-adjacent we have
   \begin{equation}\label{bilineardeceqn}
  \int_{X\cap U_{\alpha,r}\cap \Omega_{m}} |f_{m,I'}|^{2}|f_{m,I''}|^{2} \lesssim_\e  N^{C\e} K^C \left( \frac{r}{\alpha} \right)^2 \left( \sum_{I\in \mc{I}} \|f_I\|_{L^2(W_{P(L),100})}^2  \right),
  \end{equation}
\end{proposition}

\begin{proof}
  Fix a $P(L_m')$ such that $P(L_m')\cap X \cap U_{\alpha,r}\cap \Omega_m \neq \emptyset.$ Recall that $L_m'=(L_mN^{1/2})^{1/2}$ as defined in Section \ref{BKBRsec}. By Proposition \ref{bilresprop} (together with rescaling) and H\"{o}lder's inequality we have
  \begin{align*}
    \int_{P(L_m')} |f_{m,I'}|^{2}|f_{m,I''}|^{2} & \lesssim_\e N^\e K^{C}  |P(L_m')|^{-1}\int \sum_{I_m\subset I'} |f_{m,I_m}|^2  W_{P(L_m'),200} \int \sum_{I_m\subset I''} |f_{m,I_m}|^2 W_{P(L_m'),200} \\
     & \lesssim N^\e K^{C} \int (\sum_{I_m} |f_{m,I_m}|^2)^2  W_{P(L_m'),200},
  \end{align*}
  and due to $|f_{m,I_m}|\leq |f_{m+1,I_m}|$ we further have
  $$ \int_{P(L_m')} |f_{m,I'}|^{2}|f_{m,I''}|^{2} \lesssim_\e N^\e  K^{C} \int (\sum_{I_m} |f_{m+1,I_m}|^2)^2  W_{P(L_m'),200}.$$
  Now applying Proposition \ref{locconstprop} we obtain
  \begin{align*}
     \int_{P(L_m')} |f_{m,I'}|^{2}|f_{m,I''}|^{2} & \lesssim_\e N^\e K^{C} |P(L_m')|^{-1} \left(\int (\sum_{I_m} |f_{m+1,I_m}|^2)  W_{P(L_m'),100}\right)^2 \\
     & \lesssim N^\e K^{C} \int_{P(L_m')} g_m^2.
  \end{align*}
  Note that from the definition of $\Omega_m$ and the definition of $g_m:=\sum_{I_m}|f_{m+1,I_m}|^2*\rho_{I_m}$ we have $x\in \Omega_m$ implies $|g_m(x)|\sim \sup_{y\in P(L_m'(x))}|g_m(y)| \lesssim |g_m^h(x)|.$ Therefore we have (by Proposition \ref{locconstprop})
  $$\int_{P(L_m')} g_m^2 \lesssim |P(L_m')||g_m^h(x)|^2\lesssim \int|g_m^h|^2 W_{P(L_m'),100},$$
  where $x\in P(L_m')\cap \Omega_m.$
  Summing over disjoint $P(L_m')$ that intersect $X\cap U_{\alpha,r }\cap \Omega_m$ we obtain
  $$\int_{X\cap U_{\alpha,r}\cap \Omega_m} |f_{m,I'}|^{2}|f_{m,I''}|^{2} \lesssim_\e N^\e K^{C} \int |g_m^h|^2 W_{P(L_M),100} \lesssim N^{2\e} K^C \int \sum_{I_m} |f_{m+1,I_m}|^4 W_{P(L_M),100},$$
  where the last inequality is due to Lemma \ref{hilem}. By H\"{o}lder's inequality and the definition of $f_{m+1,I_{m+1}}$ we have
  \begin{align*}
    \int \sum_{I_m} |f_{m+1,I_m}|^4 W_{P(L_M),100} & \lesssim N^{C\e} \int \sum_{I_{m+1}} |f_{m+1,I_{m+1}}|^4 W_{P(L_M),100} \\
     & \lesssim N^{C\e} \left( \frac{r}{\alpha} \right)^2 \int \sum_{I_{m+1}} |f_{m+1,I_{m+1}}|^2 W_{P(L_M),100}.
  \end{align*}
  By the pointwise inequality $|f_{m+1,I_{m+1}}|\leq |f_{m+2,I_{m+1}}|$ and local $L^2$ orthogonality (Proposition \ref{localL^2orth}),
  \begin{align*}
    \int \sum_{I_{m+1}} |f_{m+1,I_{m+1}}|^2 W_{P(L_M),100} & \lesssim \int \sum_{I_{m+2}} |f_{m+2,I_{m+1}}|^2 W_{P(L_M),100} \\
     & \lesssim \int \sum_{I_{m+2}} |f_{m+2,I_{m+2}}|^2 W_{P(L_M),100}.
  \end{align*}
  Continuing this process we obtain
  \begin{equation}\label{1122}
    \int \sum_{I_{m+1}} |f_{m+1,I_{m+1}}|^2 W_{P(L_M),100} \lesssim_\e \int \sum_{I_{M}} |f_{M,I_{M}}|^2 W_{P(L_M),100}.
  \end{equation}
  Recalling that $|f_{M,I_M}| \leq |f_{I_M}|=|f_I|$ we conclude
  $$\int_{X\cap U_{\alpha,r}\cap \Omega_m} |f_{m,I'}|^{2}|f_{m,I''}|^{2} \lesssim_\e N^{C\e} K^C \left( \frac{r}{\alpha} \right)^2 \int \sum_{I} |f_{I}|^2 W_{P(L_M),100}.$$
\end{proof}

\begin{proposition}\label{broad0prop}
  For $I',I''\in\mc{I}'$ non-adjacent we have
   \begin{equation*}
  \int_{X\cap U_{\alpha,r}\cap \Omega_0} |f_{0,I'}|^{p/2}|f_{0,I''}|^{p/2} \lesssim_\e N^{\e} (\sup_{x\in X}\sum_{I} \|f_I\|_{\stkout{L}^2(W_{P_{I}(x),100})}^2)^{p/2-1}(\sum_I \|f_I\|^2_{L^2(W_{P(L),100})}).
  \end{equation*}
\end{proposition}
\begin{proof}
  By the Cauchy-Schwarz inequality we have
  \begin{multline*}
    \int_{X\cap U_{\alpha,r}\cap \Omega_0} |f_{0,I'}|^{p/2}|f_{0,I''}|^{p/2} \lesssim N^\e \int_{X\cap U_{\alpha,r}\cap \Omega_0} (\sum_{I_1}|f_{1,I_1}|^2)^{p/2} \\
    \lesssim N^\e \sup_{x\in X\cap \Omega_0}(\sum_{I_1}|f_{1,I_1}|^2)^{p/2-1}\int \sum_{I_1}|f_{1,I_1}|^2 W_{P(L_M),100}.
  \end{multline*}
  We have shown in the proof of Proposition \ref{broadprop} (inequality \eqref{1122}) that
  $$\int \sum_{I_{1}} |f_{1,I_{1}}|^2 W_{P(L_M),100} \lesssim_\e \int \sum_{I} |f_{I}|^2 W_{P(L),100}.$$
  So it suffices to show
  \begin{equation}\label{1123}
    \sup_{x\in X\cap  \Omega_0}(\sum_{I_1}|f_{1,I_1}|^2) \lesssim_\e \sup_{x\in X}\sum_{I} \|f_I\|_{\stkout{L}^2(W_{P_{I}(x),100})}^2.
  \end{equation}
  From the locally constant property (Proposition \ref{locconstprop}) we have
  $$\sum_{I_1}|f_{1,I_1}|^2(x) \lesssim \sum_{I_1}|f_{1,I_1}|^2*\rho_{I_1}(x)\lesssim \sum_{I_1}|f_{2,I_1}|^2*\rho_{I_1}(x)=g_1(x),$$
  and by Lemma \ref{low} we have for $x\in X\cap \Omega_0,$
  $g_1(x) \lesssim_\e g_{M}(x).$ So we conclude
  $$\sup_{x\in X\cap  \Omega_0}\sum_{I_1}|f_{1,I_1}|^2(x) \lesssim_\e \sup_{x\in X\cap  \Omega_0}g_M(x) \lesssim \sup_{x\in X}\sum_{I} \|f_I\|_{\stkout{L}^2(W_{P_{I}(x),100})}^2.$$
\end{proof}

\subsection{Proof of Proposition \ref{induction}}


Let $X\subset P(L).$ We choose $\alpha,r$ as in Lemma \ref{BN}.  Note that $r\leq 2 \|\sum_{I}|f_I|^2\|_{L^\infty(X)}$ since otherwise $X\cap U_{\alpha, r}=\emptyset.$ So
$$r\leq 2\|\sum_{I}|f_I|^2\|_{L^\infty(X)} \lesssim \sup_{x\in X}\sum_{I\subset I'} \|f_I\|_{\stkout{L}^2(W_{P_{I}(x),100})}^2.$$
Also $\alpha \geq r^{1/2}$ implies that $r^{3-p/2}/\alpha^{6-p}\leq 1$ as  $p\leq 6.$ Therefore
combining Proposition \ref{narrowprop}, \ref{broadprop}, \ref{broad0prop} and Lemma \ref{BN} we obtain
\begin{multline}\label{10000}
  \int_{X}|f|^p\lesssim_\e \left( \sup_{\theta'\in [\theta/4,4\theta]}Dec(N/K^2,\theta'/K^2)^p
  +\log^C(\theta^{-1}+1) N^{C\e}K^C \right)  \\
  (\sup_{x\in X}\sum_{I} \|f_I\|_{\stkout{L}^2(W_{P_{I}(x),100})}^2)^{p/2-1}
  (\sum_I \|f_I\|^2_{L^2(W_{P(L),100})}).
\end{multline}


\section{A decoupling inequality for generalized Dirichlet sequences}\label{longsec}

In this section we focus only on generalized Dirichlet sequences with parameter $\theta=1$. That is, we say $\{a_n\}_{n=1}^N$ is a generalized Dirichlet sequence if it satisfies \eqref{defa_n} with $\theta=1.$ We will present a decoupling inequality for generalized Dirichlet sequences, by combining Theorem \ref{decouplingthm} and the flat decoupling (Proposition \ref{flatdec} below). Then we show that for certain choices of the generalized Dirichlet sequences  $\{a_n\}_{n=1}^{N}$ the decoupling inequality that we obtain in this way is sharp (up to $C_\e N^\e$). 

More precisely, for $1\leq L\leq N^{1/2},$ we let $\Omega'$ denote the $L^2/N^2-$neighborhood of $\{a_n\}_{n=1}^{N},$ and let $\{J\}_{J\in \mc{J}}$ be a partition of $\Omega'$ into $\Omega' \cap B_{N^{-1/2}}$ where $B_{N^{-1/2}}$ runs over a tiling of $\R$ by balls of radius $N^{-1/2}.$ So there are about $N^{1/2}$ many $J$ and each $J$ contains $\mc{O}(N^{1/2})$ many consecutive intervals in $\Omega'.$ For each $J$ we let $\mc{I}_J$ be the partition of $J$ into $I,$ which is a union of $L$ many consecutive intervals in $\Omega'.$

We have the following decoupling inequality for the partition $\Omega'= \bigsqcup_{J\in \mc{J}}\bigsqcup_{I\in \mc{I}_J} I.$
\begin{theorem}\label{longdecthm} For $2\leq p\leq 6,$ we have
  \begin{equation}\label{334}
    \|f\|_{L^p(\R)} \lesssim_\e N^{1/4-1/(2p)+\e} \left( \sum_{J\in \mc{J}} \sum_{I\in \mc{I}_J} \|f_I\|^2_{L^p(\R)} \right)^{1/2}
  \end{equation}
  for every $f:\R\rightarrow \C$ with $\supp \hat{f}\subset \Omega'.$ There exists a choice of $\{a_n\}_{n=1}^{N}$ (satisfying \eqref{defa_n} with $\theta=1$) such that the above estimate is sharp up to $N^\e$ factor.
\end{theorem}

\subsection{Proof of \eqref{334}}
From Theorem \ref{decouplingthm} we have for every $J\in \mc{J}$ and $2\leq p\leq 6,$
\begin{equation}\label{333}
  \|f_J\|_{L^p(\R)} \lesssim_\e N^\e \left( \sum_{I\in \mc{I}_J} \|f_I\|^2_{L^p(\R)} \right)^{1/2}.
\end{equation}

Next we decouple $f_J$ into $f_I$ using the flat decoupling:
\begin{proposition}\label{flatdec}
  Let $\mc{U}$ denote the partition $[0,M)=\bigsqcup_{m=0}^{M-1}[m,m+1).$ Then for $p\geq 2$ we have
  $$\|f\|_{L^p(\R)} \lesssim_p M^{1/2-1/p} \left( \sum_{U\in \mc{U}} \|f_{U}\|_{L^p(\R)}^2 \right)^{1/2}$$
  for every $f:\R\rightarrow \C$ with $\supp \hat{f} \subset [0,M).$
\end{proposition}
Flat decoupling inequality is well-known (see for example \cite{demeter2020Fourier}) but we include a proof here for the sake of completeness.
\begin{proof}
  Fix $p\geq 2.$ It suffices to prove that
  $$\|f\|_{L^p(B_{1})} \lesssim M^{1/2-1/p} \left( \sum_{U\in \mc{U}} \|f_{U}\|_{L^p(W_{B_1,100})}^2 \right)^{1/2}$$
  for $f$ with $\supp \hat{f} \subset [0,M).$ We calculate
  \begin{align*}
    \|f\|_{L^p(B_{1})}^p & \leq \|f\|^{p-2}_{L^\infty(B_1)} \|f\|^2_{L^2(B_1)} \\
     & \lesssim (\sum_{U} \|f_U\|_{L^\infty(B_1)})^{p-2} (\sum_{U} \|f_U\|^2_{L^2(W_{B_1,100})}) \\
     & \lesssim (\sum_U \|f_U\|_{L^p(W_{B_1,100})})^{p-2} (\sum_{U} \|f_U\|^2_{L^p(W_{B_1,100})}) \\
     & \lesssim M^{(p-2)/2} (\sum_U \|f_U\|^2_{L^p(W_{B_1,100})})^{(p-2)/2} (\sum_{U} \|f_U\|^2_{L^p(W_{B_1,100})}) \\
     & \lesssim M^{(p-2)/2} (\sum_U \|f_U\|^2_{L^p(W_{B_1,100})})^{p/2}.
  \end{align*}
  Here we used the locally constant property similar to Proposition \ref{locconstprop} and local $L^2$ orthogonality similar to Lemma \ref{localL^2orth}.
\end{proof}

Now we prove the the decoupling inequality in Theorem \ref{longdecthm}.
\begin{proof}[Proof of \eqref{334} in Theorem \ref{longdecthm}]
 Combining \eqref{333} with Proposition \ref{flatdec} we obtain
\begin{align*}
  \|f\|_{L^p(\R)} &
  \lesssim_\e N^\e  \left( \sum_{J\in \mc{J}}  \|f_J\|^2_{L^p(\R)} \right)^{1/2}  \\
  & \lesssim N^{1/4-1/(2p)+\e} \left( \sum_{J\in \mc{J}} \sum_{I\in \mc{I}_J} \|f_I\|^2_{L^p(\R)} \right)^{1/2}
\end{align*}
for $f$ with $\supp \hat{f} \subset \Omega'.$
\end{proof}

\subsection{An example and sharpness of \eqref{334}}

To prove the sharpness part, we construct a sequence $\{a_n\}_{n=1}^N$ satisfying \eqref{defa_n} (with $\theta=1$) and for which \eqref{334} is sharp.
We will use the function
\[ g(x)=\frac{4x+(N^{1/2}-\sqrt{N-4x})^2}{4N} \]
to define the sequence. For $n=0,\ldots,\frac{N}{8}$, let
\[ a_n=g(n). \]
Distinguish the subsequence $a_{n_k}$ where $n_k=kN^{1/2}-k^2$.

\begin{lemma} \label{longexample} There is an absolute constant $N_0>0$ such that for every $N\geq N_0,$ the sequence $\{a_n\}_{n=1}^{N/8}$ constructed above satisfies property \eqref{defa_n} (with $\theta=1$). Furthermore, there is an absolute constant $c>0$ so that \[\big\{\frac{j}{N^{1/2}}:j=1,\ldots,\lfloor{cN^{1/2}}\rfloor \big\} \]
is a subsequence of $\{a_n\}_{n=1}^{N/8}$.
\end{lemma}

\begin{proof} First we verify the presence of the subsequence: Let $n_k$ and $a_{n_k}$ be as above. Calculate directly that
\begin{align*}
    a_{n_k}&=g(n_k)=\frac{4n_k+(N^{1/2}-\sqrt{N-4n_k})^2}{4N} \\
    &= \frac{4(kN^{1/2}-k^2)+(N^{1/2}-\sqrt{N-4(kN^{1/2}-k^2)})^2}{4N} \\
    &= \frac{4(kN^{1/2}-k^2)+(N^{1/2}-(N^{1/2}-2k))^2}{4N}\\
    &= \frac{4kN^{1/2}-4k^2+4k^2}{4N}=\frac{k}{N^{1/2}}.
\end{align*}
This calculation holds as long as $k\le \frac{N^{1/2}}{2}$. Also note that $n_k=kN^{1/2}-k^2$ is increasing as a function of $k$ as long as $k\le \frac{N^{1/2}}{2}$, so the $n_k$ define a subsequence $a_{n_0},\ldots,a_{n_K}$ where $K=\lfloor\frac{N^{1/2}}{2}\rfloor$.

To verify property (\ref{defa_n}), it suffices to check that for $N$ large enough
\begin{equation}\label{prop1}
    a_1-a_0\in [\frac{1}{2N}, \frac{2}{N}]
\end{equation}
and that
\begin{equation}\label{prop2} (a_{n+1}-a_n)-(a_n-a_{n-1}) \in [\frac{1}{4N^2},\frac{4}{N^2}] \end{equation}
whenever $1\le n\le \frac{N}{8}-1,$ since \eqref{prop1} together with \eqref{prop2} will imply $a_2-a_1\in [\frac{1}{4N},\frac{4}{N}]$ for $N$ large enough.

First we check \eqref{prop1}. Note that $a_0=0$ and
\[ a_1=g(1)=\frac{4+(N^{1/2}-\sqrt{N-4})^2}{4N} . \]
Then
\begin{align*}
    a_1-a_0&= \frac{1}{4N}\left(4+\frac{16}{(N^{1/2}+\sqrt{N-4})^2}\right)\in [\frac{1}{2N},\frac{2}{N}]
\end{align*}
if $N$ is large enough.

Next we check \eqref{prop2}. First calculate
\begin{align*}
g(x+1)-g(x)&= \frac{4+(N^{1/2}-\sqrt{N-4x-4})^2-(N^{1/2}-\sqrt{N-4x})^2}{4N} \\
&= \frac{4+2N^{1/2}(\sqrt{N-4x}-\sqrt{N-4x-4})-4}{4N}\\
&= \frac{\sqrt{N-4x}-\sqrt{N-4x-4}}{2N^{1/2}}\\
&= \frac{2}{N^{1/2}(\sqrt{N-4x}+\sqrt{N-4x-4})}
\end{align*}
Use this formula to calculate the difference
\begin{align*}
&(a_{n+1}-a_n)-(a_n-a_{n-1})\\
&\qquad =\frac{2}{N^{1/2}}\Big(\frac{1}{\sqrt{N-4n}+\sqrt{N-4n-4}}-\frac{1}{\sqrt{N-4n+4}+\sqrt{N-4n}}\Big)\\
&= \frac{2}{N^{1/2}}\frac{\sqrt{N-4n+4}-\sqrt{N-4n-4}}{(\sqrt{N-4n}+\sqrt{N-4n-4})(\sqrt{N-4n+4}+\sqrt{N-4n})}\\
&= \frac{16}{N^{1/2}(\sqrt{N-4n}+\sqrt{N-4n-4})(\sqrt{N-4n+4}+\sqrt{N-4n})(\sqrt{N-4n+4}+\sqrt{N-4n-4})}.
\end{align*}
As long as $n\le \frac{N}{8},$ and $N$ is sufficiently large, this lies in $[\frac{1}{4N},\frac{4}{N}]$ and we are done.

\end{proof}

Now we can finish the sharpness part of Theorem \ref{longdecthm}.

\begin{proof}[Proof of the sharpness part of Theorem \ref{longdecthm}]
  For $N\geq N_0,$ we take $\{a_n\}_{n=1}^{N/8}$ to be the sequence constructed in Lemma \ref{longexample}, extended arbitrarily to $\{a_n\}_{n=1}^N$ so that \eqref{defa_n} is satisfied with $\theta=1.$ We take $f=\sum_I f_I$ to be the function
  $$ \phi_{N^2/L^2} (x)\sum_{n=1}^{\lfloor{cN^{1/2}}\rfloor} e^{ixa_n}$$
  where $c$ is the constant in lemma \ref{longexample}, and $\phi_{N^2/L^2}(x)$ is an $L^\infty-$normalized Schwartz function whose Fourier transform is a smooth bump adapted to $B_{L^2/N^2}(0).$ Then we have
  $$\|f\|_{L^p(\R)}\gtrsim N^{1/2} \left(\frac{N^{3/2}}{L^2}\right)^{1/p}$$
  since $|f(x)|\sim N^{1/2}$ on $P_{N^{1/2}}^{C}(0)\cap B_{CN^2/L^2}(0).$ Since $|f_I|=\phi_{N^2/L^2},$ we have
  $$\left( \sum_{J\in \mc{J}} \sum_{I\in \mc{I}_J} \|f_I\|^2_{L^p(\R)} \right)^{1/2} \sim N^{1/4}\left( \frac{N^2}{L^2} \right)^{1/p}.$$
  Therefore \eqref{334} is sharp up to $N^\e.$
\end{proof}

\subsection{Some discussions}


If we take $L=1$ and $p=4$ in Theorem \ref{longdecthm}, we get \begin{equation}\label{2001}
    \|\sum_{n=1}^{N} b_ne^{ia_nx}\|_{L^4(B_{N^2})} \lesssim_\e N^{1/2+1/8+\e} \|b_n\|_{\ell^2},
\end{equation}
On the other hand, for the Dirichlet polynomial we have, by unique factorization in $\Z$ and local $L^2$  orthogonality, that
\begin{equation}\label{2000}
    \|\sum_{n=N+1}^{2N} b_ne^{ix\log n}\|_{L^4(B_{N^2})} = \|\sum_{m=N+1}^{2N}\sum_{n=N+1}^{2N} b_mb_ne^{ix\log (nm)}\|_{L^2(B_{N^2})}^{1/2} \lesssim_\e N^{1/2+\e} \|b_n\|_{\ell^2}.
\end{equation}
Comparing \eqref{2001} with \eqref{2000} we see that while we can construct a generalized Dirichlet sequence that contains an AP with about $N^{1/2}$ many terms and common difference $N^{-1/2}$ so that \eqref{2001} is sharp for that sequence, the Dirichlet sequence $\{\log n\}_{n=N+1}^{2N}$ does not contain such ($N^{-2}$-approximate) AP and therefore allows a better estimate \eqref{2000}. 


However we notice that the example $D_0(x)=\sum\limits_{j=1}^{cN^{1/2}} e^{ixj/N^{1/2}}$ does not exclude the possibility that Montgomery's conjecture may hold for generalized Dirichlet polynomials. By Montgomery's conjecture for generalized Dirichlet polynomials we mean for every $\e>0,$
\begin{equation}\label{MCgeneral}
    \|\sum_{n=1}^{N} b_ne^{ixa_n}\|_{L^p(B_{T})}  \lesssim_\e T^\e N^{1/2} (N^{p/2}+T)^{1/p} \|b_n\|_{\ell^\infty}
\end{equation}
for every generalized Dirichlet sequence $\{a_n\}_{n=1}^N$ with $\theta= 1.$
Indeed we know $|D_0(x)| \gtrsim N^{1/2}$ on $P^C_{N^{1/2}}(0),$ so
\[ \|D_0\|_{L^p(B_T)} \gtrsim T^{\frac{1}{p}} N^{\frac{1}{2}-\frac{1}{2p}}. \]
On the right hand side of \eqref{montconjmeaneqn} we have $C_\e N^{1/2+\e}(N^{p/2}+T)^{1/p} \geq N^{1/2}T^{1/p}.$ So there is no contraction to \eqref{MCgeneral}. Note that if we apply H\"older's inequality $\|b_n\|_{\ell^2} \leq N^{1/2} \|b_n\|_{\ell^\infty}$ to \eqref{2000} then we obtain
\[\|\sum_{n=N+1}^{2N} b_ne^{ix\log n}\|_{L^4(B_{N^2})}  \lesssim_\e N^{1+\e} \|b_n\|_{\ell^\infty},\]
which is exactly \eqref{montconjmeaneqn} with $p=4,T=N^2.$ However although we know \eqref{2001} is sharp (up to $C_\e N^\e$) for our example $D_0(x),$ the H\"older step $\|b_n\|_{\ell^2} \leq N^{1/2} \|b_n\|_{\ell^\infty}$ is not sharp because $D_0(x)$ has only $N^{1/2}$ many nonzero coefficients.

On the other hand we may construct a periodic generalized Dirichlet polynomial $f=\sum_{n=1}^N e^{it \frac{(N+n)}{N^2}}$ which contradicts \eqref{MCgeneral} for $p> 4,$ $T> N^{2+\e_0}$ with any $\e>0.$ We notice that $|f|\gtrsim N$ on $\mc{N}_{C} (N^2\Z).$ So 
\[\|f\|_{L^p(B_T)}\gtrsim N (\frac{T}{N^2})^{\frac{1}{p}}=N^{1-\frac{2}{p}}T^{\frac{1}{p}}.\]
Under the condition $p>4$ we have \[N^{1-\frac{2}{p}}T^{\frac{1}{p}} \gtrsim_{\e_0} N^{\e_1} N^{\frac{1}{2}} T^{\frac{1}{p}} \]
for some $\e_1>0$ depending on $p.$
Under the condition $T>N^{2+\e_0}$ we have
\[N^{1-\frac{2}{p}}T^{\frac{1}{p}} > N^{\e_2} N,\]
for some $\e_2>0$ depending on $p.$ 
Therefore when $p>4$ and $T> N^{2+\e_0}$ with any $\e_0>0,$ \eqref{MCgeneral} fails for the generalized Dirichlet polynomial $f.$


At the end of this section we discuss briefly what makes $N^{1/2}$ special. Suppose we consider the sequence $\{\log (N+n)\}_{n=1}^{N^\alpha}$ for some $\alpha\in (\frac{1}{2},1].$ For simplicity we will omit constants $C$ in the following discussion. Still we look at $\frac{L^2}{N^2}-$neighborhood of $\{\log (N+n)\}_{n=1}^{N^\alpha}$ with $L\geq 1.$ For $L\geq N^{1/2},$ the $\frac{L^2}{N^2}-$neighborhood is essentially the same as the $\frac{1}{N}-$neighborhood (as long as $L\leq N$), which is an interval of length about $1.$ So the induction scheme in this paper fails to work for $L\geq N^{1/2}.$ 

Another difficulty is about the ``bush'' structure of $\bigcup_I (I-I)$ in the frequency space. To illustrate this, we define $I,P_I$ as before so now there are $\frac{N^\alpha}{L}$ many $I,$ $v_I\sim \frac{1}{N}$ are $\frac{L}{N^2}$ separated, and the maximal separation of $v_I$ is $\frac{1}{N^{2-\alpha}}.$ For $\alpha>1/2,$ we no longer have an essentially linear decaying pattern of the bush $\bigcup_{I} (I-I)$ if $L\geq N^{1-\alpha},$ which is exploited in the proof of Lemma \ref{hilem}. To be precise, we consider the function $\sum_{I} 1_{I-I}(t),$ which counts the number of overlap of the sets $I-I$ at $t.$ If $\alpha \leq 1/2$ then we can verify that 
\begin{equation}\label{bushdecay}
    |\sum_{I} 1_{I-I}(t)| \lesssim \frac{N/L}{|t|} \quad \text{when  } \frac{1}{N}\lesssim |x| \lesssim \frac{L}{N}.
\end{equation}
See Figure \ref{fig:bush} for a rough graph of the function $\sum_{I} 1_{I-I}(t).$
However if $\alpha >1/2$ then we no longer have \eqref{bushdecay}. 
This is because $1/2$ is the largest value for $\alpha$ such that for every $L\leq N^{1/2},$ the $k-$th intervals in all $I-I$ are within about $N^{-1}$ distance from each other, for every $1\leq k\leq L.$ For comparison, we note that for $R^{-1/2}\times R^{-1}$ caps $\theta$ that tile the $R^{-1}$-neighborhood of the truncated parabola, the bush $\{\theta-\theta\}$ has a similar linear decay pattern:
\[|\sum_{\theta} 1_{\theta-\theta}(x)| \lesssim \frac{R^{-1/2}}{|x|} \quad \text{when  } R^{-1}\lesssim |x| \lesssim R^{-1/2}.\]

On the physical side, how $P_I$ interact also becomes more complicated when $\alpha >2$. One important property we used in the $\alpha=1/2$ case is that the maximal separation of $v_I^{-1}$ (which is about $N^{1/2}$) is less than the thickness of $P_I$ (which is about $N/L$) for every $1\leq L\leq N^{1/2}.$ However for $\alpha >1/2,$ the maximal separation is about $N^{1-\alpha}$ which is greater than the thickness $N/L$ for $L\geq N^{1-\alpha}.$ In particular this makes the pattern of the intersection $P_I\cap P_J$ more complicated and the notion of transversal less clear.

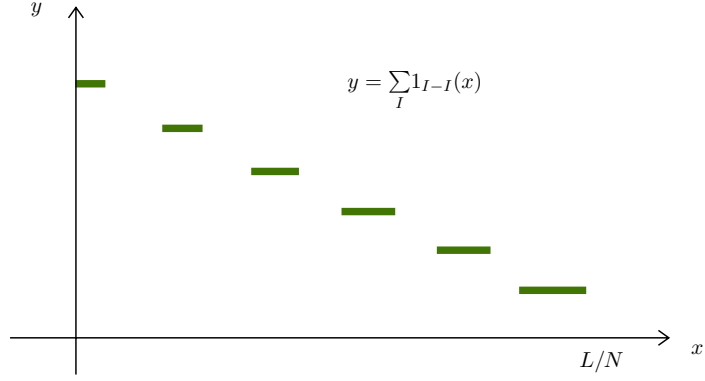
\begin{figure}
\begin{tabular}{ cc }
\scalebox{0.75}{

\tikzset{every picture/.style={line width=0.75pt}} 

\begin{tikzpicture}[x=0.75pt,y=0.75pt,yscale=-1,xscale=1]

\draw [color={rgb, 255:red, 65; green, 117; blue, 5 }  ,draw opacity=1 ][line width=3.75]    (72.5,99.33) -- (92.6,99.33) ;
\draw [color={rgb, 255:red, 65; green, 117; blue, 5 }  ,draw opacity=1 ][line width=3.75]    (130.85,129.33) -- (157.95,129.33) ;
\draw [color={rgb, 255:red, 65; green, 117; blue, 5 }  ,draw opacity=1 ][line width=3.75]    (190.67,158.33) -- (222.77,158.33) ;
\draw [color={rgb, 255:red, 65; green, 117; blue, 5 }  ,draw opacity=1 ][line width=3.75]    (251.4,185.33) -- (287.5,185.33) ;
\draw [color={rgb, 255:red, 65; green, 117; blue, 5 }  ,draw opacity=1 ][line width=3.75]    (315.5,211.33) -- (351.6,211.33) ;
\draw [color={rgb, 255:red, 65; green, 117; blue, 5 }  ,draw opacity=1 ][line width=3.75]    (370.85,238.33) -- (415.95,238.33) ;
\draw  (28.5,270.3) -- (471.5,270.3)(72.8,48) -- (72.8,295) (464.5,265.3) -- (471.5,270.3) -- (464.5,275.3) (67.8,55) -- (72.8,48) -- (77.8,55)  ;

\draw (254,89.4) node [anchor=north west][inner sep=0.75pt]    {$y=\underset{I}{\sum} 1_{I-I}( x)$};
\draw (41,43.4) node [anchor=north west][inner sep=0.75pt]    {$y$};
\draw (485,273.4) node [anchor=north west][inner sep=0.75pt]    {$x$};
\draw (410,277.4) node [anchor=north west][inner sep=0.75pt]    {$L/N$};

\end{tikzpicture}
}\\
\end{tabular}
    \caption{The overlap number of the $I-I$ has a linear decay pattern provided $L/N^{2-\alpha} \lesssim N^{-1}.$ This condition is guaranteed as long as $\alpha \leq 1/2.$  Controlling the overlap number of the $I-I$ outside of a certain neighborhood of the origin is a central step in Lemma \ref{hilem}. }
    \label{fig:bush}
\end{figure}

\section{Small-cap type decoupling}\label{smallcapsec}

In this section we prove Theorem \ref{smallcapthm_intro}, which is about small-cap type decoupling inequalities in the spirit of \cite{demeter2020small}.

First we restate Theorem \ref{smallcapthm_intro} but with the more general definition of generalized Dirichlet sequence.
Let $\{a_n\}_{n=1}^{N^{1/2}}$ be a short generalized Dirichlet sequence with parameter $\theta\in (0,1]$ as defined in Definition \ref{defnDirichlet}. Let $L,L_1$ be two integers such that $1\leq L_1\leq  L\leq N^{1/2}.$ Denote by $\Omega$  the $\theta L^2/N^2-$neighborhood of $\{a_n\}_{n=1}^{N^{1/2}}.$ We let $\{J\}_{J\in \mc{J}}=\{J_k\}_{k=0}^{\lfloor N^{1/2}/L_1 \rfloor}$ be the partition of $\Omega$ into unions of $L_1$ many consecutive intervals, that is,
\[J_k= \bigcup_{i=1}^{L_1} B_{\theta L^2/N^2} (a_{kL_1+i}). \] 
Let $\{I\}_{I\in \mc{I}}$ be the partition of $\Omega$ into unions of $L$ many consecutive intervals, which we called the canonical partition. 

A more general version of Theorem \ref{smallcapthm_intro} is the following, which we prove in the rest of this section.

\begin{theorem}\label{smallcapthm} Let $\{J\}_{J\in \mc{J}}$ be defined as in the above paragraph. 
  Suppose $p\geq 4,$ $\frac{1}{q}+\frac{3}{p}\leq 1.$
  If either of the following two conditions is satisfied
  \begin{enumerate}[label=(\alph*) ]
      \item $L_1=1,$
      \item $p=q,$
  \end{enumerate}
  then for every $\e>0,$
  \begin{equation}\label{s1}
    \|\sum_{J\in \mc{J}}f_J\|_{L^p(\R)} \lesssim_\e N^\e \log^C(\theta^{-1}+1) \left( \frac{N^{\frac{1}{2}-\frac{1}{2q}-\frac{3}{2p}}L^{\frac{2}{p}}}{L_1^{1-\frac{1}{p}-\frac{1}{q}}} +\left( \frac{N^{1/2}}{L_1} \right)^{\frac{1}{2}-\frac{1}{q}} \right)\left( \sum_{J\in \mc{J}} \|f_J\|_{L^p(\R)}^q \right)^{1/q}
  \end{equation}
  for every functions $f_J:\R\rightarrow \C$ with  $\supp \widehat{f_J}\subset J.$
\end{theorem}

As a corollary we have a more general version of Corollary \ref{discretesmallcor_intro}.

\begin{corollary}\label{discretesmallcor}
  Let $\{a_n\}_{n=1}^{N^{1/2}}$ be a short generalized  Dirichlet sequence with parameter $\theta\in (0,1].$
  Suppose $p\geq 4,$ $\frac{1}{q}+\frac{3}{p}\leq 1,$ and $N\theta^{-1}\leq T \leq N^2\theta^{-1}.$ We have for every $\e>0,$ 
  \begin{equation}\label{s120}
    \|\sum_{n=1}^{N^{1/2}}b_n e^{ita_n}\|_{L^p(B_T)} \lesssim_\e  N^{\e} \log^C(\theta^{-1}+1)  \left(N^{\frac{1}{2}(1+\frac{1}{p}-\frac{1}{q})}\theta^{-\frac{1}{p}} +T^{\frac{1}{p}} N^{\frac{1}{4}-\frac{1}{2q}} \right) \|b_n\|_{\ell^q}.
  \end{equation}
  for every $B_T,$ and every $\{b_n\}_{n=1}^{N^{1/2}}\subset \C,$ 
\end{corollary}

To prove results of the form \eqref{s1}, we may use the small cap decoupling method for $\mathbb{P}^1$ developed in \cite{demeter2020small}, which is based on refined decoupling for the canonical partition, refined flat decoupling and an incidence estimate for tubes with spacing conditions. We have three analogous results in the short generalized  Dirichlet sequence setting. Theorem \ref{refineddecthm} is the analogy of the refined canonical cap decoupling for $\mathbb{P}^1.$ Now we state and prove the other two.

\subsection{An incidence estimate for fat APs}
We start with the incidence estimate. First we introduce some notations. Suppose $P,P'$ are fat APs such that $P=P_I(y)$ and $P'=P_{I'}(y')$ for some $I,I'\in \mc{I}.$ We say $P,P'$ are parallel if $I=I'.$ For a collection $\mc{P}=\{P\}$ of fat APs, we say $x\in \R$ is an $r$-rich point if $r$ many $P$ contain it.

\begin{proposition}\label{incidenceprop}
  Let $1\leq L_1\leq L \leq N^{1/2}$ and let $\{J\}_{J\in \mc{J}},$ $\{I\}_{I\in \mc{I}}$ be defined as in the beginning of Section \ref{smallcapsec}.
  Suppose we have a collection of fat AP $\mc{P}=\{P\}$ inside a fixed $P(L),$ where each $P=P_I$ for some $I\in \mc{I}.$ Assume for every $J\in \mc{J}$ and every $P_J\subset P(L),$ $P_J$ contains either $M$ or $0$ parallel $P\in \mc{P}.$ Denote by $Q_r$ the set of $r$-rich points of $\mc{P}.$ Suppose $Q_r \neq \emptyset.$ Then one of the two cases below happens:
  \begin{enumerate}
    \item There exists a dyadic $s\in [1,\min\{L,N^{1/2}/L\}]$ and $M_s \in \N$ such that
    \begin{equation} \label{s103}
        |Q_r|\lessapprox \frac{M_s}{sr^2}(\#P)|P|
    \end{equation}
    \begin{equation}\label{s104}
        r\lessapprox \frac{M_sN^{1/2}}{s^2L}
    \end{equation}
    \begin{equation} \label{s105}
        M_s \lesssim sM\max\left\{1,s\frac{L_1}{L}\right\}.
    \end{equation}
    \item  \begin{equation} \label{s106}
        |Q_r|\leq |P(L)|
    \end{equation}
    \begin{equation}\label{s107}
        r\lesssim (\#P)\frac{|P|}{|P(L)|}.
    \end{equation}
  \end{enumerate}
  Here $\#P$ denotes the cardinality of $\mc{P}.$
\end{proposition}
\begin{proof}
  For each dyadic $1\leq s\leq \min\{L,N^{1/2}/L\},$ we let $\eta_s$ denote a smooth bump with height $1$ adapted to the annulus $|\xi|\sim \frac{L}{s}v$ in the frequency space, and let $\eta_0$ denote a smooth bump with height $1$ adapted to $P^{C\theta L^2/N^2}_{v_1}(0) \cap B_{CL^2/N^{3/2}}(0)$ (which degenerates to $B_{C\theta L^2/N^2}(0)$ when $L\leq N^{1/4}$) such that 
  \[\eta_0+\sum_{ \substack{1\leq s\leq \min\{L,N^{1/2}/L\}, \\ s:\, \text{dyadic} }  }\eta_s=1 \quad \text{on } \bigcup_{I} (I-I).\]
  For each $P\in \mc{P}$ we let $v_P(x)$ be a positive smooth function (with height $1$) adapted to $P$ in the physical space with frequency support in $C(I-I),$ where $P=P_I.$
  If we define $g=\sum_P v_P,$ then we can write
  $$g=g*\check{\eta_0}+\sum_{1\leq s\leq \min\{L,N^{1/2}/L\}} g*\check{\eta_s} .$$
  
  Fix $s\in [1, \min\{L,N^{1/2}/L\}].$  There exists a collection of fat APs $\mc{I}_s$ consisting of $I_s=P_{v_{I_s}}^{C\theta sL^2/N^2}(0)\cap B_{CL/N}(0)$ with the properties that $v_{I_s}\sim N^{-1}$ and $v_{I_s}$ are $\sim \frac{s\theta L}{N^2}$ separated, such that for every $I\in \mc{I},$ $I-I$ is contained in one and only one $I_s\in \mc{I}_s.$ 
  In fact we may let $v_{I_s}=v_I$ for any $I$ with $(I-I)\subset I_s.$ The cardinality of $\mc{I}_s$ is $N^{1/2}/(sL).$ For $I_s \in \mc{I}_s$ we let $\mc{P}_{I_s}$ be the tiling of $\R$ by fat APs of the form $P_{v_{I_s}^{-1}}^{\theta CsN/L} \cap B_{CN^2/(L^2\theta)}.$
  
  For every $P=P_I\in \mc{P}$ there exists a unique $I_s\in \mc{I}_s$ and $P_s\in \mc{P}_{I_s}$ such that $I-I \subset I_s$ and $P\subset P_s.$ 
  For every $1\leq M \lesssim s^2,$
  we define $\mc{P}_{s,M}$ be the sub-collection of $\mc{P}$ consisting of $P$ such that $P_s$ contains $\sim M$ many $P'\in \mc{P}.$ For $1\leq s \leq \min\{L,N^{1/2}/L\}$ let 
  \[g_{s,M}= \sum_{P\in \mc{P}_{s,M}} v_P* \check{\eta}_s.\]

  By the pigeonhole principle, for every $x\in Q_r$ there either exist an $s$ and $M_s$ such that 
  $g(x) \lessapprox |g_{s,M_s}(x)|.$
  or 
  $g(x) \lessapprox |g_0(x)|.$
  Again by the pigeonhole principle either we can find $s,M_s$ such that for $x$ in a subset $E$ of $Q_r$ with measure $\gtrapprox |Q_r|,$
  \[g(x) \lessapprox |g_{s,M_s}(x)|\]
  or for $x$ in a subset $E$ of $Q_r$ with measure $\gtrapprox |Q_r|,$
  \[g(x) \lessapprox |g_0(x)|.\]
  We consider these two cases separately.
  
  \vspace{3mm} 
  \noindent 
  {\bf Case 1. } Suppose $g(x) \lessapprox |g_{s,M_s}(x)|$ for $x$ in a subset $E$ of $Q_r$ with measure $\gtrapprox |Q_r|.$  We write 
  \[g_{s,M_s}=\sum_{I_s} \sum_{P_{I_s}} \sum_{P\subset P_{I_s},\, P\in \mc{P}_{s,M_s}} v_{P}* \check{\eta}_s=: \sum_{I_s} \sum_{P_{I_s}} g_{P_{I_s}}. \]
  Here the sum over $P_{I_s}$ is over $P_{I_s}\in \mc{P}_{I_s}$ such that $g_{P_{I_s}}$ is nonzero.
  
  We note that $\sum_{P_{I_s}} g_{P_{I_s}}$ with $I_s$ varying are almost orthogonal (meaning that the Fourier support of them has $\mc{O}(1)$-overlap). This is because $\supp \widehat{g_{P_{I_s}}}\subset (\cup_{I\subset I_s} (I-I)) \cap \{\xi: |\xi|\sim \frac{Lv}{s}\},$ and for every distinct $I_s,I'_s\in\mc{I}_s,$ and every $I,I'\in\mc{I}$ with $I\subset I_s, I'\subset I'_s,$ the distance $d_{I,I'}$ between the $\frac{L}{s}$-th term in $I$ and $I'$ satisfies
  \[ \frac{\theta L^2}{N^2} =\frac{s\theta L}{N^2} \frac{L}{s} \lesssim d_{I,I'}\lesssim \frac{N^{1/2}\theta}{N^2}\frac{L}{s} \lesssim \frac{1}{N}. \]
  Therefore 
  $\supp \widehat{\sum_{P_{I_s}} g_{P_{I_s}}}$ are $\mc{O}(1)$-overlapping.
  
  Hence
  \begin{align*}
      |Q_r|r^2 & \lessapprox \int_{E} g^2 \\
      & \lessapprox \int_{\R} |g_{s,M_s}|^2 \\
      & \lesssim \sum_{I_s}  \int_{\R} |\sum_{P_{I_s}} g_{P_{I_s}}|^2.
  \end{align*}
  
  We note that for $P\subset P_{I_s},$
  \[|v_P*\check{\eta}_s|\lesssim \frac{1}{s} W_{P_{I_s},100},\]
  so
  \[\int_{\R} | \sum_{P_{I_s}} g_{P_{I_s}}|^2\lesssim \int_{\R} (\sum_{P_{I_s}} \sum_{P\subset P_{I_s},\, P\in \mc{P}_{s,M_s}}  \frac{1}{s}  W_{P_{I_s},100} )^2 \lesssim \sum_{P_{I_s}} \frac{M_s^2}{s^2} |P_{I_s}|. \]
  Hence 
  \[|Q_r|r^2 \lessapprox \sum_{I_s} \sum_{P_{I_s}} |P_{I_s}| \left( \frac{M_s}{s}\right)^2.\]
  Since $\frac{|P_{I_s}|}{s}\sim |P|$ and $\sum_{I_s} \sum_{P_{I_s}} M_s \leq (\#P),$ we obtain
  \[r^2|Q_r|\lessapprox (\#P)|P| \frac{M_s}{s},\]
  which is \eqref{s103}.
  
  Now we show \eqref{s104}. We choose $x\in E.$ Then 
  \[r\lesssim g(x) \lessapprox |g_{s,M_s}(x)| \leq \sum_{I_s} \sum_{P_{I_s}} |g_{P_s}(x)| \lesssim |\mc{I}_s| \frac{M_s}{s} \lesssim \frac{N^{1/2}}{sL} \frac{M_s}{s} .\]
  
  Finally we prove \eqref{s105}. When $s\leq \frac{L}{L_1},$  every $P_{I_s}$ is contained in a single $P_J$ and therefore can contain $\lesssim M$  parallel $P\in \mc{P}.$ For every $P_{I_s},$  there are  $\lesssim s$ many $I\in \mc{I}$ such that there could exist $P_{I}$ such that $P_I\subset P_{I_s},$ so we conclude $P_{I_s}$ contain $\lesssim sM$ many $P\in \mc{P}.$ When $s\geq \frac{L}{L_1},$ every $P_{I_s}$ is contained in at most $s\frac{L_1}{L}$ many $P_J$ and therefore can contain $\lesssim sMs\frac{L_1}{L}$ many $P\in \mc{P}.$ Hence we obtain \eqref{s105}.
  
  \vspace{3mm} 
  \noindent 
  {\bf Case 2. } Suppose $g(x) \lessapprox |g_0(x)|$ for $x$ in a subset of $Q_r$ with measure $\gtrapprox |Q_r|.$ \eqref{s106} is trivial since $Q_r\subset P(L).$ To show \eqref{s107} we choose $x\in E.$ Then
  \[r\lesssim g(x) \lessapprox |g_0(x)| \lesssim (\#P)\frac{|P|}{|P(L)|}, \]
  where the last inequality is because
  \[|g_0(x)|=|g*\check{\eta}_0(x)|\leq \|g\|_{L^1} \|\check{\eta}_0\|_{L^\infty} \lesssim (\#P)|P| \frac{1}{|P(L)|}= (\#P)\frac{|P|}{|P(L)|}.\]
\end{proof}

\subsection{Refined flat decoupling for fat APs}
Next we have the following refined flat decoupling inequality for fat APs.

\begin{proposition}\label{flatprop}
  Suppose $2\leq q \leq p,$ and let $\{J\}_{J\in \mc{J}},$ $\{I\}_{I\in \mc{I}}$ be defined as in the beginning of Section \ref{smallcapsec}. Fix $I\in \mc{I}.$ Write $f_I=\sum_{P_I} f_{I,P_I}$ for the wave packet decomposition of $f_I.$ Suppose for non-zero wave packets $f_{I,P_I},$ $\|f_{I,P_I}\|_{L^\infty(\R)}$ are roughly constant, and for every $J\subset I,$ and every $P_J$ (in a tiling of $\R$), $P_J$ contains either $\sim M$ or $0$ wave packets $f_{I,P_I}$ (in the sense that $P_I \subset P_J$). Then
  \begin{equation}\label{s21}
    \|f_I\|_{L^p(\R)} \lesssim M^{\frac{1}{p}-\frac{1}{2}}
 \left(\frac{L}{L_1} \right)^{1-\frac{1}{p}-\frac{1}{q}}  \left(\sum_{J\subset I} \|f_J\|_{L^p(\R)}^q \right)^{1/q}.
  \end{equation}
\end{proposition}
\begin{proof}
  Fix a  $P_J$ that contains $\sim M$ many wave packets $f_{I,P_I}.$ We first show
  \begin{equation}\label{s22}
    \|f_I\|_{L^p(P_J)} \lesssim M^{\frac{1}{p}-\frac{1}{2}} \left(\frac{L}{L_1} \right)^{1-\frac{1}{p}-\frac{1}{q}} \left(\sum_{J\subset I} \|f_J\|_{L^p(W_{P_J,100})}^q \right)^{1/q}.
  \end{equation}
  Assume $\|f_{I,P_I}\|_{L^\infty(\R)} \sim H$ for every non-zero $f_{I,P_I}.$ By assumption we have
  \[\|f_I\|_{L^p(P_J)}\lesssim H(M|P_I|)^{1/p}.\]
  On the other hand by local $L^2$ orthogonality we have
  \[H(M|P_I|)^{1/2} \lesssim \|f_I\|_{L^2(P_J)} \lesssim (\sum_{J\subset I} \|f_J\|_{L^2(W_{P_J,100})}^2)^{1/2},\]
  and by H\"{o}lder's inequality the right hand side is bounded by
  \[(\frac{L}{L_1})^{\frac{1}{2}-\frac{1}{q}}|P_J|^{\frac{1}{2}-\frac{1}{p}} (\sum_{J\subset I} \|f_J\|_{L^p(W_{P_J,100})}^q)^{1/q}.\]
  Noting that $\frac{|P_I|}{|P_J|}=\frac{L_1}{L},$ we conclude
  \begin{align*}
      \|f_I\|_{L^p(P_J)} & \lesssim H(M|P_I|)^{\frac{1}{2}} (M|P_I|)^{\frac{1}{p}-\frac{1}{2}} \\
      & \lesssim M^{\frac{1}{p}-\frac{1}{2}} (\frac{L}{L_1} )^{1-\frac{1}{p}-\frac{1}{q}} (\sum_{J\subset I} \|f_J\|_{L^p(W_{P_J,100})}^q )^{1/q}.
  \end{align*}
  So \eqref{s22} holds.
  
  Since $q\leq p,$ \eqref{s21} follows from \eqref{s22} by raising \eqref{s22} to the $p$-th power, summing over $P_J$ in a tiling of $\R,$ and applying Minkovski's inequality (see Proposition \ref{localdecsec}).
\end{proof}

\subsection{Proof of Theorem \ref{smallcapthm}}
Now we are ready to prove Theorem \ref{smallcapthm}.
We first show a bilinear version of Theroem \ref{smallcapthm} and then conclude Theorem \ref{smallcapthm} by a broad-narrow argument. Still let $\{J\}_{J\in \mc{J}}$ be defined as in the beginning of Section \ref{smallcapsec}. We say two sub-collections of $\mc{J},$ $\mc{J}_1$ and $\mc{J}_2$, are transversal if $d(J_1,J_2)\gtrsim N^{-1/2}$ for every $J_1\in\mc{J}_1,J_2\in \mc{J}_2.$

\begin{theorem}\label{smallcapbilthm} 
  Suppose $4\leq q\leq p\leq 6,$ $\frac{1}{q}+\frac{3}{p}\leq 1.$
  If either of the following two conditions is satisfied
  \begin{enumerate}[label=(\alph*)]
      \item $L_1^{\frac{1}{2}-\frac{1}{q}} \leq L^{1-\frac{3}{p}-\frac{1}{q}},$ 
      \item $p=q,$
  \end{enumerate}
  then for every $\e>0,$
  \begin{multline}
      \label{s2}
    \|\prod_{i\in \{1,2\}}|\sum_{J\in \mc{J}_i}f_J|^{1/2}\|_{L^p(\R)} \lesssim_\e N^\e \log^C(\theta^{-1}+1) \left( \frac{N^{\frac{1}{2}-\frac{1}{2q}-\frac{3}{2p}}L^{\frac{2}{p}}}{L_1^{1-\frac{1}{p}-\frac{1}{q}}} +\left( \frac{N^{1/2}}{L_1} \right)^{\frac{1}{2}-\frac{1}{q}} \right)  
    \\
    \prod_{i\in \{1,2\}} ( \sum_{J\in \mc{J}_i} \|f_J\|_{L^p(\R)}^q )^{1/(2q)}
  \end{multline}
  for every  transversal sub-collections $\mc{J}_1,\mc{J}_2$  of $\mc{J},$ and every functions $f_J:\R\rightarrow \C$ with $\supp \widehat{f_J}\subset J.$
\end{theorem}

\begin{proof}
  By a local to global argument similar to Proposition \ref{localdecsec}, to show \eqref{s2} it suffices to show for every $P(L),$
  \begin{multline}
      \label{s2loc}
    \|\prod_{i\in \{1,2\}}|\sum_{J\in \mc{J}_i}f_J|^{1/2}\|_{L^p(P(L))} \lesssim_\e N^\e \log^C(\theta^{-1}+1) \left( \frac{N^{\frac{1}{2}-\frac{1}{2q}-\frac{3}{2p}}L^{\frac{2}{p}}}{L_1^{1-\frac{1}{p}-\frac{1}{q}}} +\left( \frac{N^{1/2}}{L_1} \right)^{\frac{1}{2}-\frac{1}{q}} \right) \\ 
    \prod_{i\in \{1,2\}} ( \sum_{J\in \mc{J}_i} \|f_J\|_{L^p(\R)}^q )^{1/(2q)}.
  \end{multline}
  
  We fix a $P(L).$
  Write $F_1=\sum_{J\in \mc{J}_1}f_J$ and $F_2=\sum_{J\in \mc{J}_2}f_J.$ For $i\in \{1,2\}$ We write $F_i=\sum_{P\in \mc{P}_i} F_{i,P}$ for the wave packet decomposition with respect to $\{I\}_{I\in \mc{I}}.$ So 
  \[F_i=\sum_{I\in \mc{I}}F_{i,I}=\sum_{I\in \mc{I}}\sum_{P_I}F_{i,I,P_I} =:\sum_{P\in\mc{P}_i} F_{i,P}.\]
  Write $\mc{I}_1=\{I\in \mc{I}: I\subset \cup_{J\in \mc{J}_1} J\}$ and  $\mc{I}_2=\{I\in \mc{I}: I\subset \cup_{J\in \mc{J}_2} J\}.$ Let $F=F_1+F_2.$
  
  By a dyadic pigeonholing argument and rescaling we may assume that for every non-zero $F_{i,P},$ $\|F_{i,P}\|_{L^\infty}\sim 1.$ We assume $\mc{P}_i$ contains only non-zero $F_{i,P}.$ By a further dyadic pigeonholing argument we may assume that
  for every $P_J$ (in a tiling of $\R$), $P_J$ either contains $M_i$ or $0$ many wave packets $F_{i,P},$ for $i\in \{1,2\}.$ Lastly, by one more dyadic pigeonholing argument we may assume that for each $i\in\{1,2\},$ $\|F_I\|_{L^p(\R)}$ are comparable for nonzero $F_I$ with $I\in \mc{I}_i.$
  
  For dyadic $1\leq r_1,r_2 \leq N^{1/2}/L$ we let $Q_{r_1,r_2}$ denote the collection of $P(L')$ (in the tiling of $P(L)$) that intersect $\sim r_1$ many $P\in \mc{P}_1,$ and $\sim r_2$ many $P\in \mc{P}_2.$ Recall that $L'=(N^{1/2}L)^{1/2}$ and  the square function $\sum_{I} |f_I|^2$ is locally constant on $P(L').$
  
  From the refined decoupling inequality (Theorem \ref{refineddecthm}) we have
  \[\|(F_1F_2)^{1/2}\|_{L^6(Q_{r_1,r_2})} \leq \|F_1\|_{L^6(Q_{r_1,r_2})}^{1/2}\|F_2\|_{L^6(Q_{r_1,r_2})}^{1/2} \lesssim_\e N^{\e} \log^C(\theta^{-1}+1) r_1^{1/6}r_2^{1/6} \prod_{i\in \{1,2\}} (\sum_{I\in \mc{I}_i} \int |F_I|^2)^{\frac{1}{12}}. \]
  On the other hand from bilinear restriction (Proposition \ref{bilresprop}) we have for every $P(L')\subset Q_{r_1,r_2}$
  \[\|(F_1F_2)^{1/2}\|_{L^4(P(L'))} \lesssim_\e N^\e r_1^{1/4}r_2^{1/4} |P(L')|^{1/4}\]
  and thus
  \[\|(F_1F_2)^{1/2}\|_{L^4(Q_{r_1,r_2})} \lesssim_\e N^\e r_1^{1/4}r_2^{1/4} |Q_{r_1,r_2}|^{1/4}.\]
  Therefore by the interpolation inequality we obtain
  \begin{equation}\label{s101}
      \|(F_1F_2)^{1/2}\|_{L^p(Q_{r_1,r_2})} \lesssim_\e N^\e \log^C(\theta^{-1}+1) r_1^{1/p}r_2^{1/p} |Q_{r_1,r_2}|^{\frac{3}{p}-\frac{1}{2}} \prod_{i\in \{1,2\}} (\sum_{I\in \mc{I}_i} \|F_I\|_{L^2}^2)^{\frac{1}{4}-\frac{1}{p}}.
  \end{equation}
  We assumed each wave packet $F_{i,P}$ satisfies $\|F_{i,P}\|_{L^\infty}=1,$ so
  \[\sum_{I\in \mc{I}_i} \|F_I\|_{L^2}^2 \sim (\#P_i)|P| \sim \sum_{I\in \mc{I}_i} \|F_I\|_{L^p}^p \]
  where $\#P_i$ denote the total number of nonzero wave packets in $F_i,$ that is, $|\mc{P}_i|.$ 
  Hence we may rewrite \eqref{s101} as
  \[ \|(F_1F_2)^{1/2}\|_{L^p(Q_{r_1,r_2})} \lesssim_\e N^\e \log^C(\theta^{-1}+1)  |Q_{r_1,r_2}|^{\frac{3}{p}-\frac{1}{2}} \prod_{i\in \{1,2\}} \Big( r_i^{\frac{2}{p}} (\sum_{I\in \mc{I}_i} \|F_I\|_{L^p}^q)^{1/q} ((\#P_i) |P|)^{\frac{1}{2}-\frac{3}{p}} (\#I_i)^{\frac{1}{p}-\frac{1}{q}} \Big)^{\frac{1}{2}},  \]
  where $\#I_i$ denotes the total number of $I\in \mc{I}_i$ such that $F_{I}$ is nonzero. By Proposition \ref{flatprop} we have
  \[\sum_{I\in \mc{I}_i} \|F_I\|_{L^p}^q \lesssim M_i^{\frac{q}{p}-\frac{q}{2}} \left(\frac{L}{L_1}\right)^{q-\frac{q}{p}-1} (\sum_{J\in \mc{J}_i} \|f_J\|^q_{L^p}). \]
  Therefore we conclude
  \begin{multline*}
      \|(F_1F_2)^{1/2}\|_{L^p(Q_{r_1,r_2})} \lesssim_\e N^\e \log^C(\theta^{-1}+1) |Q_{r_1,r_2}|^{\frac{3}{p}-\frac{1}{2}} \\
      \prod_{i\in \{1,2\}} \Big( r_i^{\frac{2}{p}}  ((\#P_i) |P|)^{\frac{1}{2}-\frac{3}{p}} (\#I_i)^{\frac{1}{p}-\frac{1}{q}}  M_i^{\frac{1}{p}-\frac{1}{2}} (\frac{L}{L_1})^{1-\frac{1}{p}-\frac{1}{q}} (\sum_{J\in \mc{J}_i} \|f_J\|^q_{L^p})^{\frac{1}{q}} \Big)^{\frac{1}{2}}.
  \end{multline*}
  So \eqref{s2} follows if we may show for $i\in \{1,2\},$
  \begin{equation}\label{s102}
       |Q_{r_1,r_2}|^{\frac{3}{p}-\frac{1}{2}} r_i^{\frac{2}{p}}  ((\#P_i) |P|)^{\frac{1}{2}-\frac{3}{p}} (\#I_i)^{\frac{1}{p}-\frac{1}{q}}  M_i^{\frac{1}{p}-\frac{1}{2}} \left(\frac{L}{L_1}\right)^{1-\frac{1}{p}-\frac{1}{q}} \lessapprox   \frac{N^{\frac{1}{2}-\frac{1}{2q}-\frac{3}{2p}}L^{\frac{2}{p}}}{L_1^{1-\frac{1}{p}-\frac{1}{q}}} +\left( \frac{N^{1/2}}{L_1} \right)^{\frac{1}{2}-\frac{1}{q}}.
  \end{equation}
  We show \eqref{s102} using Proposition \ref{incidenceprop}. Fix $i\in \{1,2\}.$ We split the proof into two cases depending on which case happens in Proposition \ref{incidenceprop} when applied to $\{P\}_{P\in \mc{P}_i}$ with $r=r_i.$
  
  \vspace{3mm}
  \noindent
  {\bf Case 1. (1) in Proposition \ref{incidenceprop} happens.} Let $s,M_s$ be the $s,M_s$ given in case (1) of Proposition \ref{incidenceprop}.
  By \eqref{s103} we have
  \[ \text{LHS of }  \eqref{s102} \lessapprox  r_i^{1-\frac{4}{p}}s^{\frac{1}{2}-\frac{3}{p}}M_s^{\frac{3}{p}-\frac{1}{2}}(\#I_i)^{\frac{1}{p}-\frac{1}{q}}M_i^{\frac{1}{p}-\frac{1}{2}}\left(\frac{L}{L_1}\right)^{1-\frac{1}{p}-\frac{1}{q}}.\]
  
  {\bf Case 1.1. 
  $s\leq \frac{L}{L_1}.$} Then \eqref{s105} reads $M_s \lesssim sM_i.$ Note that we have 
  \[(\#I) \gtrsim r_i\]
  since we have assumed $\|F_{i,P}\|_{L^\infty} \sim 1.$ Therefore by \eqref{s104} and \eqref{s105} we have
  \begin{align*}
      \text{LHS of }  \eqref{s102} &  \lessapprox \left(\frac{M_sN^{1/2}}{s^2L}\right)^{1-\frac{3}{p}-\frac{1}{q}}s^{\frac{1}{2}-\frac{3}{p}}M_s^{\frac{3}{p}-\frac{1}{2}}M_i^{\frac{1}{p}-\frac{1}{2}}\left(\frac{L}{L_1}\right)^{1-\frac{1}{p}-\frac{1}{q}} \\
      & = M_s^{\frac{1}{2}-\frac{1}{q}} \left(\frac{N^{1/2}}{L}\right)^{1-\frac{3}{p}-\frac{1}{q}} \left(\frac{L}{L_1}\right)^{1-\frac{1}{p}-\frac{1}{q}} s^{-\frac{3}{2}+\frac{3}{p}+\frac{2}{q}}M_i^{\frac{1}{p}-\frac{1}{2}} \\
      & \lesssim (sM_i)^{\frac{1}{2}-\frac{1}{q}} \left(\frac{N^{1/2}}{L}\right)^{1-\frac{3}{p}-\frac{1}{q}} \left(\frac{L}{L_1}\right)^{1-\frac{1}{p}-\frac{1}{q}} s^{-\frac{3}{2}+\frac{3}{p}+\frac{2}{q}}M_i^{\frac{1}{p}-\frac{1}{2}} \\
      &= M_i^{\frac{1}{p}-\frac{1}{q}} s^{-1+\frac{3}{p}+\frac{1}{q}} \left(\frac{N^{1/2}}{L}\right)^{1-\frac{3}{p}-\frac{1}{q}} \left(\frac{L}{L_1}\right)^{1-\frac{1}{p}-\frac{1}{q}}.
  \end{align*}
  
  Since $p\geq q, \frac{1}{q}+\frac{3}{p}\leq 1,$ and $s,M_i\geq 1,$ we conclude 
  \begin{align*}
      \text{LHS of }  \eqref{s102} &  \lessapprox \left(\frac{N^{1/2}}{L}\right)^{1-\frac{3}{p}-\frac{1}{q}} \left(\frac{L}{L_1}\right)^{1-\frac{1}{p}-\frac{1}{q}}  \\
      & = \frac{N^{\frac{1}{2}-\frac{1}{2q}-\frac{3}{2p}}L^{\frac{2}{p}}}{L_1^{1-\frac{1}{p}-\frac{1}{q}}}.
  \end{align*}
  
  {\bf Case 1.2. 
  $s\geq \frac{L}{L_1}.$} 
  This is the case where we see the two conditions in Theroem \ref{smallcapbilthm}.
  Now \eqref{s105} reads $M_s \lesssim s^2M_i\frac{L_1}{L}.$ 
  By $(\#I) \gtrsim r_i$ and \eqref{s104} we have
  \begin{align*}
     \text{LHS of }  \eqref{s102} & \lessapprox r_i^{1-\frac{3}{p}-\frac{1}{p}}  s^{\frac{1}{2}-\frac{3}{p}}M_s^{\frac{3}{p}-\frac{1}{2}}M_i^{\frac{1}{p}-\frac{1}{2}}\left(\frac{L}{L_1}\right)^{1-\frac{1}{p}-\frac{1}{q}}
     \\ &  \lesssim \left(\frac{M_s}{s^2}\frac{N^{1/2}}{L} \right)^{1-\frac{3}{p}-\frac{1}{q}}s^{\frac{1}{2}-\frac{3}{p}}M_s^{\frac{3}{p}-\frac{1}{2}}M_i^{\frac{1}{p}-\frac{1}{2}}\left(\frac{L}{L_1}\right)^{1-\frac{1}{p}-\frac{1}{q}}  \\
     & = M_s^{\frac{1}{2}-\frac{1}{q}} \left(\frac{N^{1/2}}{L}\right)^{1-\frac{3}{p}-\frac{1}{q}} \left(\frac{L}{L_1}\right)^{1-\frac{1}{p}-\frac{1}{q}} s^{-\frac{3}{2}+\frac{3}{p}+\frac{2}{q}}M_i^{\frac{1}{p}-\frac{1}{2}}
  \end{align*}
  Plugging in \eqref{s105} we obtain
  \begin{align*}
      \text{LHS of }  \eqref{s102} & \lessapprox \left(s^2M_i \frac{L_1}{L}\right)^{\frac{1}{2}-\frac{1}{q}} \left(\frac{N^{1/2}}{L}\right)^{1-\frac{3}{p}-\frac{1}{q}} \left(\frac{L}{L_1}\right)^{1-\frac{1}{p}-\frac{1}{q}} s^{-\frac{3}{2}+\frac{3}{p}+\frac{2}{q}}M_i^{\frac{1}{p}-\frac{1}{2}} \\
      & =M_i^{\frac{1}{p}-\frac{1}{q}} s^{-\frac{1}{2}+\frac{3}{p}} \left(\frac{N^{1/2}}{L} \right)^{1-\frac{3}{p}-\frac{1}{q}} \left(\frac{L}{L_1}\right)^{\frac{1}{2}-\frac{1}{p}}.
  \end{align*}
  Since $M_i\geq 1$ and $q\leq p,$ we conclude
  \[\text{LHS of }  \eqref{s102}  \lessapprox s^{-\frac{1}{2}+\frac{3}{p}} \left(\frac{N^{1/2}}{L} \right)^{1-\frac{3}{p}-\frac{1}{q}} \left(\frac{L}{L_1}\right)^{\frac{1}{2}-\frac{1}{p}}.\]
  
  If we use $s\leq L,$ then
  \begin{align*}
      s^{-\frac{1}{2}+\frac{3}{p}} \left(\frac{N^{1/2}}{L} \right)^{1-\frac{3}{p}-\frac{1}{q}} \left(\frac{L}{L_1}\right)^{\frac{1}{2}-\frac{1}{p}} & \leq L^{-\frac{1}{2}+\frac{3}{p}} \left(\frac{N^{1/2}}{L} \right)^{1-\frac{3}{p}-\frac{1}{q}} \left(\frac{L}{L_1}\right)^{\frac{1}{2}-\frac{1}{p}} 
  \end{align*}
  We may then verify that
  \[L^{-\frac{1}{2}+\frac{3}{p}} \left(\frac{N^{1/2}}{L} \right)^{1-\frac{3}{p}-\frac{1}{q}} \left(\frac{L}{L_1}\right)^{\frac{1}{2}-\frac{1}{p}} \leq \frac{N^{\frac{1}{2}-\frac{1}{2q}-\frac{3}{2p}}L^{\frac{2}{p}}}{L_1^{1-\frac{1}{p}-\frac{1}{q}}}\]
  if and only if 
  \[L_1^{\frac{1}{2}-\frac{1}{q}} \leq L^{1-\frac{3}{p}-\frac{1}{q}}.\]
  
  On the other hand if we use $s\leq \frac{N^{1/2}}{L},$
  then
  \begin{align*}
     \text{LHS of }  \eqref{s102} & \lessapprox \left(\frac{N^{1/2}}{L} \right)^{-\frac{1}{2}+\frac{3}{p}} \left(\frac{N^{1/2}}{L} \right)^{1-\frac{3}{p}-\frac{1}{q}} \left(\frac{L}{L_1}\right)^{\frac{1}{2}-\frac{1}{p}} \\ &
     =\left(\frac{N^{1/2}}{L} \right)^{\frac{1}{2}-\frac{1}{q}} \left(\frac{L}{L_1}\right)^{\frac{1}{2}-\frac{1}{p}}.
  \end{align*}
  The last line equals to 
  \[\left(\frac{N^{1/2}}{L} \right)^{\frac{1}{2}-\frac{1}{p}} \left(\frac{L}{L_1}\right)^{\frac{1}{2}-\frac{1}{p}}\]
  if $p=q.$
  
  In conclusion we have shown \eqref{s102} holds in this case if either condition (a) or (b) is satisfied.

  \vspace{3mm}
  \noindent
  {\bf Case 2. (2) in Proposition \ref{incidenceprop} happens. } 
  By \eqref{s106}, \eqref{s107} we have 
  \begin{equation}\label{s110}
      \text{LHS of }  \eqref{s102} \lessapprox  |P(L)|^{\frac{3}{p}-\frac{1}{2}} \left(\frac{(\#P_i)|P|}{|P(L)|}\right)^{\frac{2}{p}}(\#I_i)^{\frac{1}{p}-\frac{1}{q}}M_i^{\frac{1}{p}-\frac{1}{2}}\left(\frac{L}{L_1}\right)^{1-\frac{1}{p}-\frac{1}{q}} ((\#P_i)|P|)^{\frac{1}{2}-\frac{3}{p}}.
  \end{equation}
   Note that we have 
   \[(\#P_i) \lesssim (\#I_i)M_i \frac{|P(L)|}{|P_J|}\sim (\#I_i)M_i \frac{|P(L)|}{|P|}\frac{L_1}{L}\]
   since the right hand side is the maximal number of $P$ one can fit into a $P(L)$ under the assumption that each $P_J$ can contain $\lesssim M_i$ many $P\in \mc{P}_i.$ Substituting the above for $M_i$ in \eqref{s110} and simplifying the algebra we obtain
   \[ \text{LHS of }  \eqref{s102} \lessapprox (\#I_i)^{\frac{1}{2}-\frac{1}{q}} \left( \frac{L}{L_1} \right)^{\frac{1}{2}-\frac{1}{q}}.\]
   Since $\#I_i\leq \frac{N^{1/2}}{L}$ and $q\geq 2,$ we conclude 
   \[\text{LHS of }  \eqref{s102} \lessapprox \left( \frac{N^{1/2}}{L_1} \right)^{\frac{1}{2}-\frac{1}{q}}.\]
   Hence \eqref{s102} holds in this case.
   
   In conclusion we have shown \eqref{s102} and therefore \eqref{s2loc} and \eqref{s2}.
\end{proof}


\begin{proof}[Proof of Theorem \ref{smallcapthm} using Theorem \ref{smallcapbilthm}]
  The proof resembles Section 5.1 in \cite{demeter2020small}. First we fix $(p,q)$ with $4\leq p\leq 6,$ and either  $\frac{1}{q}+\frac{3}{p}=1$ or $p=q.$ Note that under such assumption we always have $p\geq q$ and $q\geq 2.$
  
  Recall that $\Omega$ is the $\theta L^2/N^2$-neighborhood of $\{a_n\}_{n=1}^{N^{1/2}},$ which is a union of $N^{1/2}$ many intervals of length $C\theta L^2/N^2.$ We let $\tau$ denote the union of $l$ many consecutive intervals in $\Omega,$ and write $\ell(\tau)=l,$ so in this notation $\ell(I)=L$ and $\ell(J)=L_1.$
  Let $F=\sum_{J\in \mc{J}} f_J,$ and denote by $F_\tau$ the Fourier projection of $F$ to $\tau,$ that is, $(1_{\tau}\hat{F})\,\check{\, }.$
  
  Fix $K>1.$  We have the following inequality
  \[|F(x)| \leq \sum_{\ell(\tau)=\frac{N^{1/2}}{K}}|F_{\tau}(x)| \leq C\max_{\ell(\tau)=\frac{N^{1/2}}{K}}|F_\tau(x)| + K^C \max_{\substack{ \ell(\tau_1)=\ell(\tau_2)=\frac{N^{1/2}}{K},\\ d(\tau_1,\tau_2) \gtrsim \frac{1}{KN^{1/2}}}} |F_{\tau_1}F_{\tau_2}|^{1/2}.\]
  Iterating this (for the first term) we obtain
  \begin{multline}\label{ss1}
      \|F\|^p_{L^p(\R)} \lesssim C^m \sum_{\ell(\tau)=L} \|F_{\tau}\|^p_{L^p(\R)} \\
      +C^mK^C \sum_{\substack{l=\frac{N^{1/2}}{K^a} \text{ for } a\in \Z,  \\ KL \leq l\leq N^{1/2}} } \sum_{\tau: \ell(\tau)=l} \sum_{\substack{\tau_1,\tau_2\subset \tau, \\ \ell(\tau_1)=\ell(\tau_2)=K^{-1}l, \\ d(\tau_1,\tau_2) \gtrsim K^{-1} l}} \|(F_{\tau_1}F_{\tau_2})^{1/2}\|^p_{L^p(\R)}.
  \end{multline}
  Here $m$ satisfies $N^{1/2}/K^m=L.$ 
  
  By Proposition \ref{flatprop} we have
  \[\sum_{\ell(\tau)=L} \|F_{\tau}\|^p_{L^p(\R)} \lesssim \sum_{\ell(\tau)=L} \left(\frac{L}{L_1} \right)^{\frac{p}{2}-\frac{p}{q}} \left( \sum_{J\subset \tau}\|F_J\|^q_{L^p(\R)}\right)^{\frac{p}{q}}. \]
  Since $\frac{p}{q}\geq 1,$ we obtain
  \begin{equation}\label{sss1}
      \sum_{\ell(\tau)=L} \|F_{\tau}\|^p_{L^p(\R)} \lesssim  \left(\frac{L}{L_1} \right)^{\frac{p}{2}-\frac{p}{q}} \left( \sum_{J\in \mc{J}}\|F_J\|^q_{L^p(\R)}\right)^{\frac{p}{q}} \leq \left(\frac{N^{1/2}}{L_1} \right)^{\frac{p}{2}-\frac{p}{q}} \left( \sum_{J\in \mc{J}}\|F_J\|^q_{L^p(\R)}\right)^{\frac{p}{q}}.
  \end{equation}

  Now we estimate the second term on the right hand side of \eqref{ss1}. Let $s=\frac{N^{1/2}}{l}.$ Then using the change of variable $x\mapsto s^2x$ as in the proof of Proposition \ref{narrowprop}, and by Theorem \ref{smallcapbilthm} we have
  \[\|(F_{\tau_1}F_{\tau_2})^{1/2}\|_{L^p(\R)} \lesssim_\e N^\e \log^C(\tilde{\theta}^{-1}+1) \left( \frac{\tilde{N}^{\frac{1}{2}-\frac{1}{2q}-\frac{3}{2p}}\tilde{L}^{\frac{2}{p}}}{\tilde{L_1}^{1-\frac{1}{p}-\frac{1}{q}}} +\left( \frac{\tilde{N}^{1/2}}{\tilde{L_1}} \right)^{\frac{1}{2}-\frac{1}{q}} \right)  ( \sum_{J\subset \tau} \|f_J\|_{L^p(\R)}^q )^{1/q},\]
  where $\tilde{N}=\frac{N}{s^2},$ $\tilde{\theta}=\frac{\theta}{s^2},$ $\tilde{L_1}=L_1,$ $\tilde{L}=L.$
  Plugging in the expressions for $\tilde{N},\tilde{\theta},\tilde{L_1},\tilde{L}$ we obtain
  \begin{multline}\label{sss2}
       \|(F_{\tau_1}F_{\tau_2})^{1/2}\|_{L^p(\R)} \lesssim_\e N^\e \log^C({\theta}^{-1}+1)  \left(s^{-1+\frac{1}{q}+\frac{3}{p}} \frac{N^{\frac{1}{2}-\frac{1}{2q}-\frac{3}{2p}}L^{\frac{2}{p}}}{L_1^{1-\frac{1}{p}-\frac{1}{q}}} +s^{-\frac{1}{2}+\frac{1}{q}} \left( \frac{N^{1/2}}{L_1} \right)^{\frac{1}{2}-\frac{1}{q}} \right) \\ 
       ( \sum_{J\subset \tau} \|f_J\|_{L^p(\R)}^q )^{1/q}.
  \end{multline}
  
  We let $K=N^{\e'}$ for some $\e'>0$ which will be chosen depending on $\e.$ Then from \eqref{sss1} and \eqref{sss2} we conclude
  \begin{align*}
      \|F\|_{L^p(\R)} & \lesssim_{\e,\e'} N^{\e+C\e'} \log^C(\theta^{-1}+1) \Bigg( (\sum_{\substack{s=K^a \text{ for } a\in \Z   \\ 1\leq s\leq \frac{N^{1/2}}{KL}  }} s^{-1+\frac{1}{q}+\frac{3}{p}}) \frac{N^{\frac{1}{2}-\frac{1}{2q}-\frac{3}{2p}}L^{\frac{2}{p}}}{L_1^{1-\frac{1}{p}-\frac{1}{q}}}  \\ 
      & \quad 
      + (\sum_{\substack{s=K^a \text{ for } a\in \Z  \\ 1\leq s\leq \frac{N^{1/2}}{KL}  }} s^{-\frac{1}{2}+\frac{1}{q}}) \left( \frac{N^{1/2}}{L_1} \right)^{\frac{1}{2}-\frac{1}{q}}  \Bigg) (\sum_{J\in \mc{J}} \|f_J\|_{L^p(\R)}^q )^{1/q} \\
      & \lesssim_{\e, \e'} N^{\e+C\e'} \log^C(\theta^{-1}+1) \left( \frac{N^{\frac{1}{2}-\frac{1}{2q}-\frac{3}{2p}}L^{\frac{2}{p}}}{L_1^{1-\frac{1}{p}-\frac{1}{q}}} + \left( \frac{N^{1/2}}{L_1} \right)^{\frac{1}{2}-\frac{1}{q}}  \right) (\sum_{J\in \mc{J}} \|f_J\|_{L^p(\R)}^q )^{1/q}.
  \end{align*}
  
  Therefore we have shown Theorem \ref{smallcapthm} under condition (a) and the extra condition $\frac{1}{q}+\frac{3}{p}=1,$ $p\leq 6,$ or under condition (b) with the extra condition $p\leq 6.$
  
  First assume (a) and we want to remove the condition $\frac{1}{q}+\frac{3}{p}=1,$ $p\leq 6.$
  First we note that it suffices to show \eqref{s1} for every $(p,q)$ with $p\geq 4,$ $\frac{1}{q}+\frac{3}{p}=1.$ This is because for a general $(p,q)$ with $p\geq 4,$ $\frac{1}{q}+\frac{3}{p}\leq 1$ we may consider \eqref{s1} with $(p,q)$ replaced by $(p,q_0)$ where $\frac{1}{q_0}+\frac{3}{p}=1.$ Then \eqref{s1} with $(p,q)$ follows from H\"{o}lder's inequality applied in the index $J$ to the right hand side of \eqref{s1} with $(p,q_0),$ since $|\mc{J}| \lesssim \frac{N^{1/2}}{L_1}.$ 
  Second we note that it suffices to show \eqref{s1} for every $(p,q)$ with $4\leq p\leq 6,$ $\frac{1}{q}+\frac{3}{p}=1.$ This is because when $p\geq 6,$ we always have
  \[\frac{N^{\frac{1}{2}-\frac{1}{2q}-\frac{3}{2p}}L^{\frac{2}{p}}}{L_1^{1-\frac{1}{p}-\frac{1}{q}}} \geq \left( \frac{N^{1/2}}{L_1} \right)^{\frac{1}{2}-\frac{1}{q}}\]
  and \eqref{s1} reduces to
  \[\|\sum_{J\in \mc{J}}f_J\|_{L^p(\R)} \lesssim_\e N^\e \log^C(\theta^{-1}+1)  \frac{N^{\frac{1}{2}-\frac{1}{2q}-\frac{3}{2p}}L^{\frac{2}{p}}}{L_1^{1-\frac{1}{p}-\frac{1}{q}}} \left( \sum_{J\in \mc{J}} \|f_J\|_{L^p(\R)}^q \right)^{1/q}.\]
  So \eqref{s1} with $q>6,$ $\frac{1}{q}+\frac{3}{p}=1$ follows from interpolating \eqref{s1} with $(p,q)=(6,2),$ and with $(p,q)=(\infty,1).$ When $p=\infty,q=1,$ \eqref{s1} becomes the triangle inequality which holds trivially. Hence we have shown Theorem \ref{smallcapthm} under condition (a).
  
  Now assume (b) and we want to remove the condition $p\leq 6.$ As in the previous paragraph, when $p\geq 6$ we always have 
  \[\frac{N^{\frac{1}{2}-\frac{1}{2p}-\frac{3}{2p}}L^{\frac{2}{p}}}{L_1^{1-\frac{1}{p}-\frac{1}{p}}} \geq \left( \frac{N^{1/2}}{L_1} \right)^{\frac{1}{2}-\frac{1}{p}},\]
  and therefore \eqref{s1} with $q>6,$ $p=q$ follows from interpolating \eqref{s1} with $(p,q)=(6,6),$ and with $(p,q)=(\infty,\infty).$ So Theorem \ref{smallcapthm} holds under condition (b) as well.
\end{proof}

\section{Appendix} \label{appsec}
Corollary \ref{discretesmallcor_intro} can be derived from small-cap decoupling inequalities for the parabola in \cite{demeter2020small}.    This is through a transference method which we learned from James Maynard. We record a detailed proof here. The same argument would also imply Corollary \ref{discretesmallcor} if the corresponding $\ell^q L^p$ small cap decoupling inequalities for the parabola are known. 

We first recall the small-cap decoupling inequalities in \cite{demeter2020small}.
\begin{theorem}[\cite{demeter2020small}]\label{smallcapparathm}
  Suppose $\alpha\in [\frac{1}{2},1],$ and let $\Gamma=\{\gamma\}$ be the partition of $\mc{N}_{R^{-1}}(\mb{P}^1)$ into $R^{\alpha}$ many $R^{-\alpha}\times R^{-1}$ rectangles $\gamma.$ Assume $p=2+\frac{2}{\alpha}.$ Then for every $\e>0$ we have
  \begin{equation}\label{smallcapparaeqn}
      \|\sum_{\gamma\in \Gamma} f_\gamma \|_{L^p(\R^2)} \lesssim_\e R^{\alpha(\frac{1}{2}-\frac{1}{p})+\e}(\sum_{\gamma} \|f_\gamma\|_{L^p(\R^2)}^p)^{\frac{1}{p}}
  \end{equation}
  for every $f_\gamma:\R^2\rightarrow \C$ with $\supp \widehat{f_\gamma} \subset \gamma.$
\end{theorem}
Theorem \ref{smallcapparathm} continues to hold, by essentially the same proof, with $\mb{P}^1$ replaced by a $C^2$ curve of the form $\{(x,g(x)):x\in [0,1]\}$ with $g'(0)=0,$ $g''(x)\sim 1$ for $x\in [0,1].$  See Appendix of \cite{guth2020improved}. Additionally we may interpolate between \eqref{smallcapparaeqn} and the elementary inequalities
\[\|\sum_{\gamma\in \Gamma} f_\gamma \|_{L^2(\R^2)} \lesssim(\sum_{\gamma} \|f_\gamma\|_{L^2(\R^2)}^2)^{\frac{1}{2}}\]
\[\|\sum_{\gamma\in \Gamma} f_\gamma \|_{L^\infty(\R^2)} \lesssim R^{\alpha}(\sup_{\gamma} \|f_\gamma\|_{L^\infty(\R^2)})\]
to obtain the following version of Theorem \ref{smallcapparathm}.

\begin{theorem}[\cite{demeter2020small}]\label{smallcapparathm2}
  Suppose $G$ is a $C^2$ convex curve of the form $\{(x,g(x)):x\in [0,1]\}$ where $g'(0)=0,$ $g''(x)\sim 1$ for $x\in [0,1].$ Suppose $\alpha\in [\frac{1}{2},1],$ and let $\Gamma=\{\gamma\}$ be the partition of $\mc{N}_{R^{-1}}(G)$ into $R^{\alpha}$ many $R^{-\alpha}\times R^{-1}$ rectangles $\gamma.$ Assume $p\geq 2.$ Then for every $\e>0$ we have
  \begin{equation}\label{smallcapparaeqn2}
      \|\sum_{\gamma\in \Gamma} f_\gamma \|_{L^p(\R^2)} \lesssim_\e R^\e (R^{\alpha(\frac{1}{2}-\frac{1}{p})}+R^{\alpha(1-\frac{1}{p})-(1+\alpha)\frac{1}{p}})(\sum_{\gamma} \|f_\gamma\|_{L^p(\R^2)}^p)^{\frac{1}{p}}
  \end{equation}
  for every $f_\gamma:\R^2\rightarrow \C$ with $\supp \widehat{f_\gamma} \subset \gamma.$
\end{theorem}

For the rest of this section we work under the assumption of Corollary \ref{discretesmallcor_intro}. In particular $\theta=1.$ For simplicity we assume $a_1=0,$ and $v:=a_2-a_1=N^{-1}.$ Let $1\leq L \leq N^{1/2}.$  
It suffices to show \eqref{s120} for $4\leq p\leq 6$ and we assume that (since the $p>6$ case follows from interpolating between $p=6$ and $p=\infty$).


By \eqref{defa_n} we may write $a_n=\frac{n-1}{N}+e_n$ where $e_n=a_n-\frac{n-1}{N} \sim \frac{(n-1)^2}{N^2}.$ For every $t\in \R$ we may write it as $t_1+t_2$ where $t_1 \in 2\pi N\Z$ and $t_2\in [0,2\pi N).$ Without loss of generality we assume $2\pi N$ divides $T,$ so $(2\pi)^{-1}N^{-1} T \in \Z.$  Now we may write
  \begin{align*}
      \int_{0}^T |\sum_{n=1}^{N^{1/2}} b_n e^{ita_n}|^p dt & = \sum_{t_1\in 2\pi N\Z \cap [0,T-2\pi N]} \int_{0}^{2\pi N}  |\sum_{n=1}^{N^{1/2}} b_n e^{i(t_1+t_2)(\frac{n-1}{N}+e_n)}|^p dt_2 \\
      & = \sum_{t_1\in 2\pi N\Z \cap [0,T-2\pi N]} \int_{0}^{2\pi N}  |\sum_{n=1}^{N^{1/2}} b_n e^{i(t_1 e_n+t_2\frac{n-1}{N}+t_2e_n)}|^p dt_2.
  \end{align*}
  We write $e(n)=e_n$ and let $e:[1,N^{1/2}]\rightarrow \R$ be the piece-wise linear function such that for every $n\in \Z \cap [1,N^{1/2}-1],$ $e(x)$ is linear on $[n,n+1]$ and $e(n)=e_n.$ Since $e_{n+1}-e_n \sim \frac{n}{N^2},$ we have $|e'(x)|\lesssim \frac{1}{N^{3/2}}$ for $x\in [1,N^{1/2}]\setminus \Z.$
  
  By Abel's summation formula we have 
  \begin{align}
      |\sum_{n=1}^{N^{1/2}} b_n e^{i(t_1 e_n+t_2\frac{n-1}{N}+t_2e_n)}| & \leq |\sum_{n=1}^{N^{1/2}} b_n e^{i(t_1 e_n+t_2\frac{n-1}{N}})|+ \int_{0}^{N^{1/2}} |\sum_{n=1}^u  b_n e^{i(t_1 e_n+t_2\frac{n-1}{N})}| |t_2e'(u)|du \nonumber \\
      & \lesssim |\sum_{n=1}^{N^{1/2}} b_n e^{i(t_1 e_n+t_2\frac{n-1}{N}})|+ \frac{1}{N^{1/2}}\int_{0}^{N^{1/2}} |\sum_{n=1}^u  b_n e^{i(t_1 e_n+t_2\frac{n-1}{N})}|du. \label{s133}
  \end{align}
  The last inequality uses $t_2\lesssim N.$
  
  We first estimate 
  \[A:=\sum_{t_1\in 2\pi N\Z \cap [0,T-2\pi N]} \int_{0}^{2\pi N}  |\sum_{n=1}^{N^{1/2}} b_n e^{i(t_1 e_n+t_2\frac{n-1}{N})}|^p dt_2.\]
  Since $e_n\lesssim \frac{1}{N}$ for  every $1\leq n \leq N^{1/2},$  $\sum_{n=1}^{N^{1/2}} b_n e^{i(t_1 e_n+t_2\frac{n-1}{N})}$ is locally constant on intervals of length $N$ in $t_1$ (in the sense of Proposition \ref{locconstprop}). Therefore we have
  \begin{equation}\label{s131}
      A \lesssim \frac{1}{N} \int_\R \int_{0}^{2\pi N}  |\sum_{n=1}^{N^{1/2}} b_n e^{i(t_1 e_n+t_2\frac{n-1}{N})}|^p  dt_2 W_{[0,T],100}(t_1) dt_1.
  \end{equation}
  Recall that here $W_{[0,T],100}(t_1)$ means a weight function adapted to $[0,T]$ with decay rate $100.$ 
  
  We consider two cases, $T\geq N^{3/2}$ and $T\leq N^{3/2}.$
  
  \vspace{3mm}
  \noindent
  {\bf Case 1. $T\geq N^{3/2}.$ }
   We observe that  $\sum_{n=1}^{N^{1/2}} b_n e^{i(t_1 e_n+t_2\frac{n-1}{N})}$ is $2\pi N$-periodic in $t_2,$ so  we have
  \begin{align*}
      A & \lesssim \frac{1}{N}\frac{N^{3/2}}{T} \int_\R \int_{0}^{TN^{-1/2}}  |\sum_{n=1}^{N^{1/2}} b_n e^{i(t_1 e_n+t_2\frac{n-1}{N})}|^p dt_2 W_{[0,T],100}(t_1) dt_1 
  \end{align*}
  By a change of variable $t_1 \mapsto N^{1}t_1,$ $t_2\mapsto N^{1/2}t_2,$ we obtain
  \[A \lesssim N^{1/2}\frac{N^{3/2}}{T} \int_\R \int_\R |\sum_{n=1}^{N^{1/2}} b_n e^{i(t_1 e_nN+t_2\frac{n-1}{N^{1/2}})}|^p W_{B_{TN^{-1}}(0),100}(t_1,t_2) dt_2  dt_1.\]
  
  Now we let $g(x)$ be a $C^2$ strictly convex function defined on $[0,1]$ such that $|g(\frac{n-1}{N^{1/2}})-e_nN|\leq N^{-1}/4$ for $n=1,\ldots,N^{1/2}.$ (See Lemma \ref{quadlem} below.) Since $N^{-1} \leq T^{-1}N,$ we have for every $n,$ the ball of radius $T^{-1}N/4$ centered at $(\frac{n-1}{N^{1/2}},e_nN)$ fits in exactly one of the $\gamma$ in the partition of the $T^{-1}N$ neighborhood of $G=\{(x,g(x)):x\in [0,1]\}$ by $N^{-1/2} \times T^{-1}N$ rectangles. Under our assumption that $T\in [N^{3/2},N^2]$ we have $\frac{\log (N^{-1/2})}{\log (T^{-1}N)} \in [\frac{1}{2},1].$ Therefore we may apply Theorem \ref{smallcapparathm2} with $R=TN^{-1},$ $R^\alpha=N^{1/2}$ to the curve $G,$ which yields for every $T\in [N^{3/2},N^2],$
  \begin{multline}\label{s132}
       \int_\R \int_\R |\sum_{n=1}^{N^{1/2}} b_n e^{i(t_1 e_nN+t_2\frac{n-1}{N^{1/2}})}|^p W_{B_{TN^{-1}}(0),100}(t_1,t_2) dt_2  dt_1 \\
       \lesssim_\e N^\e \left(T^{\frac{1}{p}}N^{\frac{1}{2}-\frac{2}{p}}+T^{\frac{2}{p}} N^{\frac{1}{4}-\frac{5}{2p}} \right)^p \|b_n\|_{\ell^p}^p.
  \end{multline}
  
  Hence 
  \[A\lesssim_\e N^\e  \left(N^{\frac{1}{2}}+T^{\frac{1}{p}} N^{\frac{1}{4}-\frac{1}{2p}} \right)^p \|b_n\|_{\ell^p}^p.\]
  
  \vspace{3mm}
  \noindent
  {\bf Case 2. $T\leq N^{3/2}.$ } From \eqref{s131} and a change of variable we have
  \[A \lesssim N^{1/2} \int_\R \int_{0}^{2\pi N^{1/2}}  |\sum_{n=1}^{N^{1/2}} b_n e^{i(t_1 e_nN+t_2\frac{n-1}{N^{1/2}})}|^p  dt_2 W_{[0,TN^{-1}],100}(t_1) dt_1.\]
  Since $T\leq N^{3/2},$ we may bound the right hand side  trivially by
  \[N^{1/2} \int_\R \int_\R  |\sum_{n=1}^{N^{1/2}} b_n e^{i(t_1 e_nN+t_2\frac{n-1}{N^{1/2}})}|^p   W_{B_{N^{1/2}}(0),100}(t_1,t_2) dt_2 dt_1,\]
  so by \eqref{s132} with $T=N^{3/2}$ we have
  \[A \lesssim_\e N^\e N^{\frac{1}{2}} \left(N^{\frac{3}{2p}}N^{\frac{1}{2}-\frac{2}{p}}+N^{\frac{3}{2}\frac{2}{p}} N^{\frac{1}{4}-\frac{5}{2p}} \right)^p \|b_n\|_{\ell^p}^p. \]
  Since $p\geq 4$ we may verify
  \[N^{\frac{3}{2p}}N^{\frac{1}{2}-\frac{2}{p}}\geq N^{\frac{3}{2}\frac{2}{p}} N^{\frac{1}{4}-\frac{5}{2p}}.\]
  Hence 
  \[A\lesssim_\e N^\e \left(N^{\frac{1}{2}} \right)^p \|b_n\|_{\ell^p}^p.\]
  
  In conclusion we have shown 
  \begin{equation}\label{s134}
      A\lesssim_\e N^\e  \left(N^{\frac{1}{2}}+T^{\frac{1}{p}} N^{\frac{1}{4}-\frac{1}{2p}} \right)^p \|b_n\|_{\ell^p}^p.
  \end{equation}
  
  Next we estimate the second term in \eqref{s133}. We define 
  \[B:= \sum_{t_1\in 2\pi N\Z \cap [0,N^2/L^2-2\pi N]} \int_{0}^{2\pi N}  \left|\frac{1}{N^{1/2}}\int_{0}^{N^{1/2}} |\sum_{n=1}^u  b_n e^{i(t_1 e_n+t_2\frac{n-1}{N}})|du\right|^p dt_2.\]
  By Minkowski's inequality we have
  \begin{align*}
      B^{\frac{1}{p}} & \leq  \frac{1}{N^{1/2}}\int_{0}^{N^{1/2}} \bigg(\sum_{t_1\in 2\pi N\Z \cap [0,T-2\pi N]} \int_{0}^{2\pi N}   |\sum_{n=1}^u  b_n e^{i(t_1 e_n+t_2\frac{n-1}{N})}|^p dt_2 \bigg)^{\frac{1}{p}} du.
  \end{align*}
  Then applying \eqref{s134} to the expression in the brackets we obtain
  \begin{align*}
      B^{\frac{1}{p}} & \lesssim_\e N^\e \frac{1}{N^{1/2}}\int_{0}^{N^{1/2}} \left(N^{\frac{1}{2}}+T^{\frac{1}{p}} N^{\frac{1}{4}-\frac{1}{2p}} \right) \|b_n\|_{\ell^p} du \\
      & = N^\e  \left(N^{\frac{1}{2}}+T^{\frac{1}{p}} N^{\frac{1}{4}-\frac{1}{2p}} \right) \|b_n\|_{\ell^p}.
  \end{align*}
  
  Combining the estimates for $A$ and $B$ we conclude
  \[\|\sum_{n=1}^{N^{1/2}}b_n e^{ita_n}\|_{L^p(B_T)} \lesssim_\e  N^{\e}  \left(N^{\frac{1}{2}}+T^{\frac{1}{p}} N^{\frac{1}{4}-\frac{1}{2p}} \right) \|b_n\|_{\ell^p}.\]



We used the following lemma in the proof above.
\begin{lemma}\label{quadlem}
  Suppose $\{a_n\}_{n=1}^{N^{1/2}}$ is a short generalized Dirichlet sequence with $\theta=1,$ $ a_2-a_1=N^{-1},$ $a_1=0.$ Let $e_n=a_n-\frac{n-1}{N}.$ Then for every $c>0,$ there exists a $C^2$ curve $g:[0,1]\rightarrow \R$ with $g''(x) \sim 1$ for $x\in [0,1]$ such that $|g(\frac{n-1}{N^{1/2}})-e_{n}N|\leq cN^{-1}$ for every $n=1,\ldots,N^{1/2}.$
\end{lemma}
\begin{proof}
  We first define $g_0:[0,1]\rightarrow \R$ to be a $C^1$ piece-wise quadratic polynomial with $g_0'(0)=0$ such that  $g_0$ restricted to $[\frac{n}{N^{1/2}},\frac{n+1}{N^{1/2}}]$ is a quadratic polynomial for every $n=0,\ldots,N^{1/2}-1,$ and 
  \[g_0(\frac{n-1}{N^{1/2}})=e_nN.\]
  Since 
  \[\frac{N(e_{n+1}-2e_n+e_{n-1})}{N^{-1}} \sim 1 \]
  we have $g_0'' \sim 1$ on $[0,1]\setminus N^{-1/2}\Z,$ and consequently $\|g_0\|_{L^\infty([0,1])} \lesssim 1$ because $g_0'(0)=0.$ 
  Now we let $g=g_0* \phi$ be the $c'N^{-1}$ mollification of $g_0.$ Here $\phi$ is an $L^1$-normalized smooth bump adapted to $B_{c'N^{-1}}(0)$ and $c'>0$ is sufficiently small depending on $c.$ Then we have for every $x\in [0,1],$
  \[g''(x)=\int_\R g_0''(y)\phi(x-y)dy \sim 1 ,\]
  and 
  \begin{align*}
      |g(\frac{n-1}{N^{1/2}})-e_nN| &  \leq  \int_{\R} |g_0(y)-g_0(\frac{n-1}{N^{1/2}})| \phi(\frac{n-1}{N^{1/2}}-y) dy \\
      & \leq c'N^{-1} \sup_{ y\in [0,1]}|g_0'|  \\
      & \leq cN^{-1}
  \end{align*}
  if $c'= \frac{c}{\|g_0'\|_{L^\infty([0,1])}+1}.$  
\end{proof}

We can use the same approach to transfer an $L^p$ estimate for a longer generalized Dirichlet  polynomial to an $L^p$ estimate on an exponential sum with frequency support near a $C^2$ convex curve. 

Suppose $\{a_n\}_{n=1}^N$ is a generalized Dirichlet sequence with $\theta=1,$ $a_2-a_1=\frac{1}{N},$ $a_1=0,$ and let $\alpha \in (\frac{1}{2},1].$ As before we write $e_n=\frac{n-1}{N} \sim \frac{(n-1)^2}{N^2}.$ The same calculation as above shows that
\[\int_{[0,T]} |\sum_{n=1}^{N^{\alpha}}b_ne^{ita_n}|^p dt  \lesssim \sum_{t_1\in 2\pi N\Z \cap [0,T-2\pi N]} \int_{0}^{2\pi N}  |\sum_{n=1}^{N^{\alpha}} b_n e^{i(t_1 e_n+t_2\frac{n-1}{N}+t_2e_n)}|^p dt_2. \]

One difficulty that appears is that we cannot treat $e^{it_2e_n}$ as an error term as before. This is because when we apply the partial summation formula we get
\[ |\sum_{n=1}^{N^{\alpha}} b_n e^{i(t_1 e_n+t_2\frac{n-1}{N}+t_2e_n)}| 
      \lesssim |\sum_{n=1}^{N^{\alpha}} b_n e^{i(t_1 e_n+t_2\frac{n-1}{N}})|+ \frac{1}{N^{1-\alpha}}\int_{0}^{N^{\alpha}} |\sum_{n=1}^u  b_n e^{i(t_1 e_n+t_2\frac{n-1}{N})}|du.\]
However now $N^{1-\alpha}>N^{\alpha}$ and we cannot estimate the second term on the right hand side as before using the estimate for the first term and Minkowski's inequality. We could still find a $C^2$ convex curve such that $(\frac{n-1}{N}+e_n,e_n)$ lies in an $N^{-1}$-neighborhood of it, but the extra $e_n$ disallows us to use the $2\pi N$-periodicity in the $t_2$ variable.

Another difficulty we find is the integrand is locally constant on intervals of length $N^{2-2\alpha}$ in the $t_1$ variable, and since $N<N^{2-2\alpha},$ that prevents us from transferring the discrete summation into $\sum_{t_1\in 2\pi N\Z \cap [0,T-2\pi N]}$ into $\int_{[0,T]}.$ We may though transfer the discrete sum into an integral over a fat AP  $\int_{P_{2\pi N}^{N^{2-2\alpha}}\cap B_{[0,T]}},$ and that might suggest some new decoupling problems in $\R^2$ that might be helpful for estimating longer generalized Dirichlet polynomials.

Finally we remark that for the Dirichlet sequence $\{\log n\}_{n=N+1}^{2N},$ we may implement this transference method to higher order approximations of $\log n.$ For examples we can write 
\[|\sum_{n=N+1}^{N+N^\alpha} b_ne^{it\log n}|=|\sum_{n=1}^{N^\alpha} b_{n+N}e^{it\log (1+\frac{n}{N})}|=|\sum_{n=1}^{N^\alpha} b_{n+N}e^{it(\frac{n}{N}-\frac{n^2}{2N^2}+e'_n)}|\]
where $e'_n:=\log (1+\frac{n}{N})-\frac{n}{N}+\frac{n^2}{2N^2} \sim \frac{n^3}{N^3}.$
If we write $t=t_1+t_2+t_3$ with $t_1\in 2\pi N^2\Z,$ $t_2\in 2\pi N\Z,$ $t_3\in [0,2\pi N),$ then we could transfer $L^p$ estimates on $\sum_{n=N+1}^{N+N^\alpha} b_ne^{it\log n}$ to $3$-dimensional $L^p$ estimates on exponential sums with frequency supported on a non-degenerate curve in $\R^3.$ More generally one can exploit more terms in the Taylor expansion and get higher dimensional estimates. We do not know how much this would help with estimates on Dirichlet polynomials using decoupling techniques.


\bibliographystyle{alpha}
\bibliography{dirichletrefs}

\end{document}